\documentclass[12pt,a4paper]{article}

\usepackage{amsthm, amsfonts,amssymb,euscript}
\usepackage{amsmath}
\usepackage{bm}
\usepackage{amssymb}
\usepackage{amsthm}
\usepackage{yfonts}
\usepackage{graphicx}
\usepackage{caption}
\usepackage{subcaption}
\usepackage{amsfonts}
\usepackage{float}
\usepackage{color}
\usepackage{slashed}

\usepackage[usenames,dvipsnames,svgnames,table]{xcolor}
\usepackage[colorlinks=true]{hyperref}
\hypersetup{linkcolor=BrickRed, urlcolor=green, citecolor=blue, linktoc=page}

\numberwithin{equation}{section}

\newtheorem*{proposition*}{Proposition}
\newtheorem*{theorem*}{Theorem}
\newtheorem*{conjecture*}{Conjecture}
\newtheorem*{claim*}{Claim}
\newtheorem*{lemma*}{Lemma}
\newtheorem*{corollary*}{Corollary}
\newtheorem{theorem}{Theorem}[section]
\newtheorem{proposition}[theorem]{Proposition}

\newtheorem{lemma}[theorem]{Lemma}
\newtheorem{corollary}[theorem]{Corollary}

\newtheorem*{definition*}{Definition}
\newtheorem{definition}{Definition}[section]
\newtheorem*{assumption*}{\mathcal{A}ssumption}

\newtheorem*{remark*}{Remark}
\newtheorem{remark}{Remark}[section]

\newtheorem{thmx}{Theorem}

\newcommand{\R}{\mathbb{R}}
\newcommand{\s}{\mathbb{S}}

\newcommand{\N}{\mathbb{N}}

\newcommand{\snabla}{\slashed{\nabla}}

\allowdisplaybreaks

\begin{document}

\title{Late-time asymptotics for the wave equation on spherically symmetric, stationary spacetimes}
\author{Y. Angelopoulos, S. Aretakis, and D. Gajic}
\date{February 15, 2018}
\normalsize

\maketitle

\begin{abstract}
We derive precise late-time asymptotics for solutions to the wave equation on spherically symmetric, stationary and asymptotically flat spacetimes including as special cases the Schwarzschild and Reissner--Nordstr\"{o}m families of black holes. We also obtain late-time asymptotics for the time derivatives of all orders and for the radiation field along null infinity. We show that the leading-order term in the asymptotic expansion is related to the existence of the conserved Newman--Penrose quantities on null infinity. 
As a corollary we obtain a characterization of all solutions which satisfy Price's polynomial law  $\tau^{-3}$ as a \textit{lower} bound. 
Our analysis relies on physical space techniques and uses the vector field approach for almost-sharp decay estimates introduced in our companion paper. In the black hole case, our estimates hold in the domain of outer communications up to and including the event horizon. 
Our work is motivated by the stability problem for black hole exteriors and  strong cosmic censorship for black hole interiors. 

\end{abstract}

\tableofcontents

\section{Introduction}

\subsection{Introduction and background}
\label{introo1}

This paper derives precise late-time asymptotics for solutions to the wave equation
\begin{equation}
\label{waveequation}
\square_g\psi=0.
\end{equation} 
on four-dimensional spherically symmetric, stationary and asymptotically flat globally hyperbolic Lorentzian manifolds $(\mathcal{M},g)$. Obtaining such late-time asymptotics  has important applications in mathematical physics and is relevant in the study of the following problems that arise in general relativity: 1) the  long-time study of the Einstein equations,  2) the black hole stability problem,  3) strong cosmic censorship and 4) the propagation of gravitational waves. 

According to the strong Huygens principle, solutions $\psi$ to the flat wave equation 
 on the four-dimensional Minkowski spacetime with compactly supported initial data have trivial asymptotics since for all $x\in\mathbb{R}^{3}$ there is a time $t_{0}(x)$ such that $\psi(t,x)=0$ for all $t\geq t_{0}(x)$.  The situation is drastically different however if we consider the wave equation on curved backgrounds. Indeed, heuristic arguments first put forth by Price \cite{RP72} in 1972 suggest that solutions $\psi$ to  the wave equation on  Schwarzschild backgrounds evolving from compactly supported initial data satisfy
\begin{equation}
\psi(t,{r}_{0},\theta,\varphi)\sim t^{-3}
\label{price1}
\end{equation}
 asymptotically as $t\rightarrow \infty$ along constant $r={r}_{0}$ hypersurfaces. Here $(t,r,\theta,\varphi)$ denote the Boyer--Lindquist coordinates and $r_{0}>2M$, where $M>0$ is the mass parameter of the Schwarzschild spacetime. In fact, Price predicted that 
\begin{equation}
 \psi_{\ell}(t,{r}_{0},\theta,\varphi)\sim t^{-3-2\ell},
\label{pricel}
\end{equation} asymptotically as $t\rightarrow \infty$ along constant $r=r_{0}$ hypersurfaces, 
 where $\psi_{\ell}$ is supported on the $\ell^{\text{th}}$ spherical harmonic frequency. Further work was subsequently obtained by Gundlach, Price and Pullin \cite{CGRPJP94b} (see also \cite{gomez1994})
 for the scalar field along the Schwarzschild event horizon $\mathcal{H}=\left\{(v,r=2M,\omega): v\in\mathbb{R},\, \omega\in\mathbb{S}^{2} \right\}$ (here $v$ is an appropriate ``time'' parameter on $\mathcal{H}$): 
 \begin{equation}
 \psi\left.\!\!\right|_{\mathcal{H}} (v,r=2M,\omega)\    \sim v^{-3} 
 \label{pricehorizon}
 \end{equation}
 and for the radiation field $r\psi$ along the null infinity $\mathcal{I}=\left\{(u,r=\infty, \omega): u\in\mathbb{R},\, \omega\in\mathbb{S}^{2} \right\}$ (here $u$ is an appropriate ``time'' parameter on $\mathcal{I}$):
 \begin{equation}
 r\psi\left.\!\!\right|_{\mathcal{I}}(u,r=\infty, \omega)\ \sim u^{-2}. 
 \label{pricerf}
 \end{equation}
 The structure of the tail of solutions to the wave equation has been intensively studied from a heuristic point of view. Late-time asymptotics for a more general class of curved spacetimes were  predicted in \cite{clsy95} via the analytic continuation of the Green's function.  For works on tails for extremal black holes see \cite{other2, harvey2013,ori2013,mhighinsta,zimmerman1} and also the discussion in  Section \ref{introappli}. For late-time tails of non-linear wave equations see \cite{bizonnonlinear}. 
 
Although very interesting, statements of the form \eqref{price1} are not completely satisfactory from a mathematical point of view for the following two reasons:
\begin{enumerate}

\item \textbf{(Precise leading order term)} The above heuristics do not seem to provide the exact leading-order term in terms of an explicit natural quantity $Q(r,\theta, \varphi,\text{i.d.}[\psi])$ of the initial data such that along constant $r=r_{0}$ hypersurfaces
\begin{equation}
\psi(t,r_{0}, \theta, \varphi)=Q(r_{0},\theta,\varphi,\text{i.d.}[\psi])\cdot \frac{1}{t^3}+O(t^{-3-\epsilon}), \ \epsilon>0, \ \ \text{ as }t\rightarrow \infty.
\label{price2}
\end{equation}
\item \textbf{(Global quantitative estimate)} In view of its  asymptotic character, \eqref{price1} (or even \eqref{price2}) does not provide quantitative bounds for the size of the solutions at \textbf{all} times  in terms of the initial data.  
\end{enumerate}
In fact, the above heuristics only suggest an \textit{upper} bound on the best possible rate in quantitative decay estimates. That is, they suggest that given $R>0$ the following estimate should hold for $r\leq R$,
\begin{equation}
|\psi(t,r,\theta,\varphi)|\leq C_{R}\cdot \sqrt{E_{t=0}[\psi]}\cdot \frac{1}{t^{3}} 
\label{pricepsi}
\end{equation}
where $C_{R}$ is a uniform constant that depends on $R$, and $E_{t=0}[\psi]$ an appropriate weighted higher-order energy norm of the initial data. There has been great progress in establishing estimates of the form \eqref{pricepsi}. Specifically, Dafermos and Rodnianski obtained such bounds for the fully non-linear Einstein--Maxwell-scalar field system in the context of spherical symmetry. Kronthaler \cite{kro} obtained $t^{-3}$ decay for the wave equation on Schwarzschild for a class of smooth spherically symmetric compactly supported initial data with support away from the event horizon. The symmetry assumption on $\psi$ was subsequently removed by Donninger, Schlag and Soffer in \cite{other1} where sharp decay rates were derived for general solutions on Schwarzschild. The same authors obtained in \cite{dssprice}  improved $\ell$-dependent decay rates for fixed spherical harmonic modes consistent with Price's heuristics \eqref{pricel}. Specifically they proved the $2\ell+2$ decay rate for general initial data and the $2\ell+3$ decay rate for static initial data. Sharp upper bounds  were obtained by Metcalfe, Tataru and Tohaneanu \cite{metal} for a general class of asymptotically flat spacetimes (without exact symmetries) based on properties of the fundamental solution for the constant coefficient d'Alembertian  (see also \cite{tataru3}).  A new vector field approach to almost-sharp decay rates for spherically symmetric backgrounds was presented in our companion paper \cite{paper1}. For extensive decay results on a very general class of spacetimes (with non-constant Bondi mass) we refer the reader to the works of Moschidis \cite{moschidis1,molog}.

On the other hand, obtaining a rigorous proof of the asymptotic estimate \eqref{price1} had remained an open problem. The purpose of the present paper is to provide a rigorous proof  of the polynomial tails in the late-time asymptotic expansion of solutions to the wave equation on a general class of spherically symmetric, stationary and asymptotically flat spacetimes (including Schwarzschild and sub-extremal Reissner--Nordstr\"{o}m).  In fact our work  addresses both points discussed above: 
\begin{enumerate}
\item we provide an explicit expression for $Q(r,\theta,\varphi,\text{i.d.}[\psi])$ in terms of the initial data of $\psi$ and give a natural characterization in terms of the Newman--Penrose constant on null infinity,

\item 

we obtain \textit{sharp, global in space and time  quantitative upper and lower bounds.} 
Specifically, given $R>0$ we obtain bounds for solutions to the wave equation of the form 
\begin{equation}
\left| \psi(t,r,\theta,\varphi) -Q(r,\theta,\varphi,\text{i.d.}[\psi])  \cdot \frac{1}{t^{3}}\right| \leq C_{R}\sqrt{E_{t=0}[\psi]}\cdot \frac{1}{t^{3+\epsilon}}, 
\label{ourintro}
\end{equation}
for $r\leq R$, with $C_{R}$ a uniform constant that depends on $R$,  $\epsilon>0$ and $E_{t=0}[\psi]$ some appropriate weighted higher-order energy norm of the initial data. In fact we obtain quantitative bounds global in space (for all $r$) which at the limit $r\rightarrow \infty$ yield the asymptotic behavior of the radiation field at null infinity proving therefore \eqref{pricerf}  (see \eqref{np0rfinfo1} below).
\end{enumerate}

 As a corollary, we provide a  characterization of
 \begin{enumerate}
 \item all initial data which give rise to solutions to the wave equation which satisfy Price's polynomial law \eqref{price1} as a \textit{lower} bound,
 \item  all spherically symmetric initial data which lead to solutions which decay in time faster than any polynomial rate. 
\end{enumerate}
Our proof is based on purely physical space techniques and makes use of the vector field approach for almost-sharp decay estimates introduced in our companion paper \cite{paper1}. We make no use of conformal compactifications and, in the black hole case, we do not need to assume that the initial data are supported away from the event horizon. Our results and methods are consistent with possible applications to non-linear problems, the main prototype of which is the stability problem for exterior regions (see for instance \cite{lecturesMD}). For other potential applications see Section \ref{introappli}.

An informal statement of our main results can be found in Section \ref{introsumm} whereas all the main theorems are listed in Section \ref{themaintheorems}. We conclude this introductory note with the following comments.

\begin{itemize}

\item

 We show that the quantity $Q(r,\theta,\varphi,\text{i.d.}[\psi])$ is independent of $r,\theta,\varphi$ and hence is a constant that depends only on the initial data of $\psi$. In fact, we show that it is equal to a multiple of the Newman--Penrose constant of the time-integral of $\psi$, that is of the regular solution $\widetilde{\psi}$ of the wave equation which satisfies $T\widetilde{\psi}=\psi$.  We call the latter constant, which is  introduced for the first time in this paper,  \textit{``the time-inverted Newman--Penrose constant of $\psi$''} and denote it by $I_{0}^{(1)}[\psi]$. 

For the definition of the Newman--Penrose constant along null infinity see \cite{NP1, np2} and Section \ref{npsection}. For details regarding the time-inverted Newman--Penrose constant see Section \ref{sec:consttinvphi}. 

The existence of the Newman--Penrose constant is a special case of conservation laws along null hypersurfaces, the general theory of which was presented in \cite{aretakisglue}.

\item An immediate corollary of \eqref{ourintro} and the above comment is the following: \textit{All (compactly-supported) initial data which give rise to solutions to the wave equation which satisfy Price's polynomial law $\tau^{-3}$ as a lower bound satisfy}
\begin{equation}
I_{0}^{(1)}[\psi]\neq 0.
\label{charaintro}
\end{equation}
It follows from the definition of $I^{(1)}_{0}[\psi]$, for which we obtain an explicit expression, that \eqref{charaintro} holds for generic (compactly-supported) initial data.

\item The estimate \eqref{ourintro} provides the first lower pointwise bound for solutions to the wave equation on Schwarzschild (and sub-extremal Reissner--Nordstr\"{o}m) backgrounds. The only previously shown result in this direction is the remarkable work of Luk--Oh \cite{Luk2015} where it is shown that a certain weighted energy flux through the horizon generically blows up in time\footnote{This implies in particular that the \emph{time derivatives} of solutions arising from generic initial data cannot decay with a polynomial decay rate faster than $v^{-4}$ along the event horizon. Note also that on Schwarzschild--anti de Sitter (and more generally Kerr--anti de Sitter) backgrounds, Holzegel--Smulevici have shown in \cite{gusmu2} that the decay rate in any \emph{uniform} energy decay estimate cannot be \emph{faster} than logarithmic. See also the analogous results of Keir for ultracompact neutron stars \cite{keir} and supersymmetric microstate geometries \cite{keir2} (where the sharp uniform decay rate is in fact sub-logarithmic).}.  This blow-up result is used in \cite{Luk2015} for obtaining a rigorous proof of the linear instability of the inner Cauchy horizon. 
The lower pointwise bounds for the scalar field and its derivatives along the event horizon obtained in the present paper recover in particular the blow-up result on the event horizon of \cite{Luk2015}. See also Section \ref{lukoh} for a more detailed discussion regarding the relation of \cite{Luk2015} and the present paper.

\item
We provide asymptotics in the case of non-compactly supported initial data as well. We show that the leading order term in this case is proportional to $I_{0}[\psi]\cdot \frac{1}{t^{2}}$, where $I_{0}[\psi]$ denotes the Newman--Penrose constant of $\psi$ (in fact we obtain a quantitative bound similar to \eqref{ourintro}). If the initial data are non-compactly supported but sufficiently decaying towards infinity such that $I_{0}[\psi]=0$ then the solution $\psi$ asymptotes to $ -8I_{0}^{(1)}[\psi] \cdot \frac{1}{t^{3}}$ in which case we obtain the estimate \eqref{ourintro}. If additionally  $ I_{0}^{(1)}[\psi] =0$ then we obtain that the leading-order term of the spherical mean of $\psi$ is proportional to $I_{0}^{(2)}[\psi]\cdot \frac{1}{t^{4}}$ where $I_{0}^{(2)}[\psi]$ denotes what we call the time-inverted Newman--Penrose of second order (see Section \ref{timeinvertednpsection}). The $t^{-4}$ decay rate was first observed numerically in \cite{karzdma}; see also \cite{priceburko,dssprice} and Remark \ref{rmk:initstatic}. We show that this increase in the decay rate in fact continues to \emph{all} orders if we assume the additional vanishing of time-inverted Newman--Penrose constants of higher order. This allows us to obtain a \textit{characterization of all spherically symmetric initial data which lead to solutions which decay in time faster than any polynomial rate in terms of the vanishing of the time-inverted Newman--Penrose constants of all orders}; see Remark \ref{rmk:vanishingkthnpconsts} for more details.

\item

We derive asymptotics for the time derivatives of all orders as well as for the radiation field along the null infinity $\mathcal{I}$ and all its time derivatives\footnote{Regularity results for the radiation field were obtained in \cite{baskinwang} using a partial compactification of Schwarzschild spacetime  (see also \cite{baskinw}). }. The latter confirms the heuristics \eqref{pricerf}. In fact for sufficiently decaying initial data we obtain
\begin{equation}
\left|r\psi\left.\right|_{\mathcal{I}}(u,\cdot) +2I^{(1)}_{0}[\psi]\cdot \frac{1}{u^{2}}\right|\leq C\cdot \sqrt{E_{t=0}[\psi]}\cdot \frac{1}{u^{2+\epsilon}}, \ \ \epsilon>0.
\label{np0rfinfo1}
\end{equation}

\item

 In the black hole case, our estimates hold in the domain of outer communications up to and including the event horizon confirming the heuristics \eqref{pricehorizon}. In fact for smooth compactly supported initial data we obtain
\begin{equation}
\left| \psi\left.\right|_{\mathcal{H}}(v,\cdot) +8I^{(1)}_{0}[\psi]\cdot \frac{1}{v^{3}}\right|\leq C\cdot \sqrt{E_{t=0}[\psi]}\cdot \frac{1}{v^{3+\epsilon}}, \ \ \epsilon>0.
\label{np0rfinfo2}
\end{equation} 
 This is of importance for potential applications to the stability problem for the exterior black hole region and to the strong cosmic censorship and the structure of the interior of black holes. See also Section \ref{introappli}.

 \item (Relation with scattering constructions) The strongly correlated asymptotics \eqref{np0rfinfo1} and \eqref{np0rfinfo2}  along the event horizon $\mathcal{H}$ and the null infinity $\mathcal{I}$ are in complete agreement with the results of Dafermos--Rodnianski--Shlapentokh-Rothman in \cite{linearscattering} where it is shown that generic polynomially decaying scattering data on $\mathcal{H}$ and $\mathcal{I}$ lead to singular backwards-in-time solutions. It is a very interesting problem to investigate the relevance of the forwards-in-time theory of the present paper and  definitive correlation conditions between polynomially decaying scattering data on $\mathcal{H}$ and $\mathcal{I}$ which give rise to regular backwards solutions.

\end{itemize}

\subsection{Summary of main results}
\label{introsumm}

In this section we present rough versions of the main results. For the rigorous statements of the main theorems see Section \ref{themaintheorems}.

We consider spherically symmetric, stationary and asymptotically flat spacetimes $(\mathcal{M}, g)$ and study the global behavior of solutions $\psi$ to the linear wave equation \eqref{waveequation} on such backgrounds. The spacetimes under consideration include the Schwarzschild and the Reissner--Nordstr\"{o}m family of black holes. For the precise assumptions on the spacetime metric $g$ see Section \ref{thespacemanif}. 

Let $\tau$ be a ``time'' function, the level sets $\Sigma_{\tau}$ of which are hyperboloidal hypersurfaces terminating at null infinity (see Section \ref{sec:foliations}). We will obtain the leading order term of the late-time asymptotics of $\psi\left.\! \right|_{\Sigma_{\tau}}\!(\tau,\cdot)$ as a function of $\tau$ and the initial data for $\psi$. 

We make the following ``black-box assumptions'' for the wave equation (see also Section \ref{sec:waveassm}):
\begin{enumerate}
\item\textbf{Boundedness of the non-degenerate energy:} We assume the bound
\[\int_{\Sigma_{\tau}}|\partial\psi |^{2} \leq C \int_{\Sigma_{0}}|\partial\psi |^{2}  \]
for all $\tau\geq 0$, for appropriate derivatives $\partial\psi$ of $\psi$ (see Section \ref{sec:waveassm} for the details). 

\item \textbf{Integrated local energy decay estimate:} For all $R>0$, we assume the bound
\[ \int_{0}^{t}\int_{\Sigma_{\tau}\cap \left\{r\leq R \right\}}|\partial\psi |^{2} \leq C_{R}\sum_{k=0}^{n}\int_{\Sigma_{0}} |\partial T^{k}\psi |^{2}  \]
for some $n\in\mathbb{N}$ which is related to the trapping effect. Note that this estimate, also known as the Morawetz estimate, implies that the event horizon of black hole regions, if present, must necessarily be non-degenerate. 
\end{enumerate}

We first introduce the following expression for compactly supported initial data
\begin{equation}
I_{0}^{(1)}[\psi]=
-\frac{M}{4\pi}\int_{\Sigma_{0}}\Big( 2(1-h_{\Sigma_{0}}D)r\partial_{\rho}\phi-(2-Dh_{\Sigma_{0}})rh_{\Sigma_{0}} T\phi-(r\cdot (Dh_{\Sigma_{0}})')\cdot\phi\Big)\sin\theta dr d\theta d\varphi,
\label{inftinp1}
\end{equation}
where $\phi=r\psi$, $\partial_{\rho}$ is a radial vector field tangential to $\Sigma_{0}$, $T$ is the stationary Killing vector field, and $h_{\Sigma_{0}}$ is a function that depends on the embedding of $\Sigma_{0}$ in $\mathcal{M}$  as in Section \ref{thespacemanif}. For the definition of $I_{0}^{(1)}[\psi]$ for non-compactly supported initial data see Section \ref{timeinvertednpsection}.

We show the following result.

\begin{thmx} \textbf{(Asymptotics for solutions with compactly supported data)} Let $\psi$ be a solution to the wave equation \eqref{waveequation} on the Lorentzian manifolds $(\mathcal{M},g)$, defined in Section \ref{thespacemanif}, with smooth compactly supported initial data. 
Then  $\psi$ satisfies the following asymptotic estimate in the region where $\left\{ r\leq R\right\}$
\begin{equation}
 \psi\left.\! \right|_{\Sigma_{\tau}}\!(\tau,\cdot)  \sim_{asym} -8I_{0}^{(1)}[\psi]\cdot \frac{1}{\tau^{3}},
\label{infasy1}
\end{equation}
as $\tau\rightarrow \infty$. where the constant $I_{0}^{(1)}[\psi]$ is given  by \eqref{inftinp1}.  
Furthermore, in the ``near-infinity'' region  where  $\left\{ r\geq R\right\}$ we have
\begin{equation}
 \psi\left.\! \right|_{\Sigma_{\tau}}\!(\tau,\cdot) \sim_{asym}-4I_{0}^{(1)}[\psi]\cdot \left( 1+\frac{u}{v} \right)\cdot \frac{1}{u^{2}\cdot v},
\label{nearinfiinfo1}
\end{equation}
where $(u,v)$ are double null coordinates. 
\end{thmx}

The asymptotic estimate \eqref{infasy1} is to be considered in the sense that 
\[\left| \psi\left.\! \right|_{\Sigma_{\tau}}\!(\tau,\cdot) +8I_{0}^{(1)}[\psi]\cdot \frac{1}{\tau^{3}}\right| \leq C_{R}\sqrt{E[\psi]}\cdot \frac{1}{\tau^{3+\epsilon}}, \]
for $r\leq R$ where $\epsilon>0$ and  $E[\psi]$ is some appropriate weighted higher-order energy norm. Similar comment applies for the asymptotic estimate \eqref{nearinfiinfo1} (see Theorem \ref{thm:asmpsinpn011} of Section \ref{themaintheorems}).

The asymptotics \eqref{infasy1} and \eqref{nearinfiinfo1} hold for general initial data which are sufficiently regular and decaying towards infinity such that the Newman--Penrose constant vanishes. See also Theorem \ref{thm:asmpsinpn011} of Section \ref{themaintheorems}.

An immediate corollary of the above theorem is a complete characterization of all solutions to the wave equation which satisfy Price's polynomial law $\tau^{-3}$ as a lower bound. There arise from initial data on a Cauchy hypersurface such that 
\begin{equation}
I_{0}^{(1)}[\psi]\neq 0.
\label{charaintro}
\end{equation}
Note that \eqref{charaintro} holds for generic initial data with vanishing Newman--Penrose constant. In this case,  \eqref{infasy1} provides indeed the leading term in the late-time asymptotics of $\psi$. 
 If, additionally, $\psi$ satisfies $I_{0}^{(1)}[\psi]= 0$ then the estimate \eqref{infasy1} shows that $\psi$ decays faster than $\tau^{-3}$. In fact, in this case we still obtain the leading order term for the spherical mean $\int_{\mathbb{S}^{2}}\psi$. 
 \begin{thmx} \textbf{(Higher-order asymptotics for the spherical mean)} For compactly supported initial data such that $I^{(1)}_{0}[\psi]=0$ we obtain that 
\begin{equation*}
\begin{split}
 \int_{\mathbb{S}^{2}}\psi\left.\! \right|_{\Sigma_{\tau}}\!(\tau,\cdot)  \sim_{asym} 24\cdot {I_0^{(2)} [\psi ]}  \cdot \frac{1}{\tau^{4}}. 
\end{split}
\end{equation*}
where $I^{(2)}_{0}[\psi]$ that can be computed explicitly by the initial data and is as defined in Section \ref{timeinvertednpsection}. Inductively, if
\[ I_{0}^{(1)}[\psi]=I_{0}^{(2)}[\psi]=\cdots=I_{0}^{(k-1)}[\psi]=0 \]
for $k\in\mathbb{N}^{*}$, for constants $I_{0}^{(1)}[\psi], I_{0}^{(2)}[\psi],\cdots, I_{0}^{(k-1)}[\psi]$ which can be explicitly computed by the initial data as defined in Section \ref{timeinvertednpsection} then 
\begin{equation*}
\begin{split}
 \int_{\mathbb{S}^{2}}\psi\left.\! \right|_{\Sigma_{\tau}}\!(\tau,\cdot)  \sim_{asym} 4(-1)^k(k+1)!\cdot \frac{I_0^{(k)} [\psi ]}{(\tau+1)^{k+2}}
\end{split}
\end{equation*}
in $\left\{ r\leq R\right\}$, for a constant $I^{(k)}_{0}[\psi]$ that can be computed explicitly by the initial data. Furthermore, in this case we obtain in $\left\{r\geq R\right\}$:
\begin{equation*}
\begin{split}
\int_{\mathbb{S}^{2}}\psi \left.\! \right|_{\Sigma_{\tau}}\!(\tau,\cdot)  \sim_{asym} 4(-1)^k k!\cdot {I_0^{(k)} [\psi ]}\cdot \left(1+\sum_{j=1}^k\left(\frac{u}{v}\right)^j\right)\cdot \frac{1}{(u+1)^{k+1}v}.
 \end{split}
\end{equation*}
The above integrals are considered with the standard volume form on $\mathcal{S}^2$: $\sin\theta d\theta d\varphi$.
\end{thmx}
\begin{remark}
\label{rmk:vanishingkthnpconsts}
The above theorem provides a complete characterization of all spherically symmetric solutions to the wave equation which decay like $\tau^{-k}$ with $k\in\mathbb{N}$ and $k\geq 3$: such solutions arise from initial data such that 
\[I_{0}[\psi]= I_{0}^{(1)}[\psi]=I_{0}^{(2)}[\psi]=\cdots =I_{0}^{(k-3)}[\psi]=0, \ \ \ I_{0}^{(k-2)}[\psi]\neq 0. \]
Hence all spherically symmetric solutions to the wave equation which decay faster than any polynomial rate (eg. exponentially decaying solutions) satisfy \[I_{0}^{(j)}[\psi]=0\] for all $j\in\mathbb{N}$.

\end{remark}

\begin{remark}\textbf{(The leading-order coefficient $I_{0}^{(1)}[\psi]$)}
 The constant $I_{0}^{(1)}[\psi]$ coincides with the Newman--Penrose constant of the time integral $\widetilde{\psi}$, i.e.~the unique regular solution $\widetilde{\psi}$ to the wave equation such that 
\[T\widetilde{\psi}=\psi.\]
We shall call this constant the time-inverted Newman--Penrose constant of $\psi$. See also Section \ref{sec:consttinvphi}. 
 This observation implies that the existence of this conserved quantity is fundamental in the late-time asymptotics even in the case of compactly supported initial data. For a classification of null hypersurfaces admitting such conserved quantities see \cite{aretakisglue}. 
\end{remark}

\begin{remark}\textbf{(Asymptotics for time-symmetric initial data)}
\label{rmk:initstatic}
 If we consider initial data on the $\left\{t=0 \right\}$ hypersurface on Schwarzschild spacetime, where $t$ is the Boyer--Lindquist coordinate with \underline{support away from the bifurcation sphere} and  $T\psi\left.\right|_{t=0}=0$ then a simple calculation shows that $I^{(1)}_{0}[\psi]=0$. Hence we obtain faster decay for such solutions (in fact in view of the above theorem we obtain the precise late-time asymptotics) which is in agreement with one of the results of \cite{dssprice}. 
\end{remark}

\begin{remark} \textbf{(Asymptotics for the radiation field)}
 The asymptotic estimate \eqref{nearinfiinfo1} yields the following asymptotic for the radiation field $r\psi$ along null infinity
\begin{equation}
r\psi\left.\right|_{\mathcal{I}^{+}}(u,\cdot) \sim_{asym} -2I^{(1)}_{0}[\psi]\cdot \frac{1}{u^{2}}.
\label{np0rfinfo}
\end{equation}
In fact, for obtaining the more general \eqref{nearinfiinfo1} we first need to obtain the above estimate. See also Section \ref{prevtech}.

\end{remark}

\begin{thmx}\textbf{(Asymptotics for $T^{k}\psi$)} The derivatives $T^{k}\psi$, $k\geq 1$, satisfy the following asymptotic estimate in the region where $\left\{ r\leq R\right\}$
\begin{equation}
T^{k} \psi\left.\! \right|_{\Sigma_{\tau}}\!(\tau,\cdot)  \sim_{asym} T^{k}\left( -8I_{0}^{(1)}[\psi]\cdot \frac{1}{\tau^{3}}\right),
\label{infasyk}
\end{equation}
as $\tau\rightarrow \infty$. Clearly the right hand side decays like $\tau^{-3-k}$. We also obtain the following estimate in the  region where  $\left\{ r\geq R\right\}$
\begin{equation}
 T^{k}\psi\left.\! \right|_{\Sigma_{\tau}}\!(\tau,\cdot) \sim_{asym}T^{k}\left( -4I_{0}^{(1)}[\psi]\cdot \left( 1+\frac{u}{v} \right)\cdot \frac{1}{u^{2}\cdot v}\right).
\label{nearinfiinfok}
\end{equation}
Clearly the right hand side decays like $u^{-2-k}\cdot v^{-1}$.
\end{thmx}

So far we have considered initial data which are sufficiently decaying such that the Newman--Penrose constant vanishes. The next theorem concerns asymptotics in the case of finite non-vanishing  Newman--Penrose constant. If fact, according to the analysis of the present paper, it is of fundamental importance to first prove the following theorem and then attempt to prove the previous theorems. 

\begin{thmx} \textbf{(Asymptotics for $\psi$ with non-vanishing Newman--Penrose constant)}
If $\psi$ is a solution to the wave equation with non-vanishing Newman--Penrose constant $I_{0}[\psi]\neq 0$ then we have
\begin{equation}
\psi\left.  \right|_{\Sigma_{\tau}}\!(\tau,\cdot)  \sim_{asym} 4 I_{0}[\psi]\cdot \frac{1}{\tau^{2}} 
\label{nzeronpinfo}
\end{equation}
in $\left\{r\leq R \right\}$, and more generally if $k\geq 1$ then
\begin{equation}
 T^{k}\psi \left. \right|_{\Sigma_{\tau}}\!(\tau,\cdot)  \sim_{asym} T^{k}\left(4 I_{0}[\psi]\cdot \frac{1}{\tau^{2}}  \right)
\label{tkn0np}
\end{equation}
in $\left\{r\leq R \right\}$. Note that the right hand side of \eqref{tkn0np} decays like $\tau^{-2-k}$. As far as the ``near-infinity'' region $\left\{ r\geq R\right\}$ is concerned, we have 
\[\psi\left.\! \right|_{\Sigma_{\tau}}\!(\tau,\cdot) \sim_{asym} 4I_{0}[\psi]\left(1+\frac{u}{v} \right)\cdot \frac{1}{u\cdot v}  \] 
and more generally
\[ T^{k}\psi\left.\! \right|_{\Sigma_{\tau}}\!(\tau,\cdot) \sim_{asym} T^{k}\left(4I_{0}[\psi]\left(1+\frac{u}{v} \right)\cdot \frac{1}{u\cdot v}  \right).  \]

 \end{thmx}

\subsection{Applications}
\label{introappli}

We list below upcoming extensions and a few potential applications of our results and techniques.

\begin{itemize}

\item 

(Refinements and extensions) In upcoming works we extend our results to the following directions:
\begin{itemize}
\item We obtain higher-order terms in the late-time asymptotic expansion and show that they contain logarithmic corrections. See also \cite{gomez1994,SH99}.
\item 
We relax the symmetry assumptions on  the spacetime metric so that our class of spacetimes includes in particular the sub-extremal Kerr family of black holes (see \cite{pricekerr, burkokerr} for work on tails on Kerr backgrounds). 

\item 
We obtain late-time asymptotics for each spherical harmonic parameter $\ell$ confirming the heuristics \eqref{pricel}. This is possible via an extension of the method of \cite{paper1} and the derivation of new elliptic estimates for solutions supported on angular frequencies greater of equal to  $\ell$.  

\end{itemize}

\item
(Linearized gravity and gravitational waves)
Determining the leading-order term in the late-time asymptotic expansion along null infinity for solutions to the Teukolsky equation, and more generally, for the linearized Einstein equations would be of  importance for the study of the propagation of gravitational waves. In a future work we will investigate the relevance of our method to the study of these systems.

\item 
(The black hole interior and  strong cosmic censorship)
Our results are relevant for the study of the wave equation in the black hole interior region and in particular for understanding the extendibility of solutions beyond the Cauchy horizon. Bounds of the scalar field in the interior regions are important for investigating the stability properties of the Cauchy horizon and hence addressing  strong cosmic censorship.  Such bounds rely heavily on upper and lower bounds for solutions to the wave equation along the event horizon, as first demonstrated rigorously by Dafermos in \cite{MD03,MD05c}. See also \cite{MD12, Luk2015, LukSbierski2016, DafShl2016, Hintz2015, Franzen2014, Luk2016a,Luk2016b} for results in the interior of sub-extremal black holes. 

In the case of extremal black hole interiors, the third author showed in \cite{gajic} that the sharpness of the decay rates along the event horizon for solutions to the wave equation  is intimately related to the extendibility properties across the Cauchy horizon; see also analogous results in extremal Kerr--Newman \cite{gajic2}. In fact, \cite{gajic} illustrates how the \emph{precise} leading-order behaviour of spherically symmetric solutions along the extremal event horizon  and the conjectured next-to-leading order behaviour predicted by numerics of \cite{harvey2013} are crucial for the statement of $C^2$-extendibility of spherically symmetric solutions across the horizon. Obtaining a rigorous proof of these asymptotics is the topic of a future work. 

We further remark that the precise exponents in the \emph{exponential} decay rates along the event horizon of solutions to (\ref{waveequation}) on sub-extremal Reissner--Nordstr\"om--de Sitter are also relevant for establishing higher-regularity extendibility properties beyond the Cauchy horizon in the cosmological black hole interior \cite{Costa2015, Costa2014, Costa2014a, Hintz2015a}.

In comparison, in the black hole exterior of sub-extremal Kerr--{anti} de Sitter satisfying the Hawking--Reall bound, the sharp uniform energy decay rate of solutions to (\ref{waveequation}) has been shown to be merely \emph{logarithmic}, see \cite{gusmu1,gusmu2}.

\item (Extremal black holes)
 A mathematical study of the wave equation on extremal Reissner--Nordstr\"om and extremal Kerr was initiated by the second author in \cite{aretakis1, aretakis2,aretakis3, aretakis4, aretakis2012, aretakis2013} establishing in particular that transversal derivatives  along the event horizon generically \emph{do not} decay and higher-order transversal derivatives \emph{blow up} asymptotically in proper time (see also  \cite{aag1}).

Subsequent numerical and heuristic work by Reall et al  \cite{hm2012,harvey2013}, Ori \cite{ori2013} and Sela\cite{Sela2015} gave evidence in support of a modified weaker power-law of solutions to (\ref{waveequation}) on extremal Reissner--Nordstr\"om. We also refer the reader to the very interesting recent works \cite{zimmerman1, zimmerman2} for the power-law decay on extremal Kerr backgrounds. 

In an upcoming work \cite{aag7} we derive the late time asymptotics of scalar fields on extremal Reissner--Nordstr\"om.

\end{itemize}

\subsection{Overview of techniques}
\label{prevtech}
In this section we present the main new ideas for obtaining the precise late-time asymptotic behavior of solutions to spherically symmetric, stationary and asymptotically flat spacetimes.

\medskip

\textbf{Step 1: Spherical mean decomposition}

\smallskip

We first decompose any solution $\psi$ to the wave equation as follows:
\[\psi=\psi_{0}+\psi_{ 1}, \]
where 
\[\psi_{0}=\frac{1}{4\pi}\int_{\mathbb{S}^{2}}\psi \sin\theta d\theta d\varphi\]
and \[\psi_{ 1}=\psi-\psi_{0}.\]
Based on the vector field approach that was first introduced in \cite{paper1}, for smooth compactly supported initial data (and more generally for sufficiently regular and sufficiently decaying initial data) we obtain the following decay rate for $\psi_{ 1}$:
\[|\psi_{1}(\tau,r,\theta,\varphi)|\leq C\cdot \sqrt{E^{\epsilon}[\psi_{1}]}\cdot \frac{1}{\tau^{\frac{7}{2}-\epsilon}},   \]
where 
\[E^{\epsilon}[\psi_{1}] =  E_{1;1}^{\epsilon}[\psi_{\ell=1}]+E^{\epsilon}_{2;1}[\psi_{\ell\geq 2}]   \] 
where the energy norms are as defined in the Appendix \ref{apx:energynorms}.

Therefore, since the leading order term in the late-time asymptotics of $\psi$ is conjectured to decay like $\tau^{-3}$, it suffices to obtain the late-time asymptotics for the spherically symmetric part $\psi_{0}$ of $\psi$. 

Since the sharp decay rate for solutions to the wave equation is intimately tied with the asymptotics of the metric towards infinity, in order to obtain the asymptotics we need to rely on properties of asymptotical flat spacetimes. One crucial property of such spacetimes is the existence of the so-called \textit{Newman--Penrose constant} along null infinity.

\medskip

\textbf{Step 2: The Newman--Penrose constant $I_{0}$ along null infinity}

\smallskip

Given a spherically symmetric solution $\psi_{0}$ to the wave equation, the function defined on null infinity:
\[I_{0}[\psi_{0}](u): u\mapsto \lim_{r\rightarrow \infty}r^{2}\partial_{r}(r\psi_{0}) (u,r) \]
is constant, that is independent of $u$ (see \cite{np2,NP1}). Here $\partial_{r}$ is to be considered with respect to the outgoing Eddington--Finkelstein coordinate system $(u,r,\theta,\varphi)$. The Newman--Penrose constant  $I_{0}[\psi_{0}]$ of $\psi_{0}$ is precisely the (constant) value of the above function. The exisence of the Newman--Penrose constant should be thought of as a \textit{conservation law} for the evolution of the scalar field $\psi_{0}$.

Hence, if the initial data corresponding to $\psi_{0}$ are such that the rescaled limit $\lim_{r\rightarrow \infty}r^{2}\partial_{r}(r\psi_{0})(u=0,r)=1$ then  the limit $\lim_{r\rightarrow \infty}r^{2}\partial_{r}(r\psi_{0})(u,r)$ is always equal to $1$ for all $u$. This could in principle lead to lower bounds for $\psi_{0}$. 

In the case, however, of smooth compactly supported initial data, by the domain of dependence theorem we have that the Newman--Penrose constant vanishes: $I_{0}[\psi_{0}]=0$. A crucial step of our method is the use of the time integral $\psi^{(1)}$. This is a function $\psi_{0}^{(1)}$ associated to $\psi_{0}$, which is constructed by inverting the time-translation operator $T$, and whose Newman--Penrose constant is non-zero for generic initial data and hence can in principle be used to derive lower bounds. 

\medskip

\textbf{Step 3: The time integral $\psi^{(1)}$ and the time-inverted N--P constant $I_{0}^{(1)}$}

\smallskip

\noindent Given a spherically symmetric solution $\psi$ to the wave equation with vanishing Newman--Penrose constant
\[I_{0}[\psi]=0,\]
and in fact such that 
\[r^{3}\partial_{r}(r\psi)<\infty, \]
there exists a \textit{unique} smooth solution $\psi^{(1)}$ to the wave equation such that 
\[T\psi^{(1)}=\psi, \]
and 
\[\lim_{r\rightarrow \infty}r\psi^{(1)}\left.\right|_{\Sigma_{0}}<\infty, \ \ \ \ \lim_{r\rightarrow \infty}r^{2}\partial_{r}\psi^{(1)}\left. \right|_{\Sigma_{0}}<\infty.  \]
We call $\psi^{(1)}$ the \textit{time integral} of $\psi$. The uniqueness of $\psi^{(1)}$ is due to an integrability condition that the limit $\lim_{r\rightarrow \infty}r^{2}\partial_{r}\psi^{(1)}\left. \right|_{\Sigma_{0}}$ is required to satisfy that ensures the global smoothness of  $\psi^{(1)}$. Furthermore, the Newman--Penrose constant $I_{0}[\psi^{(1)}]$ of $\psi^{(1)}$ is finite, and in fact equal to 
\begin{equation}
\begin{split}
I_0[\psi^{(1)}]=& -\lim_{r\to \infty}r^3\partial_{r}\phi\left.\right|_{\Sigma_{0}}+M R(2-Dh_{\Sigma_{0}}(R))\phi\left.\right|_{\Sigma_{0}\cap\left\{r=R \right\}}+2M\int_{r\geq R}r L\phi\Big|_{\mathcal{N}_{0}}\,dv\\
&-M\int_{r_{\rm min}}^R 2(1-h_{\Sigma_{0}}D)r\partial_{\rho}\phi-(2-Dh_{\Sigma_{0}})rh_{\Sigma_{0}} T\phi-(r\cdot (Dh_{\Sigma_{0}})')\cdot\phi\Big|_{\Sigma_{0}}\,d\rho,
\end{split}
\label{tinpconstaintro}
\end{equation}
where $\phi=r\psi$. Here, the function $h_{\Sigma_{0}}$ depends on the geometry of the hypersurface $\Sigma_{0}$ defined in Section \ref{sec:foliations}. The vector field $L$ is defined in Section \ref{thespacemanif}.  In particular, if the initial data for $\psi$ is compactly supported in $\{r_{\rm min}\leq r<R\}$, we have that
\begin{equation*}
\begin{split}
I_0[\psi^{(1)}]=&\:-M\int_{r_{\rm min}}^R 2(1-h_{\Sigma_{0}}D)r\partial_{\rho}\phi-(2-Dh_{\Sigma_{0}})rh_{\Sigma_{0}} T\phi-(r\cdot (Dh_{\Sigma_{0}})')\cdot\phi\Big|_{\Sigma_{0}}\,d\rho.
\end{split}
\end{equation*}
This constant will be called the \textit{time-inverted Newman--Penrose constant} of $\psi$ and denoted by $I_{0}^{(1)}[\psi]$. Clearly, $I_0[\psi^{(1)}]$ is \textbf{non-zero} for generic smooth compactly supported initial data for $\psi$. We will use this property of $\psi^{(1)}$ to obtain late-time asymptotics for $\psi^{(1)}$, and subsequently for $\psi$.

\medskip

\textbf{Step 4: Bounds for the Newman--Penrose scalar away from infinity}

\smallskip

Along any fixed $u=c$ hypersurface we have that $\lim_{r\rightarrow \infty}r^{2}\partial_{r}(r\psi^{(1)})(u,r)=I_{0}^{(1)}[\psi]$. One crucial observation of our analysis is that we obtain bounds for the rescaled derivative $r^{2}\partial_{r}(r\psi^{(1)})(u,r)$ away from infinity. In fact we obtain the following bound for the modified derivative $v^{2}\partial_{v}(r\psi^{(1)})(u,r) $
everywhere to the right of the curve $\gamma_{\alpha}$ (see Figure \ref{fig:p1p2}) where $r\sim v^{\alpha}$ with $\alpha<1$ but sufficiently close to $1$. More precisely, we define the curve $\gamma_{\alpha}$ as follows:
\[\gamma_{\alpha}:=\left\{ v-u=v^{\alpha} \right\}  \cap\left\{ r\geq R \right\}, \ \text{ with } \frac{2}{3}< \alpha<1. \]
Recall that $r^{*}=\frac{1}{2}(v-u)$ and since for $r\geq R$ we have $r^{*}\sim r$ we obtain that indeed $r\sim (v-u)\sim v^{\alpha}\sim u^{\alpha}$ along $\gamma_{\alpha}$. We next define the region $\mathcal{B}_{\alpha}$ to be the spacetime region to the right of $\gamma_{\alpha}$ (see Figure \ref{fig:p1p2}): 
\[\mathcal{B}_{\alpha}:=\left\{ v-u\geq v^{\alpha} \right\}  \cap\left\{ r\geq R \right\}, \ \text{ with } \frac{2}{3}< \alpha<1. \]

 \begin{figure}[H]
   \centering
		\includegraphics[scale=0.9]{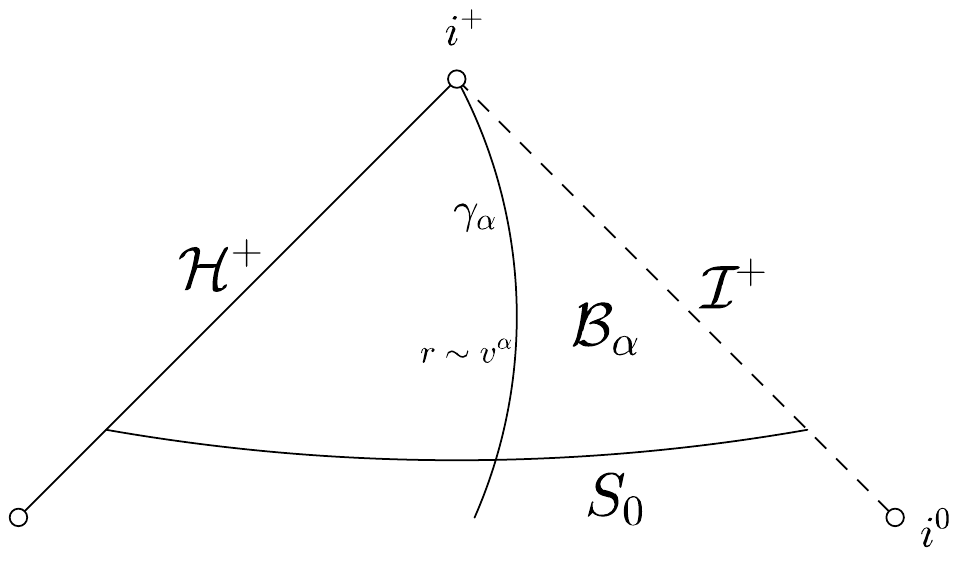}
		\caption{The curve $\gamma_{\alpha}$ and the region $\mathcal{B}_{\alpha}$}
	\label{fig:p1p2}
\end{figure}

We show the following estimate (see Proposition \ref{prop:lower1}):
\begin{equation}
 \left| \partial_{v}(r\psi^{(1)})-I_{0}[\psi^{(1)} ] \cdot \frac{2}{v^{2}}\right|\leq C\sqrt{E[\psi^{(1)}]}\cdot \frac{1}{v^{2+\eta}} 
\label{intronpbound}
\end{equation} 
everywhere in $\mathcal{B}_{\alpha}$, for some $\eta(\alpha)>0$ which has the property that $\eta(\alpha)\rightarrow 1$ as $\alpha \rightarrow 1$.

\medskip

\textbf{Step 5: Asymptotics of the radiation field $r\psi^{(1)}$ of the time integral $\psi^{(1)}$}

\smallskip

The next step is to use the estimate \eqref{intronpbound} in order to obtain asymptotics for $r\psi$ in an appropriate neighborhood of infinity. By integrating in the $v$ direction from $\gamma_{\alpha}$ we obtain for any point $(u,v)\in\mathcal{B}_{\alpha}$: 
\begin{equation}
r\psi^{(1)} (u,v)=r\psi^{(1)}  (u,v_{\gamma_{\alpha}}(u))+\int_{v_{\gamma_{\alpha}}(u)}^{v}\partial_v(r\psi^{(1)} ) (u,v')\,dv'.
\label{introvinte}
\end{equation}
We first need to obtain an estimate for the boundary term on the right hand side. We make use of an almost-sharp decay estimate that was proved using the vector field approach introduced in \cite{paper1}. We remark that the approach of \cite{paper1} allowed us to obtain almost-sharp bounds for solutions to the wave equation with {\textbf{non-vanishing} Newman--Penrose constant:
\[ \left| r^{1/2}\psi^{(1)}  \right| \leq C\cdot \sqrt{E_{\epsilon}}\cdot \frac{1}{u^{3/2-\epsilon}} \]
for $\epsilon>0$. We also use the following bound for the area-radius $r$ along $\gamma_{\alpha}$:
\[ r\sim v^{\alpha}\sim u^{\alpha}.  \]
Hence we obtain along $\gamma_{\alpha}$:
\begin{equation}
\left| r\psi^{(1)} \left. \right|_{\gamma_{\alpha}}\right| =r^{1/2}\cdot \left| r^{1/2}\psi^{(1)}  \left. \right|_{\gamma_{\alpha}}\right|\leq C\sqrt{E_{\epsilon}}\cdot \frac{1}{u^{\frac{3}{2}-\frac{\alpha}{2}-\epsilon}}.
\label{introboundongamma}
\end{equation}

Note that for $\alpha<1$ and $\epsilon>0$ sufficiently small we have that the above decay rate is strictly faster than 1, that is $\frac{3}{2}-\frac{\alpha}{2}-\epsilon>1$.

On the other hand, the integral term on the right hand side of \eqref{introvinte} is the term that is responsible for the first-order late-time asymptotics of $r\psi^{(1)}$. Using \eqref{intronpbound} and omitting various distracting technical terms we obtain (see also estimate \eqref{estimateforvintegral})
\begin{equation}
\begin{split}
\Bigg|\int_{v_{\gamma_{\alpha}}(u)}^{v}&\partial_v(r\psi^{(1)})(u,v')\,dv'-2I_0 [\psi^{(1)} ] ((u+1)^{-1}-v^{-1})\Bigg|\\
\leq&\: C\sqrt{E^{\epsilon}_{0,I_0\neq 0;0}[\psi^{(1)}]}\cdot u^{\frac{3}{2}-\frac{\alpha}{2}-\epsilon  }.
\end{split}
\label{intropartialv}
\end{equation}
Using \eqref{introvinte}, \eqref{introboundongamma} and \eqref{intropartialv} we immediately obtain that the first-order asymptotics for the radiation field  $\phi^{(1)}=r\psi^{(1)}$ on null infinity ($v=\infty$) is given by:
\[r\psi^{(1)}\left.\right|_{\mathcal{I}}\sim_{\text{asym}} 2I_{0}[\psi^{(1)}]\cdot \frac{1}{u+1}   \]
as $u\rightarrow \infty$. In fact, the same technique allows us to obtain quantitative estimates and compute the asymptotics for $r\psi^{(1)}$ in a spacetime neighborhood  $\mathcal{B}_{\delta}$ of null infinity region for an appropriate $\delta >\alpha+6\epsilon$. See also Proposition \ref{cor1}.

\medskip

\textbf{Step 6:  Asymptotics of $\psi^{(1)}$}

Having obtained the asymptotics for $r\psi^{(1)}$ in the region $\mathcal{B}_{\delta}$ we next derive the asymptotics for $\psi^{(1)}$ in full spacetime $\mathcal{M}$. We first obtain the asymptotics in the region $\mathcal{B}_{\delta}$ by splitting it in the following two sub-regions:
\[\mathcal{B}_{\delta}=\mathcal{B}_{\delta}^{(1)}\cup \mathcal{B}_{\delta}^{(2)}, \]
where 
\[ \mathcal{B}_{\delta}^{(1)}=  \mathcal{B}_{\delta}\cap \left\{r\geq Cv \right\}  \]
and 
\[  \mathcal{B}_{\delta}^{(2)}= \mathcal{B}_{\delta}\cap\left\{r\leq Cv \right\}, \]
for an appropriate constant $C$. For  the technical details regarding this splitting see the proof of Proposition \ref{prop:psi_lower}.

It is immediate to obtain the asymptotics for $\psi^{(1)}$ in the region $\mathcal{B}_{\delta}^{(1)}$ by simply dividing by $r$ the main estimate for $\phi^{(1)}$. 

To obtain the asymptotics in the region $\mathcal{B}_{\delta}^{(2)}$ we again divide by $r$ the main estimate for $\phi^{(1)}$ and this time use that 
\[r\sim v-u\geq v^{\delta}. \]
The above combined with the fact that $u\sim v$ in $\mathcal{B}_{\delta}^{(2)}$ (in view of $r\leq Cv$) yield the following asymptotics in the whole region $\mathcal{B}_{\delta}$:
\begin{equation}\label{price_0_v1intro}
\psi^{(1)} (u,v) \sim_{\text{asym}}  \frac{4I_0 [\psi^{(1)} ]}{(u+1)v},
\end{equation}
asymptotically as $u\rightarrow \infty$.

It remains to obtain the asymptotics for $\psi^{(1)}$ in the complement of the region $\mathcal{B}_{\delta}$. We show that the leading first order term in the time-asymptotic expansion of $\psi^{(1)}$ originates from the analogous term along the curve $\gamma_{\delta}$ (which in turn originates from the non-vanishing of the Nemwan--Penrose constant and specifically the bounds on the rescaled derivative $v^{2}\partial_{v}(r\psi^{(1)}$). 

We next consider the region $\left\{r\geq R \right\}$. By integrating in the null direction $\partial_{v}$, tangential to the hypersurface $\Sigma_{\tau}$, we obtain
\begin{equation}
\begin{split}
|\psi^{(1)}(u,v)-\psi^{(1)}(u,v_{\gamma_{\delta}}(u))|=&\:\left|\int_v^{v_{\gamma_{\delta}}(u)}\partial_v\psi^{(1)}(u,v')\,dv'\right|\\
\leq &\: \int_v^{v_{\gamma_{\delta}}(u)}r^{-\frac{1}{2}}\cdot r^{\frac{1}{2}}|\partial_v\psi^{(1)} |(u,v')\,dv' \\
 \leq & \left\| r^{\frac{1}{2}}\partial_v\psi^{(1)}  \right\|_{L^{\infty}(\Sigma_{\tau})} \cdot  \int_v^{v_{\gamma_{\delta}}(u)}r^{-\frac{1}{2}}(v')\,dv' \\
\end{split}
\label{introrestimate}
\end{equation}
Using that $dv\sim dr$ along $\Sigma_{\tau}$ we can bound the integral on the right hand side by 
\[r^{1/2}(u,v_{\gamma_{\delta}}(u))=  v_{\gamma_{\delta}}^{\delta/2}(u)\sim u^{\delta/2}.\]
For the weighted derivative we have the bound
\begin{equation}
\left| r^{\frac{1}{2}}\partial_v\psi^{(1)} \right| \leq C\cdot \sqrt{E_{\epsilon}[\psi^{(1)}]}\cdot \frac{1}{u^{5/2-\epsilon}}. 
\label{intropartialb}
\end{equation}
Note that this is an almost-sharp bound for the rescaled derivative which holds for solutions to the wave equation with non-vanishing Newman--Penrose constant. The techniques of \cite{paper1} in conjuction with almost-sharp decay estimates for the stationary $T$ derivatives of the scalar field and an elliptic estimates presented in this paper are the key ingredients for showing \eqref{intropartialb}.

We conclude, therefore, that the right hand side of \eqref{introrestimate} decays like 
\[\frac{1}{u^{\frac{5}{2}-\frac{\delta}{2}-\epsilon}}. \]
Note that \textit{if we were to take $\delta=1$ then we would get very slow decay which would not be sufficient to yield asymptotics for $\psi^{(1)}$}. For this reason, it is of paramount importance that we work with the regions $\mathcal{B}_{\alpha}$, bounded by the curves $\gamma_{\alpha}$, with $\alpha<1$.

Using now that $v\leq Cu$, that $\frac{5}{2}-\frac{\delta}{2}-\epsilon>2$ and the asymptotic estimate \eqref{price_0_v1intro} for $\psi^{(1)}(u,v_{\gamma_{\delta}}(u)$ we obtain that 
\begin{equation}\label{price_0_v2intro}
\psi^{(1)} (u,v) \sim_{\text{asym}}  \frac{4I_0 [\psi^{(1)} ]}{(u+1)v},
\end{equation}
asymptotically as $u\rightarrow \infty$ in the region $\left\{ r\geq R\right\}$. The asymptotic estimate \eqref{price_0_v2intro} can be extended to the region $\left\{ r\leq R\right\}$ by using once again almost-sharp bounds for the ``radial'' derivative tangential to the hypersurface $\Sigma_{\tau}$. Note that this latter step makes use of the commuted redshift estimate first obtained in \cite{lecturesMD}. We obtain
\[ \psi^{(1)}\sim_{\text{asym}} \frac{4I_{0}[\psi^{(1)}]}{(\tau+1)^{2}}. \]

\medskip

\textbf{Step 7:  Asymptotics of $\psi $}

\smallskip

We next need to derive estimates for $\psi=T\psi^{(1)}$.
The first step in this direction is to commute the wave equation with $\partial_{v}$ and obtain (lower) bounds for the second order derivative $\partial_{v}\partial_{v}(r\psi^{(1)})$ by extending the techniques that were used to derive \eqref{intronpbound}. This leads to 
\[ \left| \partial_{v}\big( \partial_{v} (r\psi^{(1)}) \big) +4I_{0}[\psi^{(1)}]\cdot \frac{1}{v^{3}} \right| \leq C
\sqrt{E[\psi^{(1)}]}\cdot \frac{1}{v^{3+\eta}}  \]
everywhere in the region $\mathcal{B}_{\alpha_{1}}$, for some $\alpha_{1}\in (3/4,1)$ and for
some $\eta=\eta(\alpha_{1})$ for which $\eta(\alpha_{1})\rightarrow 1$ as $\alpha_{1} \rightarrow 1$.

Using the above estimate and the fact that the stationary Killing field $T$ satisfies $T\sim  \partial_{v}+\partial_{u}$ and the wave equation we obtain bounds for $\partial_{v}\big(T(r\psi^{(1)})\big)$ which by integration in the $v$-direction yields the following asymptotic estimate (see Proposition \ref{asymprecTk})
\[ T(r\psi^{(1)})\sim_{\text{asym}}  -2I_{0}[ \psi^{(1)} ]\cdot\left( \frac{1}{u^{2}}-\frac{1}{v^{2}} \right)    \]
in the region $\mathcal{B}_{\alpha_{1}}$, with $\alpha\in (7/9,1)$. The above estimate, in conjuction with the techniques used for the time integral $\psi^{(1)}$ and the almost-sharp bounds 
\eqref{eq:pointdecayNYpsi1} for the hyperboloidal derivative $r^{1/2}\partial_{\rho}T\psi^{(1)}$, where $\partial_{\rho}$ is radial and tangential to $\Sigma_{\tau}$, yields the following asymptotic estimate for $\psi$:
\begin{equation}
\psi=T\psi^{(1)} \sim_{\text{asym}}  -4 I_{0}[\psi^{(1)}]\cdot\left(1+\frac{u+1}{v} \right) \cdot\frac{1}{u^{2}\cdot v}
\label{introasypsi}
\end{equation}
if $r\geq R$, and 
\begin{equation}
\psi=T\psi^{(1)} \sim_{\text{asym}}  -8I_{0}[\psi^{(1)}]\cdot \frac{1}{(\tau+1)^{3}}.
\label{introasypsibounded}
\end{equation}

\subsection{The main theorems}
\label{themaintheorems}
In this section we list the main theorems that we show. We first introduce various definitions and conventions. 
\subsubsection{Notation}

We consider four-dimensional Lorentzian manifolds $(\mathcal{M},g)$ as defined in Section \ref{thespacemanif}. The  metric $g$ takes the form 
\begin{equation}
g=-D(r)dv^2+2dvdr+r^2(d\theta^2+\sin^2\theta \, d\varphi^2)
\label{intrometric}
\end{equation}
in ingoing Eddington--Finkelstein coordinates $(v,r,\theta,\varphi)$, or
\begin{equation*}
g=-D(r)dudv+r^2(d\theta^2+\sin^2\theta \, d\varphi^2)
\end{equation*}
in double null coordinates $(u,v,\theta,\varphi)$. We denote $T=\partial_{v}$ and $\partial_{r}$ to be taken with respect to the system $(v,r,\theta,\varphi)$. Note that $T$ is a stationary Killing field. From now on, unless stated otherwise, $\partial_{v}$ will denote the null vector field taken with respect to the double null coordinate system  $(u,v,\theta,\varphi)$. See Section \ref{thespacemanif} for details regarding the range of the coordinate charts, the behaviour of the metric component $D(r)$ and the precise definitions of the various vector fields. The spacelike-null hypersurfaces $\Sigma_{\tau}$ are defined in Section \ref{sec:foliations}. 
 
 We study the Cauchy problem for the wave equation \eqref{waveequation} on such backgrounds (see Section \ref{sec:thecauchy}).  We require that solutions $\psi$ to the wave equation  \eqref{waveequation} on such backgrounds satisfy the assumptions of Section \ref{sec:waveassm}.  
 
 For an introduction to the vector field method and the definition of energy currents $J^{V}[\psi]$ associated to a vector field $V$ see Section \ref{sec:thecauchy}. 
 
 The decomposition of $\psi$ in angular frequencies is given in Section \ref{sphericaldecompo}.

 The Newman--Penrose constant $I_{0}[\psi]$ is defined in Section $\ref{npsection}$ and the time-inverted Newman--Penrose constants $I^{(k)}_{0}[\psi]$ of $k$-order are defined in Section \ref{timeinvertednpsection}. The time integral $\psi^{(1)}$ of (the spherical mean of) $\psi$ is also defined and constructed in Section \ref{npsection}.

Finally, the $r$-weighted (higher-order) initial data energy norms $$E^{\epsilon}_{0,I_0\neq 0;k}[\psi],\   E^{\epsilon}_{0,I_0=0;k}[\psi],\  \widetilde{E}^{\epsilon}_{0,I_0\neq 0;k}[\psi],\   \widetilde{E}^{\epsilon}_{0,I_0=0;k}[\psi], \  E^{\epsilon}_{1;k}[\psi], \  E^{\epsilon}_{2;k}[\psi],$$ with $k\in \N_0$ and $\epsilon>0$,  are defined in Appendix \ref{apx:energynorms}.

\subsubsection{Late-time asymptotics for solutions to the wave equation with vanishing Newman--Penrose constant}

The next theorem derives the precise late-time asymptotics for solutions to the wave equation with data for which the Newman--Penrose constant vanishes (including in particular the case of compactly supported initial data).

\begin{theorem}\textbf{(Late-time asymptotics for $\psi$ with vanishing Newman--Penrose constant)}
Let $\psi$ be a solution to the wave equation \eqref{waveequation} on the class of spacetimes $(\mathcal{M},g)$, where $g$ is as given by \eqref{intrometric}, satisfying the geometric assumptions of Section \ref{sec:waveassm}. We consider the harmonic decomposition introduced in Section \ref{sphericaldecompo}
\[\psi=\psi_{0}+\psi_{\ell=1}+\psi_{\ell\geq 2} \]
and assume that the following hold 
\begin{equation*}
\begin{split}
v^{3}\partial_v(r\psi_{0})(0,v)=\lim_{v\to \infty}v^{3}\partial_v(r\psi_{0})(0,v) +O&(v^{-\beta}),\\
E^{\epsilon}_{1}[\psi]=\widetilde{E}_{0,I_0=0;1}^{\epsilon}[\psi_{0}]+ \  E^{\epsilon}_{1;1}[\psi_{\ell=1}]+ \ E^{\epsilon}_{2;1}[\psi_{\ell\geq 2}]<&\infty,   \\
\int_{\Sigma_0}J^N[N^2\psi]\cdot n_0\,d\mu_0<&\infty,   \\
\lim_{r \to \infty }\sum_{|l|\leq 6}\int_{\s^2}(r\Omega^l\psi_{\ell\geq 2})^2\,d\omega\big|_{u'=0}<&\:\infty,\\
\lim_{r \to \infty }\sum_{|l|\leq 4}\int_{\s^2}(r^2\partial_r(r\Omega^l\psi_{\ell\geq 2}))^2\,d\omega\big|_{u'=0}<&\:\infty,\\
\lim_{r \to \infty }\sum_{|l|\leq 2}\int_{\s^2}\left((r^2\partial_r)^2(r\Omega^l\psi_{\ell\geq 2})\right)^2\,d\omega\big|_{u'=0}<&\:\infty,\\
\lim_{r \to \infty }\int_{\s^2}r^{4}\Bigg(\partial_r\Big((r^2\partial_r)^2(r\psi_{\ell\geq 2})\Big)\Bigg)^2\,d\omega\big|_{u'=0}<&\infty.
\end{split}
\end{equation*}
Then the following estimate holds for all $(\tau,r) \in \mathcal{M}\cap\{r\leq R\}$:
\begin{equation}
\begin{split}
\Bigg|\psi& (\tau,r) +8\frac{I_0^{(1)} [\psi ]}{(\tau+1)^{3}}\Bigg| \\
\leq&\: C\left(\sqrt{ E^{\epsilon}_{1}[\psi]}+I_0^{(1)}[\psi]+P_{I_0,\beta;1}[\psi^{(1)}]\right)(\tau+1)^{-3-\epsilon}
\end{split}
\label{intromaintheo1eq}
\end{equation}
and the following estimate holds for all $(u,v)\in \mathcal{M}\cap\{r\geq R\}$: 
\begin{equation}
\begin{split}
\Bigg|\psi& (u,v) +4\frac{I_0^{(1)} [\psi ]}{(u+1)^2v}\left(1+\frac{u}{v}\right)\Bigg|\leq \\
 &\: C\left(\sqrt{ E^{\epsilon}_{1}[\psi]} +I_0^{(1)}[\psi]+P_{I_0,\beta;1}[\psi^{(1)}]\right)(u+1)^{-2-\epsilon}v^{-1}
 \end{split}
 \label{intromaintheo2eq}
\end{equation}
where 
\begin{enumerate}
\item $0<\epsilon<\min\left\{\epsilon_{0},\beta\right\}$ for some $\epsilon_{0}=\epsilon_{0}(\Sigma,D,R)$,

\item $C=C(D,\Sigma,R,\epsilon)$,

\item $I_{0}^{(1)}[\psi] $ is the time-inverted Newman--Penrose constant of $\psi$ (see Definition \ref{timeinvertedNP})
and
\item $\psi^{(1)}$ is the time integral of $\psi_{0}=\int_{\mathbb{S}^{2}}\psi$ (see Section \ref{sectiontimeintegral}) and $P_{I_0,\beta;1}[\psi^{(1)}]$ is as defined in \eqref{def:PI0k}.
\end{enumerate}
 \label{thm:asmpsinpn011}
\end{theorem}
Theorem \ref{thm:asmpsinpn011} is proved in Sections \ref{pointwidec1} and \ref{asymptoticspsi2}. 
We remark that the assumptions of the Theorem \ref{thm:asmpsinpn011} imply that the norms of the initial data on the right hand side of estimates \eqref{intromaintheo1eq} and \eqref{intromaintheo2eq} are finite. See Section \ref{estiamtefortimeinte} and Corollary \ref{cor:estPnorms}.

\begin{theorem}\textbf{(Late-time asymptotics of $T^{k}\psi$ with vanishing Newman--Penrose constant)} Let $k\in\mathbb{N}^{*}$. Under the assumptions of Theorem \ref{thm:asmpsinpn011} and the additional assumptions on $\psi$:
\begin{equation*}
\begin{split}
v^{3}\partial_v(r\psi_{0})(0,v)=\lim_{v\to \infty}v^{3}\partial_v(r\psi_{0})(0,v) +O_{k-1}&(v^{-\beta}),\\
E^{\epsilon}_{k}[\psi]=\widetilde{E}_{0,I_0=0;k}^{\epsilon}[\psi_{0}]+E^{\epsilon}_{1;k}[\psi_{\ell=1}]+ E^{\epsilon}_{2;k}[\psi_{\ell\geq 2}]<&\: \infty, \\
\lim_{r \to \infty }\sum_{|l|\leq 4+2k}\int_{\s^2}(\Omega^l\phi_{\ell\geq 2})^2\,d\omega\big|_{u'=0}<&\:\infty,\\
\lim_{r \to \infty }\sum_{|l|\leq 2+2k}\int_{\s^2}(r^2\partial_r(\Omega^l\phi_{\ell\geq 2}))^2\,d\omega\big|_{u'=0}<&\:\infty,\\
\lim_{r \to \infty }\sum_{|l|\leq 2k}\int_{\s^2}\left((r^2\partial_r)^2(\Omega^l\phi_{\ell\geq 2})\right)^2\,d\omega\big|_{u'=0}<&\:\infty,\\
\lim_{r \to \infty }\sum_{|l|\leq 2k-2s}\int_{\s^2}r^{2s+2}\Bigg(\partial_r^s\Big((r^2\partial_r)^2(\Omega^l\phi_{\ell\geq 2}\Big)\Bigg)^2\,d\omega\big|_{u'=0}<&\infty,
\end{split}
\end{equation*}
for each $1\leq s\leq k$, the following estimate holds in $\mathcal{M}\cap\{r\leq R\}$:
\begin{equation}\label{np0price_0_v2intro}
\begin{split}
\Bigg| T^{k-1}\psi& (\tau,r) -4(-1)^k(k+1)!\cdot \frac{I_0^{(1)} [\psi ]}{(\tau+1)^{k+2}}\Bigg| \\
\leq&\: C\left(\sqrt{E^{\epsilon}_{k}[\psi]}+I_0^{(1)}[\psi]+P_{I_0,\beta ;k}[\psi^{(1)}]\right)(\tau+1)^{-k-2-\epsilon}.
\end{split}
\end{equation}
and the following estimate holds in $\mathcal{M}\cap \left\{r\geq R \right\}$:
\begin{equation}
\begin{split}
\Bigg| T^{k-1}\psi & (u,v) -4(-1)^k k!\cdot \frac{I_0^{(1)} [\psi ]}{(u+1)^{k+1}v}\left(1+\sum_{j=1}^k\left(\frac{u}{v}\right)^j\right)\Bigg |\\
 \leq &\: C\left(\sqrt{E^{\epsilon}_{k}[\psi]}+I_0^{(1)}[\psi]+P_{I_0,\beta ;k}[\psi^{(1)}]\right)(u+1)^{-k-1-\epsilon}v^{-1},
 \end{split}
\label{npintro1}
\end{equation}
 \label{thm:asmpsinpn0rfo00}
\end{theorem}
Theorem \ref{thm:asmpsinpn0rfo00} is proved in Sections \ref{pointwidec1} and \ref{asymptoticspsi2}. 

\begin{remark} 
Note that the leading-order term in the asymptotics of $T^{k-1}\psi$ that appears on the left-hand side 
of the estimate \eqref{np0price_0_v2intro} satisfies
\begin{equation*}
4(-1)^k(k+1)!\cdot \frac{I_{0}^{(1)} [\psi ]}{(\tau+1)^{k+2}}=T^{k-1}\left(\frac{-8 I_0^{(1)} [\psi]}{(1+\tau)^3}\right).
\end{equation*}
Similarly, the leading-order term in the asymptotics of $T^{k-1}\psi$ in the estimate \eqref{npintro1} satisfies
\begin{equation*}
4(-1)^k k!\cdot \frac{I_0^{(1)}  [\psi ]}{(u+1)^{k+1}v}\left(1+\sum_{j=1}^k\left(\frac{u}{v}\right)^j\right)=T^{k-1}\Bigg(\frac{-4I_0^{(1)} [\psi]}{(u+1)^{2}v}\cdot\left(1+\frac{u}{v}\right)\Bigg).
\end{equation*}

Hence, Theorem \ref{thm:asmpsinpn0rfo00} shows that the leading-order term in the asymptotics of $T^k\psi$ can be obtained by taking $k$ $T$-derivatives of the leading-order term in the asymptotics of $\psi$ from Theorem \ref{thm:asmpsinpn011}.
\end{remark}

As an corollary, we obtain the following bounds for the energy flux through the event horizon in the black hole case.
\begin{corollary} \textbf{(Asymptotics for the energy flux)}
Under the assumptions of Theorem \ref{thm:asmpsinpn0rfo0}, the energy flux through the event horizon $\mathcal{H}$ satisfies the following asymptotic bound
\[ \int_{\mathcal{H}\cap\left\{\tau\geq \tau_{\infty} \right\}}J^{T}[\psi]\cdot n_{\mathcal{H}}  \sim_{asym} \ \ \frac{I_{0}^{(1)}[\psi]}{\tau_{\infty}^{7}  }   \]
as $\tau_{\infty}\rightarrow \infty$.
\label{cormaintheore1}
\end{corollary}

We also determine the late-time polynomial tails of the Friedlander radiation field $T^{k}(r\psi)$.

\begin{theorem}\textbf{(Late-time asymptotics of the radiation field $T^{k}(r\psi)$ with vanishing Newman--Penrose constant)}  Under the assumptions of Theorem \ref{thm:asmpsinpn0rfo00} we have the following estimate along null infinity:
\begin{equation*}
\begin{split}
&|T^{k-1} (r\psi) (u,\infty)-(-1)^{k} k! \cdot 2I_0^{(1)}  [\psi] (u+1)^{-k-1}|\\ \: &\ \ \ \ \ \ \ \ \ \ \leq C\left(\sqrt{E^{\epsilon}_{0,I_{0}=0;k-1}[\psi]+ E^{\epsilon}_{1;k-1}[\psi_{\ell=1}] +E^{\epsilon}_{2;k-1}[\psi_{\ell\geq 2}]} +I_0[\psi^{(1)}]\right)(u+1)^{-1-k+\epsilon}
\\&\ \ \ \ \ \ \ \ \ \ \ \ \ \ \ +C\cdot P_{I_{0},\beta; k}  [\psi^{(1)}]\cdot (u+1)^{-1-k-\beta}.
\end{split}
\end{equation*}

 \label{thm:asmpsinpn0rfo0}
\end{theorem}

The next theorem obtains the precise leading order term in the late-time asymptotics for spherically symmetric solutions to the wave equation, even for vanishing constant $I_{0}^{(1)}[\psi]$.

\begin{theorem}\label{prop:tkpsi_lower1}\textbf{(Asymptotics for spherically symmetric solutions)}
Let $\psi$ be a spherically symmetric solution to the wave equation \eqref{waveequation} on the class of spacetimes $(\mathcal{M},g)$, where $g$ is given by \eqref{intrometric}. Let $n\in\mathbb{N}$ and assume the following additional asymptotics for the metric component $D$:
\begin{equation*}
D(r)=1-2Mr^{-1}+\sum_{m=0}^{n-1}d_mr^{-m-1}+O_{3+n}(r^{-n-\beta}),
\end{equation*}
where $d_m\in \mathbb{R}$ for $m=0,1,\cdots, n-1$. We assume that $\psi$ satisfies 
 the geometric assumptions of Section \ref{sec:waveassm}
and  that $$\widetilde{E}^{\epsilon}_{I_0= 0;n}[\psi]+\int_{\Sigma_0}J^N[N^2\psi]\cdot n_0<\infty$$ for some $\epsilon>0$. We assume moreover that
\begin{align*}
r^2\partial_r(r\psi)\left.\!\right|_{\Sigma_{0}}
=\sum_{m=1}^{n}p_{m}r^{-m}+O_{k}(r^{-n-\beta}),
\end{align*}
where $p_m\in \R$  for $m=1,\cdots, N$ and $\beta>\epsilon$.

Let $k\leq n$. If  
\begin{equation*}
I_0^{(0)}[\psi]=\ldots=I^{(k-1)}[\psi]=0,
\end{equation*}
then we have that for all $(u,v)\in \mathcal{R}\cap\{r\geq R\}$ we can estimate
\begin{equation*}
\begin{split}
\Bigg|\psi& (u,v) -4(-1)^k k!\cdot \frac{I_0^{(k)} [\psi ]}{(u+1)^{k+1}v}\left(1+\sum_{j=1}^k\left(\frac{u}{v}\right)^j\right)\Bigg|\\
 \leq&\: C\left(\sqrt{\widetilde{E}^{\epsilon}_{0,I_0\neq 0;k+1}[\psi^{(k)}]}+I_0^{(k)}[\psi]+P_{I_0,\beta;k}[\psi^{(k)}]\right)(u+1)^{-k-1-\epsilon}v^{-1},
 \end{split}
\end{equation*}
where
\begin{enumerate}
\item  $C=C(D,\Sigma,R,k,\epsilon)>0$  is a constant,
\item $\psi^{(k)}$ is the $k^{\text{th}}$ time integral of $\psi$,
\item $I_{0}^{(k)}[\psi] $ is the $k^{\text{th}}$-order time-inverted Newman--Penrose constant of $\psi$ and
\item $P_{I_0,\beta;k}[\psi^{(k)}]$ is defined in \eqref{def:PI0k}.
\end{enumerate}

Furthermore, we can estimate in $\mathcal{R}\cap\{r\leq R\}$:
\begin{equation}\label{np0price_0_v2}
\begin{split}
\Bigg|\psi& (\tau,\rho) -4(-1)^k(k+1)!\cdot \frac{I_0^{(k)} [\psi ]}{(\tau+1)^{k+2}}\Bigg| \\
\leq&\: C\left(\sqrt{\widetilde{E}^{\epsilon}_{0,I_0\neq 0;k+1}[\psi^{(k)}]}+I_0^{(k)}[\psi]+P_{I_0,\beta;k}[\psi^{(k)}]\right)(\tau+1)^{-k-2-\epsilon}.
\end{split}
\end{equation}
\end{theorem}

Theorem \ref{prop:tkpsi_lower1} is proved in Section \ref{asymptoticspsi2}.

\subsubsection{Late-time asymptotics for solutions to the wave equation with non-vanishing Newman--Penrose constant}

We next consider initial data for $\psi$, such that $I_0[\psi]\neq 0$ and prove the existence and precise form of the corresponding late-time polynomial tails.

\begin{theorem}\textbf{(Late-time asymptotics of $T^{k}\psi$ with non-vanishing Newman--Penrose constant)}
\label{thm:asmpsinpn0rf2}
Let $\psi$ be a solution to the wave equation \eqref{waveequation} on the class of spacetimes $(\mathcal{M},g)$, where $g$ is as given by \eqref{intrometric}, satisfying the geometric assumptions of Section \ref{sec:waveassm}. We consider the harmonic decomposition introduced in Section \ref{sphericaldecompo}
\[\psi=\psi_{0}+\psi_{\ell=1}+\psi_{\ell\geq 2} \]
and assume that the following hold  
\[ I_{0}[\psi]\neq 0,  \]
\[\widetilde{E}^{\epsilon}_{0,I_0\neq0;k+1}[\psi_{0}]< \infty,\] and
\begin{equation*}
P_{I_0,\beta;k}[\psi_{0}]<\infty,
\end{equation*}
where $P_{I_{0},\beta;k}[\psi_{0}]$ is defined in \eqref{def:PI0k},  for some $\epsilon, \beta>0$. Then we have that for all $(u,v)\in \mathcal{R}\cap\{r\geq R\}$:\begin{equation}\label{price_0_vintroo}
\begin{split}
\Bigg|T^k\psi& (u,v) -4(-1)^k k!\cdot \frac{I_0 [\psi ]}{(u+1)^{k+1}v}\left(1+\sum_{j=1}^k\left(\frac{u+1}{v}\right)^j\right)\Bigg|\\
 \leq&\: C\left(\sqrt{\widetilde{E}^{\epsilon}_{0,I_0\neq 0;k+1}[\psi_{0}]+E^{\epsilon}_{1;k+1}[\psi_{\ell=1}]+E^{\epsilon}_{2;k+1}[\psi_{\ell\geq 2}]}+I_0[\psi]\right)(u+1)^{-k-1-\epsilon}v^{-1}\\ 
 & +P_{I_0,\beta;k}[\psi_{0}]\cdot (u+1)^{-k-1-\beta}v^{-1},
 \end{split}
\end{equation}
where $C=C(D,\Sigma,R,k,\epsilon)>0$ is a constant.

Furthermore, we can estimate in $\mathcal{R}\cap\{r\leq R\}$:
\begin{equation}\label{price_0_v2}
\begin{split}
\Bigg|T^k\psi& (\tau,\rho) -4(-1)^k(k+1)!\cdot \frac{I_0 [\psi ]}{(\tau+1)^{k+2}}\Bigg| \\
\leq&\: C\left(\sqrt{\widetilde{E}^{\epsilon}_{0,I_0\neq 0;k+1}[\psi_{0}]+E^{\epsilon}_{1;k+1}[\psi_{\ell=1}]+E^{\epsilon}_{2;k+1}[\psi_{\ell\geq 2}]}+I_0[\psi]\right)(\tau+1)^{-k-2-\epsilon}\\
&+ P_{I_0,\beta;k}[\psi_{0}]\cdot (\tau+1)^{-k-2-\beta},
\end{split}
\end{equation}
where $C=C(D,\Sigma,R,k,\epsilon)>0$ is a constant.

\end{theorem}

Theorem \ref{thm:asmpsinpn0rf2} is proved in Section \ref{sec:asympsinonzeronp}.

\begin{theorem}\textbf{(Late-time asymptotics of $T^{k}(r\psi)$ with non-vanishing Newman--Penrose constant)}
\label{thm:asmphinpn0}
Under the assumptions of Theorem \ref{thm:asmpsinpn0rf2} we have  
along null infinity 
\begin{equation*}
\begin{split}
|T^{k} (r\psi) (u,\infty)&-(-1)^{k} k! \cdot 2I_0 [\psi ] (u+1)^{-k-1}| \\ \leq&\: C\left(\sqrt{E^{\epsilon}_{I_0\neq 0;k}[\psi_{0}]+E^{\epsilon}_{1;k}[\psi_{\ell=1}]+E^{\epsilon}_{1;k}[\psi_{\ell\geq 2}]}+I_0[\psi]\right)(u+1)^{-1-k+\epsilon}\\
&+C\cdot P_{I_{0},\beta; k}  [\psi]\cdot (u+1)^{-1-k-\beta}.
\end{split}
\end{equation*}

\end{theorem}

Theorem \ref{thm:asmphinpn0} is proved in Section \ref{rftkhiasy}.

\subsection{Relation with the work of Luk--Oh}
\label{lukoh}
In \cite{Luk2015} Luk--Oh establish the blow-up of a weighted energy flux along the event horizon of sub-extremal Reissner--Nordstr\"om spacetimes for generic spherically symmetric initial data via a contradiction argument which uses fundamentally the construction of compactly supported spherically symmetric solutions
 to the wave equation satisfying
\begin{equation}
\label{eq:lukohL}
\mathfrak{L}=\lim_{v\to \infty} 2r^3\partial_v(r\psi)(0,v)-M\int_{0}^{\infty}\lim_{v\to \infty}r\psi(u,v)\,du\neq 0,
\end{equation}
with $(u,v)$ double null coordinates. Using the linearity of the wave equation it therefore follows that $\mathfrak{L}\neq 0$ for \emph{generic} (compactly supported) initial data. 

Using the theory we develop in Section \ref{sec:consttinvphi} we conclude that, under the assumptions of Luk and Oh,  the time integral $\psi^{(1)}$ of the spherical mean of $\psi$, which satisfies $T\psi^{(1)}=\int_{\mathbb{S}^{2}}\psi$, is well-defined. Using then that $\psi^{(1)}$ is a solution to the wave equation, one can easily show that
\begin{equation*}
\mathfrak{L}=I_0[\psi^{(1)}]=I_0^{(1)}[\psi].
\end{equation*}

By involving the time integral of $\psi$ we have therefore provided an \emph{alternative} expression for $\mathfrak{L}$ as the time-inverted Newman--Penrose constant $I_0^{(1)}[\psi]$ (see Definition \ref{timeinvertedNP}), \emph{which depends only on initial data for $\psi$, rather than on the solution globally as in \eqref{eq:lukohL}.}

Hence, the origin of the \emph{pointwise} lower bounds that follow from the precise asymptotics in the present paper is the \emph{same} as the origin of the blow-up of the weighted energy fluxes in \cite{Luk2015}; namely, the non-vanishing of the time-inverted Newman--Penrose constant.

The quantity $\mathfrak{L}$ also plays an important role in understanding strong cosmic censorship for the spherically symmetric Einstein--Maxwell-scalar field system; see \cite{Luk2016a,Luk2016b}.

\subsection{Outline}
The class of Lorentzian metrics considered in this work is presented in Section \ref{sec:geomassm}. The Newman--Penrose constant, which, as is shown in this paper, plays a fundamental role in determining the  leading order term in the late-time asymptotic expansion for solutions to the wave equation, is presented in Section \ref{npsection}. Global almost-sharp decay estimates for $\psi-{\psi}_{0}$, where ${\psi}_{0}$ is the spherical mean, are presented in Section \ref{gdeforpsi101}. Almost-sharp decay estimates for the spherical mean ${\psi}_{0}$ are derived in Section \ref{gdeforpsi1} for initial data with non-vanishing Newman--Penrose constant and in Section \ref{gdeforpsi2} for initial data with vanishing Newman--Penrose constant. Global estimates for the radial derivative tangential to hyperboloidal hypersurfaces are obtained in Section \ref{improvedhyperboloidaldecay}. The latter estimates are crucial for obtaining the late-time asymptotics in the entire spacetime. The late-time asymptotics for solutions to the wave equation with non-vanishing Newman--Penrose constant are obtained in Section \ref{sec:asymradfieldcaseI}. Finally, in Section \ref{sec:asympsi2} we derive the asympotics for solutions to the wave equation with vanishing Newman--Penrose constant, including as a special case the case of compactly supported initial data. The results of Section \ref{sec:asympsi2} rely on the analysis of the non-vanishing Newman--Penrose constant case in Section \ref{sec:asymradfieldcaseI} and the inversion of the time-translation operator $T$ accomplished in Section \ref{sec:consttinvphi}.

\subsection{Acknowledgments}

We would like to thank Mihalis Dafermos, Georgios Moschidis, Jonathan Luk  and Sung-Jin  Oh for several insightful discussions. The second author (S.A.) acknowledges support through NSF grant DMS-1600643 and a Sloan Research Fellowship. The third author (D.G.) acknowledges support by the European Research Council grant no. ERC-2011-StG 279363-HiDGR.

\section{The geometric setting}
\label{sec:geomassm}

\subsection{The spacetime manifold}
\label{thespacemanif}
In this section, we define precisely the stationary, spherically symmetric asymptotically flat Lorentzian manifolds(-with-boundary) $(\mathcal{M},g)$ on which we study the wave equation \eqref{waveequation}. We will look at two different cases.
\subsubsection{Case I: $r_{\rm min}=r_+>0$}
Let $R>0$. In the first case, we let $D$ be a smooth function $D:[r_+,\infty)\to \R$, with $0<r_+<R$, such that $D(r)>0$ for $r\in (r_+,\infty)$, $D(r_+)=0$ and $D'(r_+)\neq 0$.

Note that these assumptions imply in particular that there exists a smooth function $d: [r_{+},\infty)\to \R$, such that $d(r)>0$ and
\begin{equation}
\label{eq:Dvanish1storder}
D(r)=(r-r_+)d(r).
\end{equation}

Furthermore, we assume that
\begin{equation}
\label{asm:asympD}
D(r)=1-\frac{2M}{r}+O_3(r^{-1-\beta}),
\end{equation}
for some $M\geq 0$ and $\beta>0$. Here, we have applied big O notation that is introduced in Section \ref{sec:addnotconv}.

We define the manifold-with-boundary $\mathcal{M}_+$: $$\mathcal{M}_+=\R \times [r_+,\infty)\times   \s^2,$$ such that $\mathcal{M}_+$ is covered by the coordinate chart $(v,r,\theta,\varphi)$ (with the usual degeneration of standard spherical polar coordinates on the unit round sphere $\s^2$), with $v\in \R$, $r\in [r_+\infty)$, $\theta\in (0,\pi)$, $\varphi\in (0,2\pi)$. We equip $\mathcal{M}_+$ with the metric
\begin{equation}
\label{metricvrcoords}
g=-D(r)dv^2+2dvdr+r^2(d\theta^2+\sin^2\theta d\varphi).
\end{equation}

The vector field $\partial_v$ is a Killing vector field that is timelike in the region $\{r_+<r<\infty\}$ and null along $\{r=r_+\}$. We will denote it by $T$.

The boundary $\mathcal{H}^+=\{(v,r,\theta,\varphi)\,:\,r=r_+\}$ is a null hypersurface, which from now on will be called the \emph{future event horizon} of the spacetime.

Let
\begin{equation*}
u=v-2r_*,
\end{equation*}
where $r_*$ is defined as
\begin{equation}
\label{eq:defrstar}
r_*=R+\int_R^r D^{-1}(r')\,dr'.
\end{equation}

Then the coordinate chart $(u,r,\theta,\varphi)$ covers $\mathcal{M}_+\setminus \mathcal{H}^+$, with $u\in \R$, $r\in (r_+,\infty)$, $\theta\in (0,\pi)$ and $\varphi\in (0,2\pi)$.

By the assumptions on $\mathcal{M}_+$ and $D$ above, we can also consider the extended manifold-with-boundary $\mathcal{M}_-$, which is defined as $$\mathcal{M}_-= \R \times [r_+,\infty) \times \s^2,$$ such that $\mathcal{M}_-$ is covered by the coordinate chart $(u,r,\theta,\varphi)$, with $u\in \R$, $r\in [r_+\infty)$, $\theta\in (0,\pi)$, $\varphi\in (0,2\pi)$ and $\mathcal{M}_-$ is equipped with the metric
\begin{equation}
\label{metricurcoords}
g=-D(r)du^2-2dudr+r^2(d\theta^2+\sin^2\theta d\varphi).
\end{equation}
In these coordinates $T=\partial_u$.

The boundary $\mathcal{H}^-=\{(u,r,\theta,\varphi)\,:\,r=r_+\}$ is a null hypersurface. We will refer to $\mathcal{H}^-$ as the \emph{past event horizon} of the spacetime.

We will denote $$\mathcal{M}=\mathcal{M}_+\cup \mathcal{M}_-=\mathcal{M}_+\cup \mathcal{H}^-=\mathcal{M}_-\cup \mathcal{H}^-.$$ See Figure \ref{fig:fullfoliations2} for the corresponding Penrose diagram. The minimum value of $r$ on $\mathcal{M}$ is denoted as $r_{\rm min}$, so $r_{\rm min}=r_+$ in this case. In this paper, we will only be concerned with $\mathcal{M}_+$.

It will also be useful to cover the manifold $\mathcal{M}_+\cap \mathcal{M}_-$ with \emph{double-null} coordinates $(u,v,\theta,\varphi)$. The metric then takes on the form:
\begin{equation}
\label{metricuvcoords}
g=-D(r)dudv+r^2(d\theta^2+\sin^2\theta d\varphi).
\end{equation}
Let us denote the double-null coordinate vector fields as follows:
\begin{align*}
L=&\:\partial_v,\\
\underline{L}=&\:\partial_u.
\end{align*}
We have that
\begin{equation*}
T=L+\underline{L}.
\end{equation*}

We can express with respect to $(v,r,\theta,\varphi)$ coordinates
\begin{equation}
\underline{L}=-\frac{1}{2}D\partial_r.
\label{relationlr0}
\end{equation}
Similarly, with respect to $(u,r,\theta,\varphi)$ coordinates we have that
\begin{equation}
L=\frac{1}{2}D\partial_r.
\label{relationlr}
\end{equation}

The domains of outer communication of sub-extremal Reissner--Nordstr\"om black hole spacetimes, for which
\begin{equation*}
D(r)=1-\frac{2M}{r}+\frac{e^2}{r^2},
\end{equation*}
where $|e|<M$ is a constant, are examples of spacetime regions satisfying the above assumptions on $\mathcal{M}$ in the case where $r_{\rm min}=r_+>0$. \footnote{Note that in sub-extremal Reissner--Nordstr\"om $\mathcal{M}$ can be further extended to include a bifurcation sphere, a 2-sphere that connects $\mathcal{H}^+$ and $\mathcal{H}^-$.}
\subsubsection{Case II: $r_{\rm min}=0$}

As a second case, we consider a smooth function $D:[0,\infty)\to \R$, such that $D(r)\geq d_D$, for some constant $d_D>0$. We moreover assume the same asymptotic behaviour as in \eqref{asm:asympD}.

We now consider the manifold $\mathcal{M}=\R\times \R^3$, and the submanifold $$\mathring{\mathcal{M}}=\mathcal{M}\setminus \{\R\times \{0\}\}=\R\times (0,\infty)\times \s^2,$$ which is covered by the coordinate chart $(v,r,\theta,\varphi)$, with $v\in \R$, $r\in (0,\infty)$, $\theta\in (0,\pi)$, $\varphi\in (0,2\pi)$.

We equip $\mathring{\mathcal{M}}$ with the metric $g$ given by the expression \eqref{metricvrcoords}. Alternatively, we can cover $\mathring{\mathcal{M}}$ with the coordinate chart $(u,r,\theta,\varphi)$, where $u=v-2r_*$ and $r_*$ is given by \eqref{eq:defrstar}. Then, $u\in \R$, $r\in (0,\infty)$, $\theta\in (0,\pi)$, $\varphi\in (0,2\pi)$. In these coordinates, the metric is given by the expression \eqref{metricurcoords}. We can also cover $\mathring{\mathcal{M}}$ with $(u,v,\theta,\varphi)$ coordinates, where $u,v\in \R$, $\theta\in (0,\pi)$ and $\varphi\in (0,2\pi)$. The metric is then given by the expression \eqref{metricuvcoords}.

Note that $g$ can be extended to the entire manifold $\mathcal{M}$ after a suitable coordinate change. See Figure \ref{fig:fullfoliations1} for the corresponding Penrose diagram. The infimum of $r$ on $\mathring{\mathcal{M}}$ is denoted as $r_{\rm min}$, so $r_{\rm min}=0$ in this case.

The Minkowski spacetime, for which $D(r)=1$ is an example of a spacetime satisfying the above assumptions on $\mathcal{M}$ in the case where $r_{\rm min}=0$.

\subsection{Foliations}
\label{sec:foliations}

Let $\mathcal{N}_{u'}=\{(u,r,\theta,\varphi)\,:\, u=u',\: r\geq R\}$ be an outgoing null hypersurface. We define $$\mathcal{A}=\bigcup_{u\in[0,\infty)}\mathcal{N}_{u}$$ and $$\mathcal{A}^{u_2}_{u_1}=\bigcup_{u\in [u_1,u_2]}\mathcal{N}_{u}.$$ Note that $\mathcal{A}^{u_2}_{u_1} \subset \mathcal{A}$ for all $0\leq u_1<u_2<\infty$.

Let $h_{\Sigma}: [r_{\rm min},R] \to \R$ be a smooth, positive function. We define the hypersurface $\Sigma$ as follows:
\begin{equation*}
\Sigma_0=\{(v,r,\theta,\varphi)\,;\, v=v_{\Sigma}(r), r\leq R\}\cup \mathcal{N}_0,
\end{equation*}
where $v_{\Sigma}(r)$ is the smooth function satisfying:
\begin{align*}
\frac{dv_{\Sigma}}{dr}=&\:h_{\Sigma},\\
v_{\Sigma}(R)=&\:2r_*(R).
\end{align*}
By construction, $(v,v_{\Sigma}(R),\theta,\varphi)\in \mathcal{N}_0$. See Figure \ref{fig:fullfoliations1} and \ref{fig:fullfoliations2} for Penrose diagrams of the two cases of $\mathcal{M}$ ($r_{\rm min}=r_+$ and $r_{\rm min}=0$) with the hypersurface $\Sigma_0$. Without loss of generality, we may assume that there exists a $v_0>0$ such that $$\min_{r\in [r_{\rm min},R]} v_{\Sigma}\geq v_0>0.$$

Let $\tau$ be a smooth function on $J^+(\Sigma_0)$, such that $\tau|_{{\Sigma_0}}=0$, and $T(\tau)=1$. In $\mathcal{A}$ we have that $\tau=u$.

Let us moreover introduce the notation $\mathcal{I}_{v'}(\tau_1,\tau_2) \doteq \{v=v',\:\tau_1\leq u\leq \tau_2\}$ for ingoing null segments in $\mathcal{A}^{\tau_2}_{\tau_1}$.  We denote
\begin{equation*}
\mathcal{R}=J^+(\Sigma_0)=\bigcup_{\tau\in [0,\infty)}\Sigma_{\tau}.
\end{equation*}

We can also consider the coordinate chart $(\tau,\rho,\theta,\varphi)$ in $\mathcal{R}\cap\{r\leq R\}$, where $\rho=r|_{\Sigma_0}$. Then we can express:
\begin{align*}
\partial_{\tau}=&\:T,\\
\partial_{\rho}=&\:\partial_r+h_{\Sigma}\partial_v\\
=&\:-2D^{-1}\underline{L}+h_{\Sigma}T.
\end{align*}

We will next construct a hyperboloidal spacelike hypersurface terminating at null infinity. Let us first consider a vector field of the form $$Y=\partial_r+h_{\mathcal{S}}(r)\partial_v,$$ with $h_{\mathcal{S}}: [r_{\rm min},\infty) \to \R$ a smooth function that satisfies the following properties:
\begin{align*}
\frac{1}{\max_{r_{\rm min}\:\leq r\leq R}D(r)}&\leq h_{\mathcal{S}}(r)< \frac{2}{D(r)}\quad \textnormal{if}\: r\leq R,\\
0<\frac{2}{D(r)}-h_{\mathcal{S}}(r)&=O_1(r^{-1-\eta})\quad \textnormal{if}\: r>R,
\end{align*}
for some $\eta>0$. As a consequence, $$g(Y,Y)=h_{\mathcal{S}}(r)(2-h_{\mathcal{S}}(r)D(r))>0$$ for all $r\in [r_{\min},\infty)$, so $Y$ is a spacelike vector field.

The corresponding integral curves $\gamma_Y \subset \mathcal{M}$ of $Y$ can be parametrised by $r$, such that $\gamma_Y: [r_{\rm min},\infty)\to \mathcal{R}$, with
\begin{equation*}
\gamma_Y(r)=(v_{\mathcal{S}}(r),r,\theta_0,\varphi_0),
\end{equation*}
in $(v,r,\theta,\varphi)$ coordinates, with $\theta_0,\varphi_0$ fixed, $\frac{dv_{\mathcal{S}}}{dr}=h_{\mathcal{S}}$ and $v_{\mathcal{S}}(R)=v_{1}$. Note that since $v$ increases along $r=R$ towards the future, if we choose $v_{1}$ sufficiently large then we can guarantee that $\gamma_{Y}(R)\in\mathcal{R}$.  

Let's now consider the region where $r_{\rm min}\leq r\leq R$. Then, 
\begin{equation*}
v_{\mathcal{S}}(r)=v_1-\int_{r}^R h_{\mathcal{S}}(r')\,dr'.
\end{equation*}
By choosing $v_1$ suitably large depending on $R$ and $h_{\mathcal{S}}$, we can ensure that $v_{\mathcal{S}}(r)\geq v_{\Sigma}(r)$ for all $r_{\rm min}\leq r\leq R$. This implies that  $\gamma_{Y}(r)\in \mathcal{R}$ for all $r_{\rm min}\leq r\leq R$.

Let's now consider the region where $r \geq  R$.
By using that $u=v-2r_*$, and expressing $\gamma_Y$ in $(u,r,\theta,\varphi)$ coordinates, we obtain:
\begin{equation*}
\gamma_Y(r)=(u_{\mathcal{S}}(r)=v_{\mathcal{S}}(r)-2r_*(r),r,\theta_0,\varphi_0),
\end{equation*}
with $\frac{du_{\mathcal{S}}}{dr}=h_{\mathcal{S}}-\frac{2}{D}$. Therefore, $\frac{du_{\mathcal{S}}}{dr}<0$ and 
\begin{equation*}
\begin{split}
\left|u_{\mathcal{S}}(r)-u_{\mathcal{S}}(R)\right|=&\:\left|\int_{R}^r h_{\mathcal{S}}(r')-\frac{2}{D(r')}\,dr'\right|\\
\leq &\: C_Y (R^{-\eta}-r^{-\eta})\leq C_Y R^{-\eta},
\end{split}
\end{equation*}
for some constant $C_Y>0$, depending on the choice of $h_{\mathcal{S}}$. For $v_0$ suitably large depending on $C_Y$ and $R$, we therefore also have that $u_{\mathcal{S}}(r)>0$, which implies that $\gamma_{Y}(r)\in \mathcal{R}$ for $r\geq R$.

We can now define the \emph{spacelike hyperboloidal hypersurface} $\mathcal{S}_0$ as
\begin{equation*}
\mathcal{S}_0=\{(v,r,\theta,\varphi)\,:\, v=v_{\mathcal{S}}(r),\: r\in [r_{\rm min},\infty)\},
\end{equation*}
and we have that $\mathcal{S}_0\subset \mathcal{R}$.

We moreover define a \emph{hyperboloidal} foliation of a subset of $\mathcal{R}$, with leaves denoted by $\mathcal{S}_{\widetilde{\tau}}$, by flowing $\mathcal{S}_0$ along the integral curves of $T$. We have that

\begin{equation*}
J^+(\mathcal{S}_0)=\bigcup_{\widetilde{\tau}\in[0,\infty)}\mathcal{S}_{\widetilde{\tau}}.
\end{equation*}

We can cover $J^+(\mathcal{S}_0)$ by coordinates $(\widetilde{\tau},\widetilde{\rho},\theta,\varphi)$, with $\widetilde{\rho}=r|_{\mathcal{S}_0}$ and $\partial_{\widetilde{\rho}}= Y$. By construction, there exists a $\tau_0=\tau_0(D, R,\mathcal{S}_0,\Sigma_0)>0$ such that
\begin{equation}
\label{eq:comptimes}
\tau-\tau_0\leq \widetilde{\tau}\leq \tau+\tau_0.
\end{equation}

 \begin{figure}[H]
\centering
\scalebox{0.55}{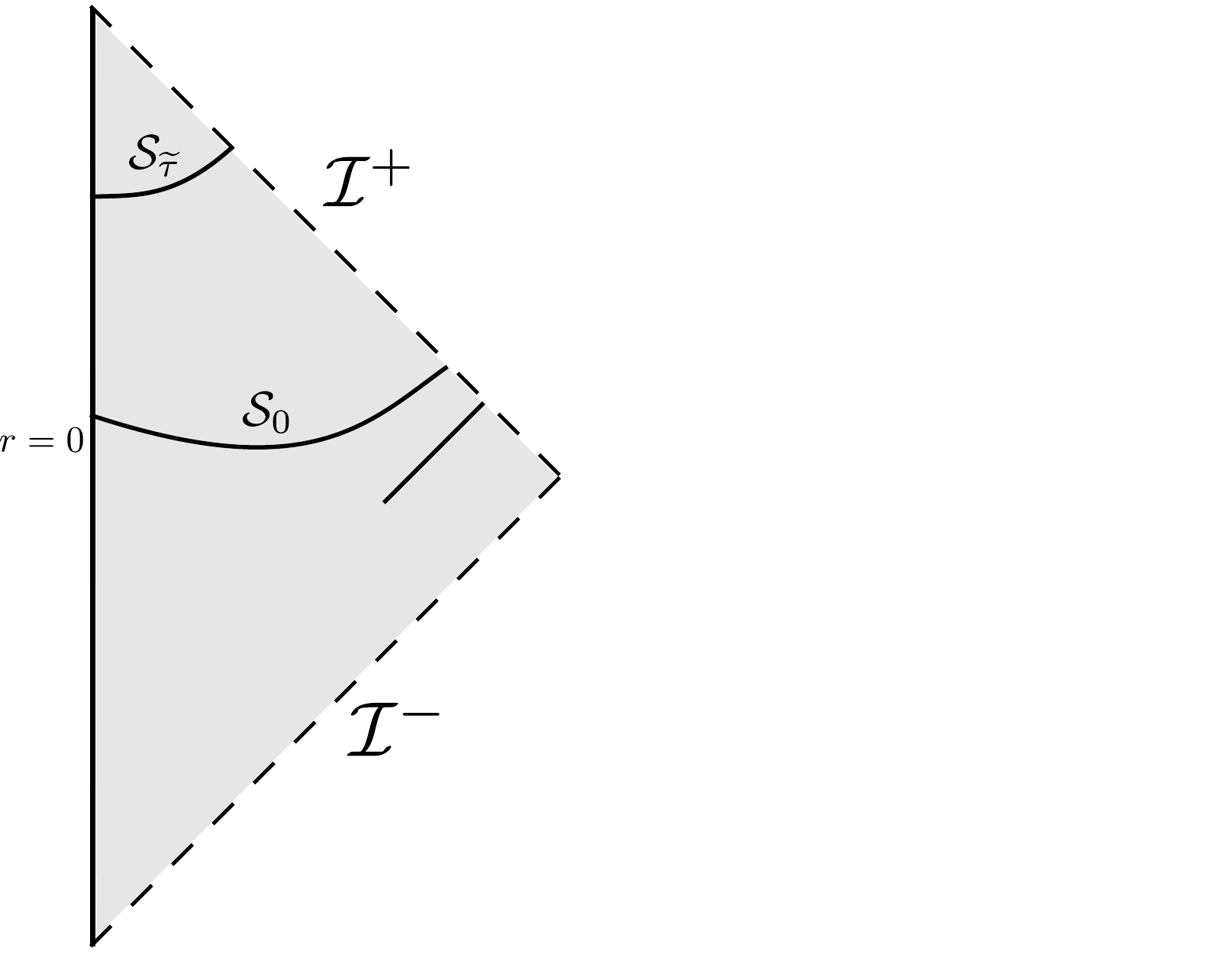}
\caption{\label{fig:fullfoliations1}A Penrose diagram of $\mathcal{M}$ in the case $r_{\rm min}=0$.}
\end{figure}

\begin{figure}[H]
\centering
\scalebox{0.45}{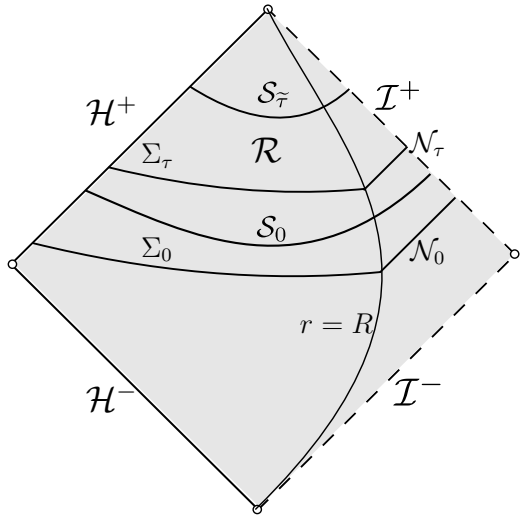}
\caption{\label{fig:fullfoliations2}Penrose diagram of $\mathcal{M}$ in the case $r_{\rm min}=r_+$.}
\end{figure}

\subsubsection{Additional notational conventions}
\label{sec:addnotconv}
In the remainder of the paper we will occasionally use the notation
\begin{equation*}
f \lesssim A,
\end{equation*}
for a positive function $f$ and a constant $A$, to indicate that there exists a \emph{uniform} constant $C>0$, depending only $D,R,\Sigma_0$ and $\mathcal{S}_0$, such that
\begin{equation*}
f\leq C\cdot A.
\end{equation*}

Similarly, we will use the notation
\begin{equation*}
f \sim A
\end{equation*}
to indicate that there exists uniform constants $C>c>0$ such that
\begin{equation*}
c\cdot A \leq f\leq C\cdot A.
\end{equation*}

We will frequently apply ``big O'' notation. We denote with $O_k(r^{-\beta})$, where $\beta \in \R$, a $C^k$ function $f:[r_{\rm min},\infty)\to \R$ that satisfies the following property: for all $0\leq j \leq k$, there exist uniform constants $C_j>0$, such that
\begin{equation*}
\left|\frac{d^jf}{dr^j}\right|\leq C_j r^{-\beta-j}.
\end{equation*}

We will also apply `little O'' notation. We denote with $o_k(r^{-\beta})$, where $\beta \in \R$, a $C^k$ function $f:[r_{\rm min},\infty)\to \R$ that satisfies the following property: for all $0\leq j \leq k$, there exist uniform constants $C_j>0$, such that
\begin{equation*}
\lim_{r\to \infty}r^{\beta+j}\left|\frac{d^jf}{dr^j}\right|=0.
\end{equation*}

We will denote by $d\mu_{\Sigma_{\tau}}$ the induced volume form on $\Sigma_{\tau}$ and by $d\mu_{\mathcal{S}_{\widetilde{\tau}}}$ the induced volume form on $\mathcal{S}_{\widetilde{\tau}}$, where we take $d\mu_{\Sigma_{\tau}}|_{\mathcal{N}_{\tau}}=r^2d\omega dr$ along $\mathcal{N}_{\tau}$. Here, $d\omega=\sin \theta\,d\theta d\varphi$ is the volume form on the round unit sphere in spherical coordinates.

We will furthermore use $n_{\Sigma_{\tau}}$ with the shorthand notation $n_{\tau}$ to denote the future-directed normal along $\Sigma_{\tau}$, with $g(n_{\tau},n_{\tau})|_{\Sigma_{\tau}\setminus \mathcal{N}_{\tau}}=-1$ and $n_{\tau}|_{\mathcal{N}_{\tau}}={L}$. We use $n_{\mathcal{S}_{\widetilde{\tau}}}$ with the shorthand notation $n_{\widetilde{\tau}}$ to denote the future-directed unit normal along $\mathcal{S}_{\widetilde{\tau}}$.

\subsection{The Cauchy problem for the wave equation}
\label{sec:thecauchy}

We  study the Cauchy problem for the wave equation on the spacetime region $(\mathcal{R},g)$ defined in Section \ref{sec:foliations}. The following result provides a global existence and uniqueness statement for the this problem. 
\begin{theorem}
\label{thm:extuniq}
Let $\Psi\in C^{\infty}({\Sigma_0})$, $\Psi'\in C^{\infty}({\Sigma_0}\setminus \mathcal{N}_0)$. Then there exists a unique smooth function $\psi: \mathcal{R} \to \R$ satisfying
\begin{align*}
 \square_g \psi&=0,
 \end{align*}
with initial data
\begin{align*}
\psi|_{{\Sigma_0}}&=\Psi,\\
n_{{\Sigma}_0}(\psi)|_{{\Sigma}_0\setminus \mathcal{N}_0}&=\Psi'.
\end{align*}
\end{theorem}

We  always assume that $\Psi\to 0$ as $r\to \infty$.

The stress-energy tensor  is a symmetric 2-tensor, with components
\begin{equation*}
\mathbf{T}_{\alpha \beta}[f]=\partial_{\alpha}f \partial_{\beta}f-\frac{1}{2}{g}_{\alpha \beta} (g^{-1})^{\kappa \lambda}\partial_{\kappa}f \partial_{\lambda} f,
\end{equation*}
with respect to a coordinate basis. Here $f$ is a function $f:\mathcal{R}\to\R$. Note that $\mathbf{T}_{\alpha \beta}[\psi] $ is divergence free for all solutions $\psi$ to the wave equation.

The energy current $J^V[f]$ is defined as follows
\begin{equation*}
J^V[f]\doteq \mathbf{T}(V,\cdot),
\end{equation*}
where $V$ is a vector field.

We will also write
\begin{equation*}
\textnormal{div}\,J^V[f]=K^V[f]+\mathcal{E}^V[f],
\end{equation*}
where
\begin{align}
\label{def:KV}
K^V[f] \doteq&\: \mathbf{T}^{\alpha \beta}\nabla_{\alpha}V_{\beta},\\
\label{def:EV}
\mathcal{E}^V[f]\doteq&\: V(f)\square_g f.
\end{align}
Clearly, $\mathcal{E}^V[\psi]=0$ for solutions $\psi$ to (\ref{waveequation}).

\subsection{Assumptions for the wave equation}
\label{sec:waveassm}

We assume that the spacetime metric $g$ is such that the following assumptions on the wave operator are satisfied. 

\subsubsection{Energy boundedness for scalar waves}
\label{sec:asmebound}
Let $N$ be a  time-invariant strictly timelike vector field $N$, such that $N=T$ in $\mathcal{A}$. A possible construction for $N$ is the following:  In the $r_{\rm min}=0$ case, we can simply take $N=T$, whereas in the $r_{\rm min}=r_+>0$ case, we can define $N$ such that $N=T-Y$ for $r_+\leq r\leq r_0$ and $N=T$ for $r\geq r_1$, with $r_+<r_0<r_1$, by using a cut-off function.

 We will assume the following energy boundedness statement for the wave equation:   There exists a constant $C=C(D,R,\Sigma)>0$, such that for all $v$
\begin{equation}
\label{ass:ebound}
\int_{\Sigma_{\tau}}J^N[\psi]\cdot n_{\tau}\,d\mu_{\Sigma_{\tau}}+\int_{\mathcal{I}_v(0,\tau)}J^N[\psi]\cdot \underline{L}\: r^2d\omega du\leq C \int_{\Sigma_0}J^N[\psi]\cdot n_{\Sigma_0}\,d\mu_{\Sigma_0},
\end{equation}
assumming that 
 \begin{equation*}
 \int_{\Sigma_0}J^N[\psi]\cdot n_{\Sigma_0}\,d\mu_{\Sigma_0}<\infty.
 \end{equation*}

\subsubsection{Morawetz estimates for scalar waves}
We assume that the following Morawetz estimate holds: there exists a constant $C=C(D,R,\Sigma)>0$, such that for all $0<\tau_1<\tau_2$
\begin{equation}
\label{ass:morawetz}
\int_{\tau_1}^{\tau_2}\left(\int_{\Sigma_{\tau}\setminus \mathcal{N}_{\tau}}J^N[ \psi]\cdot n_{\tau}\,d\mu_{\Sigma_{\tau}}\right)\,d\tau\leq C \int_{\Sigma_{\tau_1}}J^N[\psi]\cdot n_{\tau_1}+J^N[T\psi]\cdot n_{\tau_1}\,d\mu_{\Sigma_{\tau_1}}.
\end{equation}
We  moreover assume a spatially local Morawetz estimate in $\mathcal{A}$ \emph{without}  loss of derivatives on the right hand side:
\begin{equation}
\label{ass:morawetzlocal}
\int_{\tau_1}^{\tau_2}\left(\int_{\mathcal{N}_{\tau}\cap \{\widetilde{R}\leq r\leq \widetilde{R}+1\}}J^T[ \partial^{\alpha}\psi]\cdot n_{\tau}\,d\mu_{\Sigma_{\tau}}\right)\,d\tau\leq C_{\alpha} \sum_{k\leq |\alpha|}\int_{{\Sigma}_{\tau_1}}J^N[T^{k}\psi]\cdot n_{\tau_1},
\end{equation}
for $|\alpha|\geq 0$ and $\widetilde{R}\geq R$, where $C_{\alpha}=C_{\alpha}(D,R,\Sigma,\widetilde{R},\alpha)>0$.

The Morawetz estimates above have in particular been proved for Schwarzschild and more generally for  sub-extremal Reissner--Nordstr\"om spacetimes  in \cite{redshift,blu1}.

\subsection{The elliptic decomposition $\psi=\psi_{0}+\psi_{\ell=1}+\psi_{\ell\geq 2}$}
\label{sphericaldecompo}
In this section, we will introduce a decomposition of solutions $\psi$ to \eqref{waveequation} into three terms, $\psi_0$, $\psi_{\ell=1}$ and $\psi_{\ell\geq 2}$, by employing spherical harmonic modes, the eigenfunctions of the Laplacian on $\s^2$.
\label{sphericaldecompo}
\begin{definition}
\label{def:sphharmmode}
For any suitably regular function $f:\s^2\to \R$, we can decompose
\begin{equation*}
f=\sum_{\ell'=0}^{\infty} f_{\ell=\ell'},
\end{equation*}
where $f_{\ell=\ell'}: \s^2\to \R$ are eigenfunctions of $\slashed{\Delta}_{\s^2}$, the Laplacian on $\s^2$, which are orthogonal with respect to the $L^2$ norm on $\s^2$ and the eigenvalue corresponding to $f_{\ell=\ell'}$ is $-\ell'(\ell'+1)$. They are called the $\ell'$-th spherical harmonic modes.
\end{definition}

In particular,
\begin{equation*}
f_0=\frac{1}{4\pi}\int_{\s^2}f(\theta',\varphi')\,d\omega'.
\end{equation*}
The eigenfunctions $f_{\ell=\ell'}$ are moreover orthogonal with respect to the $L^2$ norm on $\s^2$.

We can use the above decomposition to the restrictions of solutions $\psi: \mathcal{R}\to \R$ to \eqref{waveequation} in $(v,r,\theta,\varphi)$ coordinates to $\s^2$ in order to obtain:
\begin{equation*}
\psi(v,r,\theta,\varphi)=\sum_{\ell'=0}^{\infty} \psi_{\ell=\ell'}(v,r,\theta,\varphi),
\end{equation*}
with
\begin{equation*}
\slashed{\Delta}_{\s^2}\psi_{\ell=\ell'}=-\ell'(\ell'+1)\psi_{\ell=\ell'}.
\end{equation*}

Due to the spherical symmetry (or warped product structure) of the metric $g$, we have that for all $\ell'\geq 0$
\begin{equation*}
\square_g\psi_{\ell=\ell'}=0,
\end{equation*}
so we can perform estimates separately for each solution $\psi_{\ell=\ell'}$ to \eqref{waveequation}.

In this paper it sufficient to split $\psi$ only into three parts:
\begin{equation*}
\psi=\psi_0+\psi_{\ell=1}+\psi_{\ell\geq 2},
\end{equation*}
where
\begin{equation*}
\psi_{\ell\geq 2}=\sum_{\ell'=2}^{\infty} \psi_{\ell=\ell'}.
\end{equation*}
We moreover denote
\begin{equation*}
\psi_1=\psi-\psi_0=\psi_{\ell=1}+\psi_{\ell\geq 2}.
\end{equation*}

\section{The Newman--Penrose constant $I_{0}$ at null infinity}
\label{npsection}

Let $\psi$ be a solution of the wave equation \eqref{waveequation}  emanating from initial data given as in Theorem \ref{thm:extuniq} on a spacetime $(\mathcal{M}, g)$ that satisfies the geometric assumptions from Section \ref{sec:geomassm}.

We  define $I_{0}[\psi](u)$ to be a function on null infinity given by 
\begin{equation}\label{np1_constant}
I_{0}[\psi](u):=\frac{1}{4\pi}\lim_{r\to \infty}\int_{\mathbb{S}^{2}} r^2 \partial_r ( r\psi )( u,r,\omega )\, d\omega.
\end{equation}
In fact, it turns out (see, for example, \cite{paper1}) that $I_{0}[\psi](u)$ is independent of $u$. 
This motivates the following definition
\begin{definition}
The \emph{first Newman--Penrose constant} $I_0[\psi]$ corresponding to a solution $\psi$ to \eqref{waveequation}  is defined as the (unique) value of the function $I_{0}[\psi](u)$ on null infinity. 
\end{definition}

If we consider solutions of \eqref{waveequation} of the form $T\psi$, where $\psi$ itself is a solution to \eqref{waveequation}, the corresponding first Newman--Penrose constant actually \emph{vanishes}.
\begin{proposition}
Let $\psi$ be a solution to the wave equation such that 
\begin{equation}
\label{ass:NpTvanishes}
\frac{1}{4\pi}\int_{\mathbb{S}^{2}}\partial_r(r\psi)(0,r,\omega)\, d\omega=\frac{I_0[\psi]}{r^2}+o_1(r^{-2}). 
\end{equation}
Then
\begin{equation}
\label{nptpsi}
I_{0}[T\psi]=0.
\end{equation}
\end{proposition}
\begin{proof}
Since $I_{0}[T\psi]$ is independent of $u$, it suffice to show that $I_{0}[T\psi]$ vanishes for $u=0$. Since $\psi$ satisfies the wave equation, the spherically symmetric part  $\phi_{0}= \frac{1}{4\pi}r\psi_{0}=r\int_{\mathbb{S}^{2}}\psi\,d\omega$ satisfies the following equation:
\begin{equation*}
2r^2\partial_r\partial_u\phi_{0}=Dr^2\partial_r^2(r\psi_{0})+D'r^2\partial_r(r\psi_{0})+rD'\phi_{0}.
\end{equation*}
Since $I_0[\psi]<\infty$ by \eqref{ass:NpTvanishes}, the second and third terms on the right-hand side vanish in the limit $r\to \infty$, using the asymptotic properties of $D$ from Section \ref{sec:geomassm}. The first term on the right-hand side vanishes along $\left\{u=0\right\}$ in the limit as $r\rightarrow \infty$, because by taking the $r$-derivative of both sides of \eqref{ass:NpTvanishes} we obtain
\begin{equation*}
\partial_r^2(r\psi_{0})(0,r)=-2\frac{I_0[\psi]}{r^3}+o(r^{-3}).
\end{equation*}
We conclude that
\begin{equation*}
\lim_{r\to \infty}r^2\partial_r T\phi_{0}(0,r)=0.
\end{equation*}
\end{proof}

\section{Global decay estimates for $\psi-\frac{1}{4\pi}\int_{\mathbb{S}^{2}}\psi$}
\label{gdeforpsi101}

We will give an overview of several energy decay and pointwise decay statements that are crucial to the late-time asymptotics results in the subsequent sections. The corresponding proofs are very similar to those in \cite{paper1}, but the use of some additional $r^p$-weighted hierarchies from \cite{paper1}.

In this section, $\psi$ will always denote a solution to \eqref{waveequation} emanating from initial data as in Theorem \ref{thm:extuniq}. 

We will moreover make use of the following energy norms on the initial data:
\begin{equation*}
E_{0, I_0 \neq 0;k}^{\epsilon}[\psi_0], E_{0, I_0 = 0;k}^{\epsilon}[\psi_0], E_{1;k}^{\epsilon}[\psi_{\ell=1}], E_{2;k}[\psi_{\ell\geq 2}],
\end{equation*}
with $k\in \N_0$. These are defined in Appendix \ref{apx:energynorms}.

\subsection{Energy decay}
\label{endehigna}
First, we will obtain $\tau$-decay estimate for suitably $r$-weighted (higher-order) energy norms along $\Sigma_{\tau}$.

\begin{proposition}[{Energy decay for $T^k\psi_{\ell=1}$}]\label{decl1}
Let $k\in \N_0$. Under the following assumption on the asymptotics of $D$:
\begin{equation*}
D(r)=1-\frac{2M}{r}+O_{3+k}(r^{-1-\beta}),
\end{equation*}
together with the assumption $E^{\epsilon}_{1;k}[\psi_{\ell=1}]<\infty$, for some $\epsilon\in (0,1)$, there exists a constant $C \doteq C(D,R,\epsilon,k)$ such that for all $\tau\geq 0$
\begin{equation}
\label{eq:edecayl1b}
\begin{split}
\int_{\Sigma_{\tau}}J^N[T^k\psi_{\ell=1}]\cdot n_{\tau}\,d\mu_{\tau}\leq&\: CE^{\epsilon}_{1;k}[\psi_{\ell=1}](1+\tau)^{-6-2k+\epsilon}.
\end{split}
\end{equation}
\end{proposition}
\begin{proof}
A proof proceeds similarly to the $\ell\geq 1$ part of the proof of Theorem 1.8 of \cite{paper1}. We can use the hierarchy for $r(r-M)\partial_r(r\psi)$ in Proposition 4.2 of \cite{paper1} with $p=3-\epsilon$, $p=2-\epsilon$ and $p=1-\epsilon$ because of the restrict to $\ell=1$. We then obtain the same decay rate as in Theorem 1.8 of \cite{paper1}, but with fewer commutations in the energy norms, yet with higher powers of $r$ appearing in the energy norms. However, we can also include the $p=4-\epsilon$ estimate of the hierarchy in Proposition 4.2 of \cite{paper1}, which results in an extra power in the energy decay rate (and additional $r$-weighted terms in the energy norms). Obtaining the energy decay estimates for $T^k$ then proceeds analogously to the corresponding part in the proof of Theorem 1.8 of \cite{paper1}, by considering additional hierarchies that arise from commuting $\square_g$ with $\partial_r^k$, in addition to the commutation with $r(r-M)\partial_r$.
\end{proof}

\begin{proposition}[{Energy decay for $T^k\psi_{\ell\geq 2}$}]\label{decl2T}
Let $k\in \N_0$. Under the following assumption on the asymptotics of $D$:
\begin{equation*}
D(r)=1-\frac{2M}{r}+O_{3+k}(r^{-1-\beta}),
\end{equation*}
together with $E^{\epsilon}_{1;k}[\psi_{\ell\geq 2}]<\infty$, for some $\epsilon\in (0,1)$, and the following additional pointwise assumptions on the initial asymptotics of $\psi$:
\begin{align*}
\lim_{r \to \infty }\sum_{|l|\leq 4+2k}\int_{\s^2}(\Omega^l\phi_{\ell\geq 2})^2\,d\omega\big|_{u'=0}<&\:\infty,\\
\lim_{r \to \infty }\sum_{|l|\leq 2+2k}\int_{\s^2}(r^2\partial_r(\Omega^l\phi_{\ell\geq 2}))^2\,d\omega\big|_{u'=0}<&\:\infty,\\
\lim_{r \to \infty }\sum_{|l|\leq 2k}\int_{\s^2}\left((r^2\partial_r)^2(\Omega^l\phi_{\ell\geq 2})\right)^2\,d\omega\big|_{u'=0}<&\:\infty,
\end{align*}
and
\begin{equation*}
\lim_{r \to \infty }\sum_{|l|\leq 2k-2s}\int_{\s^2}r^{2s+2}\Bigg(\partial_r^s\Big((r^2\partial_r)^2(\Omega^l\phi_{\ell\geq 2}\Big)\Bigg)^2\,d\omega\big|_{u'=0}<\infty,
\end{equation*}
for each $1\leq s\leq k$, there exists a constant $C \doteq C(D,R,\epsilon,k)$ such that for all $\tau\geq 0$
\begin{equation}
\label{eq:edecaylgeq2T}
\begin{split}
\int_{\Sigma_{\tau}}J^N[T^k\psi_{\ell\geq 2}]\cdot n_{\tau}\,d\mu_{\tau}\leq&\: CE^{\epsilon}_{2;k}[\psi_{\ell\geq 2}](1+{\tau})^{-6-2k+\epsilon}.
\end{split}
\end{equation}
\end{proposition}
\begin{proof}
A proof proceeds similarly to the $\ell\geq 1$ part of the proof of Theorem 1.8 of \cite{paper1}. However, because we are restricting $\ell\geq 2$, we can consider additionally the estimate $p=2-\epsilon$ in the hierarchy in Theorem 1.4, resulting in an additional power in the decay rates compared to  Theorem 1.8 of \cite{paper1}. Obtaining the energy decay rates for $T^k\psi$ proceeds almost identically to the corresponding part in the proof of Theorem 1.8 of \cite{paper1}.
\end{proof}
\subsection{Pointwise decay}
\label{pointwidec1}
We use the energy decay estimates from Section \ref{endehigna} to obtain pointwise decay estimates for $\psi$. For convenience, we consider the coordinates $(\widetilde{\tau},\widetilde{\rho},\theta,\varphi)$ corresponding to the foliation by spacelike hypersurfaces $\mathcal{S}_{\widetilde{\tau}}$.

\begin{proposition}[Pointwise decay estimates for $\psi$]
\label{prop:pointdecpsi}
We decompose
\begin{equation*}
\psi_{1}=\psi_{\ell=1}+\psi_{\ell \geq 2}.
\end{equation*}
Let $k\in \N_0$ and assume that
\begin{equation*}
D(r)=1-\frac{2M}{r}+O_{3+k}(r^{-1-\beta})
\end{equation*}
and moreover that $E^{\epsilon}_{1;k+1}[\psi_{\ell=1}]+E^{\epsilon}_{2;k+1}[\psi_{\ell\geq 2}]<\infty$, for some $\epsilon\in (0,1)$.
\begin{itemize}
\item[$\mathbf{1)}$]
For $\psi_{\ell=1}$ we have the following estimates. There exists a constant $C \doteq C(D,R,k,\epsilon)$ such that the following estimates hold for all $\widetilde{\tau}\geq 0$:
\begin{align}
\label{eq:pointrtkpsi1b1}
|T^k\psi_{\ell=1}|(\widetilde{\tau},\widetilde{\rho},\theta,\varphi)\leq C (1+\widetilde{\tau})^{-\frac{7}{2}-k+\epsilon}\sqrt{E^{\epsilon}_{1;k+1}[\psi_{\ell=1}]} , \\
\label{eq:pointrtkpsi1b2}
\sqrt{\widetilde{\rho}+1}\cdot |T^k\psi_{\ell=1}|(\widetilde{\tau},\widetilde{\rho},\theta,\varphi)\leq C (1+\widetilde{\tau})^{-3-k+\epsilon}\sqrt{E^{\epsilon}_{1;k}[\psi_{\ell=1}]} , \\
\label{eq:pointrtkpsi1b3}
(\widetilde{\rho}+1)\cdot |T^k\psi_{\ell=1}|(\widetilde{\tau},\widetilde{\rho},\theta,\varphi)\leq C (1+\widetilde{\tau})^{-\frac{5}{2}-k+\epsilon}\sqrt{E^{\epsilon}_{1;k}[\psi_{\ell=1}]}.
\end{align}

\item[$\mathbf{2)}$]
Consider $\psi_{\ell\geq 2}$ and assume the following additional pointwise initial assumptions:
\begin{align*}
\lim_{r \to \infty }\sum_{|l|\leq 4+2k}\int_{\s^2}(\Omega^l\phi_{\ell\geq 2})^2\,d\omega\big|_{u'=0}<&\:\infty,\\
\lim_{r \to \infty }\sum_{|l|\leq 2+2k}\int_{\s^2}(r^2\partial_r(\Omega^l\phi_{\ell\geq 2}))^2\,d\omega\big|_{u'=0}<&\:\infty,\\
\lim_{r \to \infty }\sum_{|l|\leq 2k}\int_{\s^2}\left((r^2\partial_r)^2(\Omega^l\phi_{\ell\geq 2})\right)^2\,d\omega\big|_{u'=0}<&\:\infty,
\end{align*}
and
\begin{equation*}
\lim_{r \to \infty }\sum_{|l|\leq 2k-2s}\int_{\s^2}r^{2s+2}\Bigg(\partial_r^s\Big((r^2\partial_r)^2(\Omega^l\phi_{\ell\geq 2}\Big)\Bigg)^2\,d\omega\big|_{u'=0}<\infty,
\end{equation*}for each $1\leq s\leq k$.

Then there exists a constant $C=C(D,R,k,\epsilon)>0$ such that for all $\widetilde{\tau}\geq 0$:
\begin{align}
\label{eq:pointrtkpsi21}
|T^k\psi_{\ell\geq 2}|(\widetilde{\tau},\widetilde{\rho},\theta,\varphi)\leq C (1+\widetilde{\tau})^{-\frac{7}{2}-k+\epsilon}\sqrt{E^{\epsilon}_{2;k+1}[\psi_{\ell\geq 2}]} , \\
\label{eq:pointrtkpsi22}
\sqrt{\widetilde{\rho}+1}\cdot |T^k\psi_{\ell\geq 2}|(\widetilde{\tau},\widetilde{\rho},\theta,\varphi)\leq C (1+\widetilde{\tau})^{-3-k+\epsilon}\sqrt{E^{\epsilon}_{2;k}[\psi_{\ell\geq 2}]} , \\
\label{eq:pointrtkpsi23}
(\widetilde{\rho}+1) \cdot |T^k\psi_{\ell\geq 2}|(\widetilde{\tau},\widetilde{\rho},\theta,\varphi)\leq C (1+\widetilde{\tau})^{-\frac{5}{2}-k+\epsilon}\sqrt{E^{\epsilon}_{2;k}[\psi_{\ell\geq 2}]}.
\end{align}
\end{itemize}
\end{proposition}
\begin{proof}
We obtain pointwise decay estimates by repeating the proof of Theorem 1.9 and 1.10 of \cite{paper1}, but using the energy decay statements of Proposition \ref{decl2T}, rather than the energy decay statements in Theorem 1.8 of \cite{paper1}.
\end{proof}

\section{Almost-sharp decay estimates for $\psi_{0}=\frac{1}{4\pi}\int_{\mathbb{S}^{2}}\psi$ I: The case $I_{0}\neq 0$ }
\label{gdeforpsi1}

Before we obtain precise late-time asymptotics for the spherical mean $\int_{\s^2}\psi\,d\omega$ in the case $I_0[\psi]\neq 0$ (Section \ref{sec:asymradfieldcaseI}) and the resulting sharp upper bound and lower bound decay estimates, we will first state several \emph{almost}-sharp upper bound decay estimates (i.e.\ sharp decay rates with an $\epsilon$ loss).

\subsection{Energy decay}
\label{endehigna1}

\begin{proposition}[{Energy decay for $T^k\psi_0$}]\label{prop:endec1T}
\hspace{1pt}\\
Let $k\in \N_0$ and assume that
\begin{equation*}
D(r)=1-\frac{2M}{r}+O_{3+k}(r^{-1-\beta}).
\end{equation*}

Assume moreover that $E_{0, I_0 \neq 0;k}^{\epsilon}[\psi_0]<\infty$, for some $\epsilon\in (0,1)$. Then there exists a constant $C \doteq C(D,R,\epsilon,k)$ such that for all $\tau\geq 0$
\begin{equation}\label{est:endec1T}
\int_{\Sigma_{\tau}} J^N [T^k\psi_0 ] \cdot n_{\tau} d\mu_{\Sigma_{\tau}} \leq C(1+\tau)^{-3-2k+\epsilon} E_{0, I_0 \neq 0;k}^{\epsilon}[\psi_0].
\end{equation}
\end{proposition}
\begin{proof}
A proof is contained in Theorem 1.8 of \cite{paper1}.
\end{proof}

\subsection{Pointwise decay}
\label{pointwidec11}

\begin{proposition}[Pointwise decay estimates for $\psi$]
\label{prop:pointdecpsi}
Let $k\in \N_0$ and assume that
\begin{equation*}
D(r)=1-\frac{2M}{r}+O_{3+k}(r^{-1-\beta}).
\end{equation*}

Assume moreover that $E_{0, I_0 \neq 0;k+1}^{\epsilon}[\psi_0]<\infty$, for some $\epsilon\in (0,1)$. Then there exists a constant $C \doteq C(D,R,\epsilon,k)$ such that for all $\tau\geq 0$
\begin{align}
\label{eq:pointrtkpsi01}
|T^k\psi_0|(\widetilde{\tau},\widetilde{\rho},\theta,\varphi)\leq C (1+\widetilde{\tau})^{-2-k+\epsilon}\sqrt{E^{\epsilon}_{0,I_0\neq0;k+1}[\psi_0]},\\
\label{eq:pointrtkpsi03}
\sqrt{\widetilde{\rho}+1}\cdot |T^k\psi_0|(\widetilde{\tau},\widetilde{\rho},\theta,\varphi)\leq C (1+\widetilde{\tau})^{-\frac{3}{2}-k+\epsilon}\sqrt{E^{\epsilon}_{0,I_0\neq0;k}[\psi_0]},\\
\label{eq:pointrtkpsi05}
\widetilde{\rho}\cdot |T^k\psi_0|(\widetilde{\tau},\widetilde{\rho},\theta,\varphi)\leq C (1+\widetilde{\tau})^{-1-k+\epsilon}\sqrt{E^{\epsilon}_{0,I_0\neq0;k}[\psi_0]}.
\end{align}

\end{proposition}
\begin{proof}
A proof is contained in the proofs of Theorem 1.9 and 1.10 of \cite{paper1}.
\end{proof}

\section{Almost-sharp decay estimates for $\psi_{0}=\frac{1}{4\pi}\int_{\mathbb{S}^{2}}\psi$ II: The case $I_{0}=0$ }
\label{gdeforpsi2}

Before we obtain precise late-time asymptotics for the spherical mean $\int_{\s^2}\psi\,d\omega$ in the case $I_0[\psi]=0$ (Section \ref{sec:asympsi2}) and the resulting sharp upper bound and lower bound decay estimates, we will first state several almost-sharp upper bound decay estimates.

\subsection{Energy decay}
\label{endehigna2}

\begin{proposition}[{Energy decay for $T^k\psi_0$}]\label{prop:endec1TI00}
\hspace{1pt}\\
Let $k\in \N_0$ and assume that
\begin{equation*}
D(r)=1-\frac{2M}{r}+O_{3+k}(r^{-1-\beta}).
\end{equation*}

Assume moreover that $E_{0, I_0 = 0;k}^{\epsilon}[\psi_0]<\infty$, for some $\epsilon\in (0,1)$. Then there exists a constant $C \doteq C(D,R,\epsilon,k)$ such that for all $u\geq 0$
\begin{equation}\label{est:endec2}
\int_{\Sigma_{\tau}} J^N [T^k\psi_0 ] \cdot n_{\tau} d\mu_{\Sigma_{\tau}} \leq C(1+\tau)^{-5-2k+\epsilon} E_{0, I_0=0;k}^{\epsilon}[\psi_0].
\end{equation}
\end{proposition}
\begin{proof}
A proof is contained in Theorem 1.8 of \cite{paper1}.
\end{proof}

\subsection{Pointwise decay}
\label{pointwidec12}

\begin{proposition}[Pointwise decay estimates for $\psi$]
\label{prop:pointdecpsiI00}
Let $k\in \N_0$ and assume that
\begin{equation*}
D(r)=1-\frac{2M}{r}+O_{3+k}(r^{-1-\beta}).
\end{equation*}

Assume moreover that $E_{0, I_0 = 0;k+1}^{\epsilon}[\psi_0]<\infty$, for some $\epsilon\in (0,1)$. Then there exists a constant $C \doteq C(D,R,k,\epsilon)>0$ such that the following estimates hold:
\begin{align}
\label{eq:pointrtkpsi02}
|T^k\psi_{0}|(\widetilde{\tau},\widetilde{\rho},\theta,\varphi)\leq C (1+\widetilde{\tau})^{-3-k+\epsilon}\sqrt{E^{\epsilon}_{0,I_0=0;k+1}[\psi_0]} , \\
\label{eq:pointrtkpsi04}
\sqrt{\widetilde{\rho}+1}\cdot |T^k\psi_0|(\widetilde{\tau},\widetilde{\rho},\theta,\varphi)\leq C (1+\widetilde{\tau})^{-\frac{5}{2}-k+\epsilon}\sqrt{E^{\epsilon}_{0,I_0=0;k}[\psi_0]} , \\
\label{eq:pointrtkpsi06}
(\widetilde{\rho}+1)\cdot |T^k\psi_0|(\widetilde{\tau},\widetilde{\rho},\theta,\varphi)\leq C (1+\widetilde{\tau})^{-2-k+\epsilon}\sqrt{E^{\epsilon}_{0,I_0=0;k}[\psi_0]}.
\end{align}
\end{proposition}
\begin{proof}
A proof is contained in the proofs of Theorem 1.9 and 1.10 of \cite{paper1}.
\end{proof}

\section{Almost-sharp decay for the hyperboloidal derivative $Y\psi$}
\label{improvedhyperboloidaldecay}
In order to obtain lower bounds for $\psi$ and moreover the precise leading-order behaviour of $\psi_0$ in time (see Section \ref{sec:asymradfieldcaseI} and \ref{sec:asympsi2}), we will also need to establish improved upper bound estimates for the derivative $Y\psi_0$, where $Y$ is the spherically symmetric vector field that is tangent to the hyperboloids $\mathcal{S}_{\widetilde{\tau}}$. These improved decay estimates are obtained by applying the elliptic estimate in Lemma \ref{lm:elliptic} below, together with red-shift estimates in the $r_{\rm min}=r_+$ case (Lemma \ref{lm:redshift} and Lemma \ref{lm:redshiftNpsi}).

\subsection{An elliptic estimate}
\label{ellipticestimate}
The following lemma is proved in Section 7.5 of \cite{paper1}.
\begin{lemma}[\textbf{A degenerate elliptic estimate for $\psi$}]
\label{lm:elliptic}
Let $\psi$ be a solution to \eqref{waveequation} on a spacetime $(\mathcal{M}, g)$. Assume moreover that with respect to $(\widetilde{\tau},\widetilde{\rho},\theta,\varphi)$ coordinates on $\mathcal{S}$:
\begin{align*}
\lim_{\rho\to \infty}\sqrt{\widetilde{\rho}}\cdot T\psi(0,\widetilde{\rho},\theta,\varphi)=&0,\\
\lim_{\rho\to \infty}\sqrt{\widetilde{\rho}}\cdot Y\psi(0,\widetilde{\rho},\theta,\varphi)=&0.
\end{align*}

Then we can estimate with respect to $(\widetilde{\rho},\theta,\varphi)$ coordinates:
\begin{equation}
\label{eq:ellipticpsi}
\begin{split}
\int_{r_{\rm min}}^{\infty}&\int_{\s^2}r^{-2}(Y(Dr^2Y\psi))^2+ D^2r^2(Y^2 \psi)^2+Dr^2|\snabla Y\psi|^2+r^2(\slashed{\Delta}\psi)^2\,d\omega d\widetilde{\rho}\\
\leq&\: C(D)\int_{r_{\rm min}}^{\infty}\int_{\s^2}r^2(YT \psi)^2+h(r)^2r^2(T^2\psi)^2\,d\omega  d\widetilde{\rho}.
\end{split}
\end{equation}
\end{lemma}

\subsection{Red-shift estimates}
\label{commutedredshift}

We make use of a standard red-shift energy estimate in the $r_{\rm min}=r_+$ case. 

Recall from Section \ref{sec:asmebound} that the vector field $N$ satisfies the following properties in the $r_{\rm min}=r_+$ case:
\begin{align*}
N=&\:T-Y\quad \textnormal{in}\quad \{r_+\leq r\leq r_0\},\\
N=&\:T \quad \textnormal{in}\quad \{r\geq r_1\},
\end{align*}
where $r_+<r_0<r_1$. We use the vector field $N$ to obtain the following energy estimate, a proof of which can be found in Lemma \ref{lm:redshift} of \cite{lecturesMD} (in a very general setting). 
\begin{lemma}
\label{lm:redshift}
Consider the case $r_{\min}=r_+$. Let $\psi$ be a solution to \eqref{waveequation}. Let $r_+<r_0<r_1$, such that $|r_1-r_+|$ is suitably small. Then we can estimate for all $0\leq \widetilde{\tau}_1< \widetilde{\tau}_2$:
\begin{equation}
\label{eq:redshiftpsi}
\begin{split}
\int_{\mathcal{S}_{\widetilde{\tau}_2}}&J^N[\psi]\cdot n_{\widetilde{\tau}_2}\,d\mu_{\mathcal{S}_{\widetilde{\tau}_2}}+b\int_{\widetilde{\tau}_1}^{\widetilde{\tau}_2}\left[\int_{\mathcal{S}_{\widetilde{\tau}}}J^N[\psi]\cdot n_{\widetilde{\tau}}\,d\mu_{\mathcal{S}_{\widetilde{\tau}}}\right]\,d\tau\leq \int_{\mathcal{S}_{\widetilde{\tau}_1}}J^N[\psi]\cdot n_{\widetilde{\tau}_1}\,d\mu_{\mathcal{S}_{\widetilde{\tau}_1}}\\
&+C\int_{\widetilde{\tau}_1}^{\widetilde{\tau}_2}\left[\int_{\mathcal{S}_{\widetilde{\tau}}}J^T[\psi]\cdot n_{\widetilde{\tau}}\,d\mu_{\mathcal{S}_{\widetilde{\tau}}}\right]\,d\widetilde{\tau},
\end{split}
\end{equation}
where $C=C(D,\mathcal{S})>0$ and $b=b(D,\mathcal{S})>0$ are constants.
\end{lemma}

We will also need a standard red-shift energy estimate for the quantity $N\psi$. A proof of Lemma \ref{lm:redshiftNpsi} can also be found in \cite{lecturesMD}.

\begin{lemma}
\label{lm:redshiftNpsi}
Consider the case $r_{\min}=r_+$. Let $\psi$ be a solution to \eqref{waveequation}. Let $r_+<r_0<r_1$, such that $|r_1-r_+|$ is suitably small. Then we can estimate for all $0\leq \widetilde{\tau}_1< \widetilde{\tau}_2$:
\begin{equation}
\label{eq:redshiftNpsi}
\begin{split}
\int_{\mathcal{S}_{\widetilde{\tau}_2}}&J^N[N\psi]\cdot n_{\widetilde{\tau}_2}\,d\mu_{\mathcal{S}_{\widetilde{\tau}_2}}+b\int_{\widetilde{\tau}_1}^{\widetilde{\tau}_2}\left[\int_{\mathcal{S}_{\widetilde{\tau}}}J^N[N\psi]\cdot n_{\widetilde{\tau}}\,d\mu_{\mathcal{S}_{\widetilde{\tau}}}\right]\,d\widetilde{\tau}\leq \int_{\mathcal{S}_{\widetilde{\tau}_1}}J^N[N\psi]\cdot n_{\widetilde{\tau}_1}\,d\mu_{\mathcal{S}_{\widetilde{\tau}_1}}\\
&+C\int_{\widetilde{\tau}_1}^{\widetilde{\tau}_2}\left[\int_{\mathcal{S}_{\widetilde{\tau}}\cap\{r_0\leq r\leq r_1\}}J^T[N\psi]\cdot n_{\widetilde{\tau}}\,d\mu_{\mathcal{S}_{\widetilde{\tau}}}\right]\,d\widetilde{\tau}\\
&+C\int_{\widetilde{\tau}_1}^{\widetilde{\tau}_2}\left[\int_{\mathcal{S}_{\widetilde{\tau}}}J^T[T\psi]\cdot n_{\widetilde{\tau}}\,d\mu_{\mathcal{S}_{\widetilde{\tau}}}\right]\,d\widetilde{\tau}\\
&+C\int_{\widetilde{\tau}_1}^{\widetilde{\tau}_2}\left[\int_{\mathcal{S}_{\widetilde{\tau}}\cap\{r \leq r_1\}}J^T[\psi]\cdot n_{\widetilde{\tau}}\,d\mu_{\mathcal{S}_{\widetilde{\tau}}}\right]\,d\widetilde{\tau},
\end{split}
\end{equation}
where $C=C(D,\mathcal{S},r_0,r_1)>0$ and $b=b(D,\mathcal{S},r_0,r_1)>0$ are constants.
\end{lemma}

\subsection{Energy decay for $Y\psi$}
\label{energyweightedpoindecay}
In this section we will assume that $\psi$ is a spherically symmetric solution to \eqref{waveequation} emanating from initial data given as in Theorem \ref{thm:extuniq} on a spacetime $(\mathcal{M}, g)$ that satisfies the geometric assumptions from Section \ref{sec:geomassm}.

We will make use of a Gr\"onwall-type lemma to \emph{combine} Lemma \ref{lm:elliptic} with Lemma \ref{lm:redshift} and Lemma \ref{lm:redshiftNpsi} and also the energy decay estimates in Section \ref{gdeforpsi1} and \ref{gdeforpsi2}.

\begin{lemma}
\label{lm:gronwall}
Let $f: [0,\infty)\to \R$ be a continuous, positive function. Assume that for all $0\leq t_1\leq t_2<\infty$,
\begin{equation}
\label{eq:intineqf}
f(t_2)+b \int_{t_1}^{t_2}f(s)\,ds\leq f(t_1)+E_0(t_2-t_1)(t_1+1)^{-p},
\end{equation}
with $E_0,b,p>0$ constants and moreover, for all $0\leq t_0\leq t_1\leq t_2$
\begin{equation}
\label{eq:intineqf2}
f(t_2)+b \int_{t_1}^{t_2}f(s)\,ds\leq f(t_1)+C_0 (t_2-t_1)f(t_0),
\end{equation}
with $C_0,b,p>0$ constants. Then,
\begin{equation}
\label{eqboundf}
f(t)\leq \left(1+C_0b^{-1}\right)f(t_0)
\end{equation}
for all $t\geq t_0$ and there exists a constant $C=C(C_0,E_0,b,p)>0$, such that
\begin{equation}
\label{eqdecayf}
f(t)\leq C(f(0)+E_0)(1+t)^{-p},
\end{equation}
for all $t\geq 0$.
\end{lemma}
\begin{proof}
We will first prove \eqref{eqboundf} using only \eqref{eq:intineqf2}.
We define
\[h(t):= \frac{1}{C_0}\cdot\frac{f(t)}{f(t_0)}.\]
Then
\[h(t_0)=\frac{1}{C_0}\]
and 
\begin{equation}
h(t_2)+b\int_{t_1}^{t_2}h(t)dt\leq h(t_1)+(t_2-t_1).
\label{eq:}
\end{equation}
It suffices to show
\[h(t)\leq (C_0^{-1}+b^{-1}).\]
We argue by contradiction. Suppose that there is $t^*>t_0$ such that 
\[h(t^*)=(C_0^{-1}+b^{-1})+y\]
for some $y>0$. 
Consider the smallest number $t_0+z$ greater than $t_0$ having the property
 \[h(t_0+z)=C_0^{-1}+b^{-1}.\]
Define
\[\epsilon:=\frac{1}{2}\min\left\{ y,z \right\}>0. \]
and consider the smallest number $t_0+w$ greater than $t_0$ with the property
\[h(t_0+w)=(C_0^{-1}+b^{-1})+\epsilon.\]
Note that $t_0<t_0+w-\epsilon$ since $\epsilon<\frac{z}{2}<z<w$. 
We next use \eqref{eq:} to obtain
\[h(t)\geq h(t_0+w)-(t_0+w-t)=(C_0^{-1}+b^{-1})+\epsilon-(t_0+w-t).  \]
for all $t\in [t_0+w-\epsilon,t_0+w]$
The above lower bound for $h$ gives the following bound for the integral
\begin{equation*}
\begin{split}
b\int_{t_0+w-\epsilon}^{t_0+w}h(t)dt\geq & b\int_{t_0+w-\epsilon}^{t_0+w}\Big((C_0^{-1}+b^{-1})+\epsilon-(t_0+w-t)\Big)dt\\ 
=&b\int_{t_0+w-\epsilon}^{t_0+w}\Big((C_0^{-1}+b^{-1})+ t-(t_0+w-\epsilon)\Big)dt\\ 
=& b\epsilon(C_0^{-1}+b^{-1})+b\int_{0}^{\epsilon}sds\\
=&b\epsilon(C_0^{-1}+b^{-1})+b\frac{\epsilon^{2}}{2}.
\end{split}
\end{equation*}
Applying \eqref{eq:} for $t_1=t_0+w-\epsilon, t_2=t_0+w$ and using that $t_2-t_1=\epsilon$, the above estimate for the integral and that $h(t_1)\leq h(t_2)$ (by the definition of $t_2=t_0+w$) we obtain:
\[b\epsilon(C_0^{-1}+b^{-1})+b\frac{\epsilon^{2}}{2}\leq \epsilon   \]
or equivalently
\[b\epsilon C_0^{-1}+\epsilon+b\frac{\epsilon^{2}}{2}\leq \epsilon   \]
which is contradiction since $\epsilon>0$. 

We will use a continuity argument to prove (\ref{eqdecayf}). First of all, fix $T\geq 0$, then
\begin{equation*}
f(t)\leq f(0)+E_0t\leq (f(0)(T+1)^p+E_0(T+1)^{p+1})(T+1)^{-p}
\end{equation*}
for any $t\in [0,T]$. We will make the following bootstrap assumption:
\begin{equation}
\label{eq:ba}
f(t)\leq A(t+1)^{-p},
\end{equation}
for all $t\geq 0$, where $A>0$ will be chosen suitably large. In particular, by choosing $A>2(f(0)(T+1)^p+E_0(T+1)^{p+1})$, we can improve (\ref{eq:ba}) for $t\leq T$. We are left with improving (\ref{eq:ba}) on the interval $[T,\infty)$. 

We consider a dyadic sequence $\{t_i\}_{i\in \N_0}$, with $t_i=2^iT$, i.e.\ $t_{i+i}=2t_i$, $t_{i+1}-t_i=t_i$. Then, by (\ref{eq:intineqf}) together with \eqref{eq:ba}:
\begin{equation*}
\int_{t_i}^{t_{i+1}}f(s)\,ds\leq \frac{A}{b}(t_{i}+1)^{-p}+\frac{E_0}{b}(t_{i+1}-t_i)(t_i+1)^{-p} \leq \frac{2^pA}{b}t_{i+1}^{-p}+\frac{2^pE_0}{b}(t_{i+1}-t_i)t_{i+1}^{-p}.
\end{equation*}
By the mean-value theorem together with the dyadicity of $t_i$, there exists $(t_*)_i\in [t_i,t_{i+1}]$ such that
\begin{equation*}
f((t_*)_i)\leq 2^{-i}T^{-1}\frac{2^pA}{b}t_{i+1}^{-p}+\frac{2^pE_0}{b}t_{i+1}^{-p}\leq 2^{-i}T^{-1}\frac{2^pA}{b}(t_*)_i^{-p}+\frac{2^pE_0}{b}(t_*)_i^{-p}.
\end{equation*}
For any $\epsilon>0$ there exists a suitably large $T=T(\epsilon,b,p)>0$, such that
\begin{equation*}
f((t_*)_i)\leq \epsilon A(t_*)_i^{-p}+\frac{2^pE_0}{b}(t_*)_i^{-p}.
\end{equation*}
Denote $\tau_i:= (t_*)_{2i+1}$ Note that  $\{\tau_i\}_{i\in \N_0}$ is also a dyadic sequence with $2\leq \frac{\tau_{i+1}}{\tau_i}\leq 2^3=8$. Now, let $t\in [\tau_i,\tau_{i+1}]$. Then, by the above estimate for $f((t_*)_i)$, together with \eqref{eqboundf} with $t_0$ replaced by $\tau_{i}$, we find that:
\begin{equation*}
\begin{split}
f(t)\leq (C_0b^{-1}+1)f(\tau_{i})\leq&\: \epsilon (C_0b^{-1}+1) A\tau_i^{-p}+\frac{2^pE_0}{b} (C_0b^{-1}+1)\tau_i^{-p}\\
\leq&\: 8^p(C_0b^{-1}+1)\epsilon At^{-p}+\frac{2^{(1+3)p}E_0}{b}(C_0b^{-1}+1)t^{-p}\\
\leq&\: 8^p(C_0b^{-1}+1)\epsilon A\left(\frac{1}{2}t+\frac{1}{2}T\right)^{-p}\\
&+\frac{2^{(1+3)p}E_0}{b}(C_0b^{-1}+1)\left(\frac{1}{2}t+\frac{1}{2}T\right)^{-p}\\
\leq&\: 2^{4p}(C_0b^{-1}+1)\epsilon A\left(t+T\right)^{-p}\\
&+\frac{2^{5p}E_0}{b}(C_0b^{-1}+1)\left(t+T\right)^{-p}
\end{split}
\end{equation*}
By choosing $T\geq 1$ and $\epsilon=\frac{1}{2}\cdot2^{-4p}(C_0b^{-1}+1)^{-1}$ and $A>4\cdot\frac{2^{5p}E_0}{b}(C_0b^{-1}+1)$, we therefore obtain
\begin{equation*}
f(t)\leq \frac{3}{4}A(t+1)^{-p},
\end{equation*}
for all $t\geq T$. This improves \eqref{eq:ba} also for $t\geq T$.
\end{proof}

\begin{proposition}
\label{prop:edecayNpsi}
Let $k\in \N_0$ and fix $\epsilon\in(0,1)$. 

Assume that $E^{\epsilon}_{0,I_0\neq 0;k+1}[\psi]<\infty$ and $$\sum_{j=0}^k\int_{\Sigma_0}J^N[NT^j\psi]\cdot n_{\widetilde{\tau}}<\infty,$$ then there exists a constant $C=C(D,R,k,\epsilon)>0$ such that for all $\widetilde{\tau}\geq 0$:
\begin{align}
\label{eq:edecayNpsi1}
\int_{\mathcal{S}_{\widetilde{\tau}}}J^N[NT^k\psi]\cdot n_{\widetilde{\tau}}\,d\mu_{\mathcal{S}_{\widetilde{\tau}}}\leq&\: C\widetilde{E}^{\epsilon}_{0,I_0\neq 0;k+1}[\psi](1+\widetilde{\tau})^{-5-2k+\epsilon},\\
\label{eq:edecayYpsi1}
\int_{\mathcal{S}_{\widetilde{\tau}}}J^N[YT^k\psi]\cdot n_{\widetilde{\tau}}\,d\mu_{\mathcal{S}_{\widetilde{\tau}}}\leq&\: C\widetilde{E}^{\epsilon}_{0,I_0\neq 0;k+1}[\psi](1+\widetilde{\tau})^{-5-2k+\epsilon},
\end{align}
with
\begin{equation*}
\widetilde{E}^{\epsilon}_{0,I_0\neq 0;k+1}[\psi]=E^{\epsilon}_{0,I_0\neq 0;k+1}[\psi]+\sum_{j=0}^k\int_{\Sigma_0}J^N[NT^j\psi]\cdot n_0\,d\mu_0.
\end{equation*}
If we moreover assume that $E^{\epsilon}_{0,I_0= 0;k+1}[\psi]<\infty$, then there exists a constant $C=C(D,R,k,\epsilon)>0$ such that for all $\widetilde{\tau}\geq 0$
\begin{align}
\label{eq:edecayNpsi2}
\int_{\mathcal{S}_{\widetilde{\tau}}}J^N[NT^k\psi]\cdot n_{\widetilde{\tau}}\,d\mu_{\mathcal{S}_{\widetilde{\tau}}}\leq&\: C\widetilde{E}^{\epsilon}_{0,I_0= 0;k+1}[\psi](1+\widetilde{\tau})^{-7-2k+\epsilon},\\
\label{eq:edecayYpsi2}
\int_{\mathcal{S}_{\widetilde{\tau}}}J^N[YT^k\psi]\cdot n_{\widetilde{\tau}}\,d\mu_{\mathcal{S}_{\widetilde{\tau}}}\leq&\: C\widetilde{E}^{\epsilon}_{0,I_0= 0;k+1}[\psi](1+\widetilde{\tau})^{-7-2k+\epsilon},
\end{align}
with
\begin{equation*}
\widetilde{E}^{\epsilon}_{0,I_0=0;k+1}[\psi]=E^{\epsilon}_{0,I_0= 0;k+1}[\psi]+\sum_{j=0}^k\int_{\Sigma_0}J^N[NT^j\psi]\cdot n_0\,d\mu_0.
\end{equation*}
\end{proposition}
\begin{proof}
Without loss of generality we consider the case that $k=0$. The case $k>0$ can be proved identically using the almost-sharp energy decay estimates for $T^k\psi$. 

If $r_{\rm min}=0$, the statements of the proposition follow immediately from the elliptic estimate \eqref{eq:ellipticpsi} together with Proposition \ref{prop:endec1T} and \ref{prop:endec1TI00}. In the remainder of the proof we will therefore restrict to the $r_{\rm min}=r_+$ case.

Consider the inequality \eqref{eq:redshiftNpsi}. We estimate the second term on the right-hand side by applying the elliptic estimate \eqref{eq:ellipticpsi}:
\begin{equation*}
\int_{\widetilde{\tau}_1}^{\widetilde{\tau}_2}\left[\int_{\mathcal{S}_{\widetilde{\tau}}\cap\{r_0\leq r\leq r_1\}}J^T[N\psi]\cdot n_{\widetilde{\tau}}\,d\mu_{\mathcal{S}_{\widetilde{\tau}}}\right]\,d\widetilde{\tau}\leq C\int_{\widetilde{\tau}_1}^{\widetilde{\tau}_2}\left[\int_{\mathcal{S}_{\widetilde{\tau}}}J^T[T\psi]\cdot n_{\widetilde{\tau}}\,d\mu_{\mathcal{S}_{\widetilde{\tau}}}\right]\,d\widetilde{\tau}.
\end{equation*}

We estimate the third term on the right-hand side of \eqref{eq:redshiftNpsi} by applying the following Hardy inequality:
\begin{equation*}
\int_{r_+}^{\infty} f^2(\rho)\,d\rho\leq C \int_{r_+}^{\infty} (\rho-r_+)^2\cdot (\partial_{\rho}f)^2(\rho)\,d\rho\leq C \int_{r_+}^{\infty} \rho^2 D^2(\rho)\cdot (\partial_{\rho}f)^2(\rho)\,d\rho,
\end{equation*}
with $\lim_{\rho\to \infty}\rho f^2(\rho)=0$, see for example Lemma 2.2 of \cite{paper1} for a derivation. We obtain
\begin{equation*}
\begin{split}
\int_{\widetilde{\tau}_1}^{\widetilde{\tau}_2}\left[\int_{\mathcal{S}_{\widetilde{\tau}}\cap\{r \leq r_1\}}J^T[\psi]\cdot n_{\widetilde{\tau}}\,d\mu_{\mathcal{S}_{\widetilde{\tau}}}\right]\,d\widetilde{\tau}\leq&\: C\int_{\widetilde{\tau}_1}^{\widetilde{\tau}_2}\int_{r_+}^{\infty} \int_{\s^2}(T\psi)^2+(Y\psi)^2\,d\omega d\widetilde{\rho} d\widetilde{\tau}\\
\leq&\: C\int_{\widetilde{\tau}_1}^{\widetilde{\tau}_2}\int_{r_+}^{\infty} \int_{\s^2}D^2r^2(YT\psi)^2+D^2r^2(Y^2\psi)^2\,d\omega d\widetilde{\rho} d\widetilde{\tau}.
\end{split}
\end{equation*}
The right-hand side of the above inequality can then be estimated by applying once more \eqref{eq:ellipticpsi}. We are then left with the following inequality:
\begin{equation}
\label{eq:redshiftNpsimod}
\begin{split}
\int_{\mathcal{S}_{\widetilde{\tau}_2}}&J^N[N\psi]\cdot n_{\widetilde{\tau}_2}\,d\mu_{\mathcal{S}_{\widetilde{\tau}_2}}+b\int_{\widetilde{\tau}_1}^{\widetilde{\tau}_2}\left[\int_{\mathcal{S}_{\widetilde{\tau}}}J^N[N\psi]\cdot n_{\widetilde{\tau}}\,d\mu_{\mathcal{S}_{\widetilde{\tau}}}\right]\,d\widetilde{\tau}\leq \int_{\mathcal{S}_{\widetilde{\tau}_1}}J^N[N\psi]\cdot n_{\widetilde{\tau}_1}\,d\mu_{\mathcal{S}_{\widetilde{\tau}_1}}\\
&+C\int_{\widetilde{\tau}_1}^{\widetilde{\tau}_2}\left[\int_{\mathcal{S}_{\widetilde{\tau}}}J^T[T\psi]\cdot n_{\widetilde{\tau}}\,d\mu_{\mathcal{S}_{\widetilde{\tau}}}\right]\,d\widetilde{\tau}.
\end{split}
\end{equation}

We can freely add $J^T[T\psi]$ terms to the integrals on the left-hand side in the above inequality, using the conservation property of the $T$-energy, to obtain
\begin{equation}
\label{eq:redshiftNpsimod}
\begin{split}
\int_{\mathcal{S}_{\widetilde{\tau}_2}}&\left(J^N[N\psi]+J^T[T\psi]\right)\cdot n_{\widetilde{\tau}_2}\,d\mu_{\mathcal{S}_{\widetilde{\tau}_2}}+b\int_{\widetilde{\tau}_1}^{\widetilde{\tau}_2}\left[\int_{\mathcal{S}_{\widetilde{\tau}}}\left(J^N[N\psi]+J^T[\psi]\right)\cdot n_{\widetilde{\tau}}\,d\mu_{\mathcal{S}_{\widetilde{\tau}}}\right]\,d\widetilde{\tau}\\
&\leq \int_{\mathcal{S}_{\widetilde{\tau}_1}}\left(J^N[N\psi]+J^T[\psi]\right)\cdot n_{\widetilde{\tau}_1}\,d\mu_{\mathcal{S}_{\widetilde{\tau}_1}}+C\int_{\widetilde{\tau}_1}^{\widetilde{\tau}_2}\left[\int_{\mathcal{S}_{\widetilde{\tau}}}J^T[T\psi]\cdot n_{\widetilde{\tau}}\,d\mu_{\mathcal{S}_{\widetilde{\tau}}}\right]\,d\widetilde{\tau}.
\end{split}
\end{equation}
Note that, for every $0\leq \widetilde{\tau}_0\leq \widetilde{\tau}_1\leq \widetilde{\tau}_2$:
\begin{equation*}
\begin{split}
\int_{\widetilde{\tau}_1}^{\widetilde{\tau}_2}\left[\int_{\mathcal{S}_{\widetilde{\tau}}}J^T[T\psi]\cdot n_{\widetilde{\tau}}\,d\mu_{\mathcal{S}_{\widetilde{\tau}}}\right]\,d\widetilde{\tau}\leq&\: (\widetilde{\tau}_2-\widetilde{\tau}_1) \int_{\mathcal{S}_{\widetilde{\tau}_0}}J^T[T\psi]\cdot n_{\widetilde{\tau}_0}\,d\mu_{\mathcal{S}_{\widetilde{\tau}_0}}\\
\leq&\: (\widetilde{\tau}_2-\widetilde{\tau}_1) \int_{\mathcal{S}_{\widetilde{\tau}_0}}(J^T[T\psi]+J^N[N\psi])\cdot n_{\widetilde{\tau}_0}\,d\mu_{\mathcal{S}_{\widetilde{\tau}_0}}.
\end{split}
\end{equation*}

Furthermore, by Proposition \ref{prop:endec1T}, we can alternatively estimate
\begin{equation*}
\int_{\widetilde{\tau}_1}^{\widetilde{\tau}_2}\left[\int_{\mathcal{S}_{\widetilde{\tau}}}J^T[T\psi]\cdot n_{\widetilde{\tau}}\,d\mu_{\mathcal{S}_{\widetilde{\tau}}}\right]\,d\widetilde{\tau}\leq CE^{\epsilon}_{0,I_0\neq 0;1}[\psi] (1+\widetilde{\tau}_1)^{-5+\epsilon}(\widetilde{\tau}_2-\widetilde{\tau}_1)
\end{equation*}
in the case $I_0[\psi]\neq 0$ and by Proposition \ref{prop:endec1TI00}, we can estimate
\begin{equation*}
\int_{\widetilde{\tau}_1}^{\widetilde{\tau}_2}\left[\int_{\mathcal{S}_{\widetilde{\tau}}}J^T[T\psi]\cdot n_{\widetilde{\tau}}\,d\mu_{\mathcal{S}_{\widetilde{\tau}}}\right]\,d\widetilde{\tau}\leq CE^{\epsilon}_{0,I_0= 0;1}[\psi] (1+\widetilde{\tau}_1)^{-7+\epsilon}(\widetilde{\tau}_2-\widetilde{\tau}_1)
\end{equation*}
in the case $I_0[\psi]=0$.

The estimates \eqref{eq:edecayNpsi1} and \eqref{eq:edecayNpsi2} now follow from an application of Lemma \ref{lm:gronwall} to the inequality \eqref{eq:redshiftNpsimod}.

Finally, \eqref{eq:edecayYpsi1} and \eqref{eq:edecayYpsi2} follow from the elliptic estimate \eqref{eq:ellipticpsi} applied in the region $\{r\geq r_0\}$ and from the already established decay estimates \eqref{eq:edecayNpsi1} and \eqref{eq:edecayNpsi2} for $N\psi$ in the region $\{r\leq r_0\}$, together with the decay estimates for $T\psi$ from Proposition \ref{prop:endec1T} and \ref{prop:endec1TI00}.
\end{proof}

\subsection{Global pointwise decay for $Y\psi$}
\label{globalpointwiseypsi}

\begin{corollary}
\label{cor:rhalfYpsidecay}
Let $k\in \N_0$ and fix $\epsilon\in (0,1)$.  Assume that
\begin{equation*}
D(r)=1-\frac{2M}{r}+O_{3+k}(r^{-1-\beta}).
\end{equation*}
If we also assume that $\widetilde{E}^{\epsilon}_{0,I_0\neq 0;k+1}[\psi]<\infty$, then there exists a constant $C=C(D,R,k,\epsilon)>0$ such that for all $\widetilde{\tau}\geq 0$
\begin{equation}
\label{eq:pointdecayNYpsi1}
\sqrt{\widetilde{\rho}+1}\cdot |NT^k\psi|(\widetilde{\tau},\widetilde{\rho})+\sqrt{\widetilde{\rho}+1}\cdot|YT^k\psi|(\widetilde{\tau},\widetilde{\rho})\leq C\sqrt{\widetilde{E}^{\epsilon}_{0,I_0\neq 0;k+1}[\psi]}(1+\widetilde{\tau})^{-5/2-k+\epsilon},\\
\end{equation}
If we moreover assume that $\widetilde{E}^{\epsilon}_{0,I_0= 0;k+1}[\psi]<\infty$, then there exists a constant $C=C(D,R,k,\epsilon)>0$ such that for all $\widetilde{\tau}\geq 0$:
\begin{equation}
\label{eq:pointdecayNYpsi2}
\sqrt{\widetilde{\rho}+1}\cdot |NT^k\psi|(\widetilde{\tau},\widetilde{\rho})+\sqrt{\widetilde{\rho}+1}\cdot |YT^k\psi|(\widetilde{\tau},\widetilde{\rho})\leq C\sqrt{\widetilde{E}^{\epsilon}_{0,I_0= 0;k+1}[\psi]}(1+\widetilde{\tau})^{-7/2-k+\epsilon}.
\end{equation}
\end{corollary}
\begin{proof}
The estimates follow from applying the fundamental theorem of calculus in $\widetilde{\rho}$, together with Cauchy--Schwarz, and using the estimates in Proposition \ref{prop:edecayNpsi}.
\end{proof}

\begin{proposition}
\label{prop:decayNpsi}
Let $k\in \N_0$ and fix $\epsilon\in( 0,1)$. Assume that
\begin{equation*}
D(r)=1-\frac{2M}{r}+O_{3+k}(r^{-1-\beta}).
\end{equation*}
If we also assume that $\widetilde{E}^{\epsilon}_{0,I_0\neq 0;k+2}[\psi]<\infty$, then there exists a constant $C=C(D,R,k,\epsilon)>0$ such that for all $\widetilde{\tau}\geq 0$
\begin{equation}
\label{eq:pointdecayNpsi1}
||YT^k\psi||_{L^{\infty}(\mathcal{S}_{\widetilde{\tau}})}+||NT^k\psi||_{L^{\infty}(\mathcal{S}_{\widetilde{\tau}})}\leq C\sqrt{\widetilde{E}^{\epsilon}_{0,I_0\neq 0;k+2}[\psi]}(1+\widetilde{\tau})^{-3-k+\epsilon},\\
\end{equation}
If we moreover assume that $\widetilde{E}^{\epsilon}_{0,I_0= 0;k+2}[\psi]<\infty$, then there exists a constant $C=C(D,R,k,\epsilon)>0$ such that for all $\widetilde{\tau}\geq 0$:
\begin{equation}
\label{eq:pointdecayNpsi2}
||YT^k\psi||_{L^{\infty}(\mathcal{S}_{\widetilde{\tau}})}+||NT^k\psi||_{L^{\infty}(\mathcal{S}_{\widetilde{\tau}})}\leq C\sqrt{\widetilde{E}^{\epsilon}_{0,I_0= 0;k+2}[\psi]}(1+\widetilde{\tau})^{-4-k+\epsilon}.
\end{equation}
\end{proposition}
\begin{proof}
We will only prove the $k=0$ case. For the $k\geq 1$ case we replace $\psi$ with $T^k\psi$ everywhere and we apply the appropriate decay estimates for $T^k\psi$. 

To obtain an estimate for $|\partial_r\psi|$ we first apply \eqref{waveequation} in the form:
\begin{equation*}
r^{-2}\partial_r(Dr^2\partial_r\psi)=-2\partial_r\partial_v\psi-2r^{-1}\partial_v\psi.
\end{equation*}
Using that $Y=\partial_r+h_{\mathcal{S}_{0}}\partial_v$ in $(v,r,\theta,\varphi)$ coordinates, we obtain
\begin{equation*}
\begin{split}
r^{-1}Y(Dr^2\partial_r\psi)=&-2r\partial_r(T\psi)-2T\psi+h_{\mathcal{S}_{0}}Dr\partial_rT\psi\\
=&\: (h_{\mathcal{S}_{0}}D-2)r\partial_r(T\psi)-2T\psi\\
=&\: O(r^{-\eta})\partial_r(T\psi)-2T\psi,
\end{split}
\end{equation*}
where we used that $|2-Dh_{\mathcal{S}_{0}}|\lesssim r^{-1-\eta}$, for some $\eta>0$.

We now integrate along $\mathcal{S}_{\widetilde{\tau}}$ and use the above equation to obtain
\begin{equation*}
\begin{split}
Dr^2\partial_r\psi(\widetilde{\tau},\widetilde{\rho})=&\:0+\int_{r_{\rm min}}^{\widetilde{\rho}}Y(Dr^2\partial_r\psi)(\widetilde{\tau},\widetilde{\rho}')\,d\widetilde{\rho}'\\
\leq&\: \int_{r_{\rm min}}^{\widetilde{\rho}} r\cdot r^{-1}|Y(Dr^2\partial_r\psi)|(\widetilde{\tau},\widetilde{\rho}')\,d\widetilde{\rho}'\\
\lesssim &\:\int_{r_{\rm min}}^{\widetilde{\rho}} r\,d\widetilde{\rho}'\cdot \left[||\partial_r(T\psi)||_{L^{\infty}(\mathcal{S}_{\widetilde{\tau}})}+||T\psi||_{L^{\infty}(\mathcal{S}_{\widetilde{\tau}})}\right]\\
\lesssim &\: Dr^2 \left[||\partial_r(T\psi)||_{L^{\infty}(\mathcal{S}_{\widetilde{\tau}})}+||T\psi||_{L^{\infty}(\mathcal{S}_{\widetilde{\tau}})}\right],
\end{split}
\end{equation*}
where in the last inequality we used that in the $r_{\rm min}=r_+$ case we can write $D(r)=d(r)(r-r_+)$, with $d(r_+)>0$.
Hence,
\begin{equation*}
||Y\psi||_{L^{\infty}(\mathcal{S}_{\widetilde{\tau}})}+||N\psi||_{L^{\infty}(\mathcal{S}_{\widetilde{\tau}})}\lesssim ||\partial_r\psi||_{L^{\infty}(\mathcal{S}_{\widetilde{\tau}})}+||T\psi||_{L^{\infty}(\mathcal{S}_{\widetilde{\tau}})} \lesssim ||\partial_r(T\psi)||_{L^{\infty}(\mathcal{S}_{\widetilde{\tau}})}+||T\psi||_{L^{\infty}(\mathcal{S}_{\widetilde{\tau}})}.
\end{equation*}

By Proposition \ref{prop:pointdecpsi} we have that
\begin{align*}
||T\psi||_{L^{\infty}(\mathcal{S}_{\widetilde{\tau}})}\lesssim&\: \sqrt{E^{\epsilon}_{0,I_0\neq 0;2}[\psi]}(1+\widetilde{\tau})^{-3+\epsilon},\\
||T\psi||_{L^{\infty}(\mathcal{S}_{\widetilde{\tau}})}\lesssim&\: \sqrt{E^{\epsilon}_{0,I_0= 0;2}[\psi]}(1+\widetilde{\tau})^{-4+\epsilon}.
\end{align*}

By Proposition \ref{cor:rhalfYpsidecay} with $k=1$ we moreover have that
\begin{align*}
(1+\widetilde{\rho})^{\frac{1}{2}}||\partial_rT\psi||_{L^{\infty}(\mathcal{S}_{\widetilde{\tau}})}\leq&\: C\sqrt{\widetilde{E}^{\epsilon}_{0,I_0\neq 0;2}[\psi]}(1+\widetilde{\tau})^{-\frac{7}{2}+\epsilon},\\
(1+\widetilde{\rho})^{\frac{1}{2}}||\partial_rT\psi||_{L^{\infty}(\mathcal{S}_{\widetilde{\tau}})}\leq&\: C\sqrt{\widetilde{E}^{\epsilon}_{0,I_0= 0;2}[\psi]}(1+\widetilde{\tau})^{-\frac{9}{2}+\epsilon}.
\end{align*}

The statement of the proposition now follows immediately for the $k=0$ case. The $k>0$ case can be treated identically by replacing $\psi$ with $T^k\psi$ and applying the appropriate higher-order decay estimates.
\end{proof}

\section{Asymptotics I: The case $I_{0}\neq 0$}
\label{sec:asymradfieldcaseI}

We obtain in this section the precise asymptotics of solutions to the wave equation with a non-vanishing first Newman--Penrose constant $I_0[\psi]$. Here, $\psi$ will always denote a spherically symmetric solution to \eqref{waveequation} emanating from initial data given as in Theorem \ref{thm:extuniq} on a spacetime $(\mathcal{M}, g)$ that satisfies the geometric assumptions from Section \ref{sec:geomassm}, such that moreover $I_0[\psi]\neq 0$.

We will moreover make the additional stronger assumption on the $r$-asymptotics of $D(r)$:
\begin{equation*}
D(r)=1-\frac{2M}{r}+O_{3+k}(r^{-1-\beta}),
\end{equation*}
whenever we need to appeal to the decay estimates from Section \ref{gdeforpsi1}--\ref{improvedhyperboloidaldecay} that require the above assumption on $D$ with a suitably large $k\geq 1$.

\subsection{Asymptotics of $v^{2}\partial_{v}(r\psi)$ in the region $\mathcal{B}_{\alpha}$}
\label{npasy}
The starting point for deriving the asymptotics of the radiation field $\phi$ is to determine the asymptotics of the first Newman--Penrose quantity near infinity. First, let us introduce the following $L^{\infty}$ norm for the first Newman--Penrose quantity along $\Sigma_0$:
\begin{equation}
\label{def:PI0}
 P_{I_{0},\beta}[\psi] := \left\|v^{2+\beta}\cdot \left(\partial_v\phi-2\frac{I_0 [\psi ]}{v^2}\right)\right\|_{L^{\infty}(\Sigma_{0})}.
\end{equation}

In this section, we will restrict to regions of the form
\begin{equation*}
\mathcal{B}_{\alpha}:=\{r\geq R\}\cap \{0\leq u\leq v-v^{\alpha}\}\subset \mathcal{A},
\end{equation*}
with $\alpha\in (0,1)$ suitably chosen.

Note that the boundary of $\mathcal{B}_{\alpha}$ contains a subset of the timelike hypersurface
\begin{equation*}
\gamma_{\alpha}=\{v-u=v^{\alpha}\}.
\end{equation*}

Without loss of generality, we will assume that
\begin{equation*}
v_{\gamma_{\alpha}}(u)\geq v_{r=R}(u),
\end{equation*}
for all $u\geq 0$. We can always redefine $\gamma_{\alpha}$ to be the hypersurface along which $v-u=v^{\alpha}+r_*(R)$ and similarly redefine $\mathcal{B}_{\alpha}$ to ensure the above inequality holds for all $u\geq 0$.

Consequently, for all $(u,v)\in \mathcal{B}_{\alpha}$, we have that $v\geq v_{\gamma_{\alpha}}(u)$.

\begin{proposition}\label{prop:lower1}
Consider the region $\mathcal{B}_{\alpha}=\{r\geq R\}\cap \{0\leq u\leq v-v^{\alpha}\}$, with a fixed $\alpha\in (\frac{2}{3},1)$. Let $\epsilon\in\left(0,(3\alpha-2)/2\right)$ and assume that $E^{\epsilon}_{0,I_0\neq 0;0}[\psi]<\infty$, and moreover that there exists a $\beta>0$ such that
\begin{equation}
\label{eq: assinitialdat}
P_{I_{0},\beta}[\psi] <\infty .
\end{equation}

Then we have that
\begin{equation}\label{eq:lower1}
\left|v^2\partial_v\phi (u,v)- 2I_0 [\psi ]\right| \leq C\sqrt{E^{\epsilon}_{0,I_0\neq 0;0}[\psi]}\frac{1}{v^{3\alpha-2-2\epsilon}}+  P_{I_{0},\beta}[\psi]\cdot  v^{-\beta}
\end{equation}
for all $(u,v)\in \mathcal{B}_{\alpha}$, with $C=C(D,\Sigma,R,\alpha,\epsilon)>0$ constants.
\end{proposition}
\begin{proof}
The equation \eqref{waveequation} implies the following equation for $\phi$ (see for example the derivation in the Appendix of \cite{paper1}):
\begin{equation}
\label{eqphi}
\partial_u\partial_v\phi =-\frac{DD' }{4r}\phi,
\end{equation}
where by the assumptions on $D$ it follows that
\begin{equation*}
\frac{DD' }{r}=O(r^{-3}). 
\end{equation*}
Integrating in the $u$-direction, together with (\ref{eqphi}), it follows that for $(u,v)\in \mathcal{B}_{\alpha}$
\begin{equation*}
v^2\partial_v \phi (u,v)=v^2\partial_v\phi (0,v)-v^2\int_{0}^u\frac{DD' }{4r}\phi (u',v)\,du'.
\end{equation*}

We can express
\begin{equation}
\frac{1}{2}(v-u)=r_*(r)=r+\int_{R}^r (D^{-1}(r')-1)\,dr'.
\label{uvr}
\end{equation}

Using the asymptotics of $D$ it therefore follows that for suitably large $R>0$, \[r\gtrsim v-u\geq v_{\gamma_{\alpha}}(u)-u=v_{\gamma_{\alpha}}^{\alpha}(u) \gtrsim (u+1)^{\alpha}\] and moreover,
\begin{equation*}
r\gtrsim v-u\geq v-u_{\gamma_{\alpha}}(v)=v^{\alpha},
\end{equation*}
which implies that
\[v\lesssim r^{\alpha^{-1}}\]
 for $R>0$ suitably large, and for $\eta$ such that $0<\eta< 3\alpha-2-\epsilon$ we can estimate
\begin{equation*}
\begin{split}
v^2\int_{0}^u\frac{|DD'|}{r}|\phi |(u',v)\,du'&\leq Cv^{-\eta}\int_{0}^u  r^{-3}\cdot v^{2+\eta}\cdot|\phi |(u',v)\,du'\\
&\leq Cv^{-\eta}\int_{0}^u  r^{-3}\cdot r^{(2+\eta)\alpha^{-1}}\cdot|\phi |(u',v)\,du'\\
&= Cv^{-\eta}\int_{0}^u  r^{-3+(2+\eta)\alpha^{-1}}\cdot|\phi |(u',v)\,du'\\
&= Cv^{-\eta}\int_{0}^u  (u'+1)^{-3\alpha+(2+\eta)}\cdot|\phi |(u',v)\,du'\\
&\leq C\sqrt{E^{\epsilon}_{0,I_0\neq 0;0}[\psi]} v^{-\eta}\int_{0}^u(u'+1)^{-3\alpha+2+\eta}(u'+1)^{-1+\epsilon}\,du'\\
&= C\sqrt{E^{\epsilon}_{0,I_0\neq 0;0}[\psi]} v^{-\eta}\int_{0}^u(u'+1)^{-3\alpha+2+\eta-1+\epsilon}\,du'\\
&\leq C\sqrt{E^{\epsilon}_{0,I_0\neq 0;0}[\psi]} v^{-\eta},
\end{split}
\end{equation*}
where we used the upper bound for $|\phi |$ from Proposition \ref{prop:pointdecpsi} and that, by the definition of $\eta$, \[-3\alpha+2+\eta-1+\epsilon<-1.\] By assumption \eqref{eq: assinitialdat}, we moreover have that
\begin{equation*}
|v^2\partial_v\phi (0,v)-2I_0 [\psi ]|\leq P_{I_0,\beta}[\psi]\cdot v^{-\beta}.
\end{equation*}
If we choose
\[\eta=(3\alpha-2-\epsilon)-\epsilon=3\alpha-2-2\epsilon\]
then we conclude that, given $\alpha\in (\frac{2}{3},1)$ and $\epsilon\in (0,(3\alpha-2)/2)$, we can estimate
\begin{equation*}
\left|v^2\partial_v\phi (u,v)- 2I_0 [\psi ]\right| \leq C\sqrt{E^{\epsilon}_{0,I_0\neq 0;0}[\psi]}v^{2-3\alpha+2\epsilon}+P_{I_0,\beta}[\psi]\cdot v^{-\beta}
\end{equation*}
for all $(u,v)\in \mathcal{B}_{\alpha}$, with $C=C(D,\Sigma,R,\alpha,\epsilon)>0$ a constant.
\end{proof}

\subsection{Asymptotics for the radiation field $r\psi$}
\label{asyrafield1}
The next proposition gives the asymptotic behaviour of the radiation field $\phi$ along null infinity. 
\begin{proposition}\label{cor1}
Under the assumptions of Proposition \ref{prop:lower1}, with additionally $\alpha\in [\frac{5}{7},1)$ and $\epsilon\in(0,\frac{1}{6}(1-\alpha))$, we have for all $(u,v)\in \mathcal{B}_{\alpha}$ that
\begin{equation}
\begin{split}
&\bigg|\phi (u,v)-2I_0 [\psi] \left((u+1)^{-1}-v^{-1}\right)\bigg| \leq \\   &\ \ \ \ \ \ \ \ \ \ \ C\left(\sqrt{E^{\epsilon}_{0,I_0\neq 0;0}[\psi]}+I_0[\psi]\right)(u+1)^{\frac{\alpha}{2}-\frac{3}{2}+2\epsilon}+C\cdot P_{I_{0},\beta}[\psi]\cdot (u+1)^{-1-\beta},
\end{split}
\label{asyphiprop}
\end{equation}
where $C=C(D,\Sigma,R,\alpha,\epsilon)>0$ is a constant.

In particular, we obtain the following asymptotics for $\phi$ along $\mathcal{I}^+$
\begin{equation*}
\left|\phi (u,\infty)-2I_0 [\psi ] (u+1)^{-1}\right|\leq C\left(\sqrt{E^{\epsilon}_{0,I_0\neq 0;0}[\psi]}+I_0[\psi]\right)(u+1)^{-1-\epsilon}+C\cdot P_{I_{0},\beta}[\psi]\cdot(u+1)^{-1-\beta}.
\end{equation*}
In fact, if we further impose $   \frac{1-\alpha}{2}<\beta+2\epsilon$, then the estimate \eqref{asyphiprop} provides first-order asymptotics for $\phi$ in the region $\mathcal{B}_{\delta}$ for $\delta $ such that $1>\delta> \frac{\alpha}{2}+\frac{1}{2}+2\epsilon>\alpha+2\epsilon$.

\end{proposition}
\begin{proof}
Denote $\gamma_{\alpha}=\{v-u=v^{\alpha}\}$. Then we can integrate in the $v$-direction in $\mathcal{B}_{\alpha}$ (defined in Proposition \ref{prop:lower1}) to estimate
\begin{equation*}
\phi (u,v)=\phi (u,v_{\gamma_{\alpha}}(u))+\int_{v_{\gamma_{\alpha}}(u)}^{v}\partial_v\phi (u,v')\,dv'.
\end{equation*}
By Proposition \ref{prop:pointdecpsi}, we have 
\begin{equation*}
\phi (u,v_{\gamma_{\alpha}}(u))=r^{\frac{1}{2}}\cdot r^{\frac{1}{2}}\psi (u,v_{\gamma_{\alpha}}(u))\leq C\sqrt{E^{\epsilon}_{0,I_0\neq 0;0}[\psi]}v_{\gamma_{\alpha}}^{\frac{\alpha}{2}}(u)(u+1)^{-\frac{3}{2}+\epsilon}.
\end{equation*}
Note that $u_{\gamma_{\alpha}}(v)+1=v-v^{\alpha}+1\geq \frac{1}{2}v$, for suitably large $R>0$, so
\begin{equation}
u+1 \leq v_{\gamma_{\alpha}}(u) \leq 2(u+1).
\label{uvsimilarity}
\end{equation}
Hence we have that
\begin{equation*}
\phi(u,v_{\gamma_{\alpha}}(u))\leq C\sqrt{E^{\epsilon}_{0,I_0\neq 0;0}[\psi]}(u+1)^{\frac{\alpha}{2}-\frac{3}{2}+\epsilon}.
\end{equation*}
Furthermore, using estimate \eqref{eq:lower1} we have that
\begin{equation*}
\begin{split}
\left|\int_{v_{\gamma_{\alpha}}(u)}^{v}\partial_v\phi(u,v')\,dv'-2I_0 [\psi ] (v_{\gamma_{\alpha}}(u)^{-1}-v^{-1})\right|\leq&\: C\sqrt{E^{\epsilon}_{0,I_0\neq 0;0}[\psi]}(v_{\gamma_{\alpha}}(u)^{-3\alpha+1+2\epsilon}-v^{-3\alpha+1+2\epsilon})\\
&+C\cdot P_{I_0,\beta}[\psi](v_{\gamma_{\alpha}}(u)^{-\beta-1}-v^{-\beta-1}).
\end{split}
\end{equation*}
We have 
\[  \left( v_{\gamma_{\alpha}}(u)^{-1}-v^{-1} \right) =
 \left( v_{\gamma_{\alpha}}(u)^{-1}-(u+1)^{-1} \right) +
\left( (u+1)^{-1}-v^{-1} \right)  \]
and
\begin{equation*}
\left|v_{\gamma_{\alpha}}(u)^{-1}-(u+1)^{-1}\right|=\left|\frac{u+1-v_{\gamma_{\alpha}}(u)}{(u+1)v_{\gamma_{\alpha}}(u)}\right|=(u+1)^{-1}v_{\gamma_{\alpha}}^{\alpha-1}(u)\leq C(u+1)^{-2+\alpha}.
\end{equation*}

Hence,
\begin{equation}
\begin{split}
\Bigg|\int_{v_{\gamma_{\alpha}}(u)}^{v}&\partial_v\phi(u,v')\,dv'-2I_0 [\psi ] ((u+1)^{-1}-v^{-1})\Bigg|\\
\leq&\: CI_0 [\psi ] (u+1)^{-2+\alpha}+C\sqrt{E^{\epsilon}_{0,I_0\neq 0;0}[\psi]}(v_{\gamma_{\alpha}}(u)^{-3\alpha+1+2\epsilon}-v^{-3\alpha+1+2\epsilon})\\
&+C\cdot P_{I_0,\beta}[\psi](v_{\gamma_{\alpha}}(u)^{-\beta-1}-v^{-\beta-1}).
\end{split}
\label{estimateforvintegral}
\end{equation}
In view of \eqref{uvsimilarity}  we have $v_{\gamma_{\alpha}}(u)\sim (u+1)$ along $\gamma_{\alpha}$ and hence  we can further estimate
\begin{equation*}
\begin{split}
\sqrt{E^{\epsilon}_{0,I_0\neq 0;0}[\psi]}&(v_{\gamma_{\alpha}}(u)^{-3\alpha+1+2\epsilon}-v^{-3\alpha+1+2\epsilon})+\cdot P_{I_0,\beta}[\psi](v_{\gamma_{\alpha}}(u)^{-\beta-1}-v^{-\beta-1})\\
\leq&\: C\sqrt{E^{\epsilon}_{0,I_0\neq 0;0}[\psi]}(u+1)^{-3\alpha+1+2\epsilon}+C\cdot P_{I_0,\beta}[\psi](u+1)^{-1-\beta}.\\
\end{split}
\end{equation*}

We have that
\begin{align*}
-2+\alpha<&\:\frac{\alpha}{2}-\frac{3}{2}+\epsilon,
\end{align*}
for all $\alpha\in (\frac{2}{3},1)$.

Note that $1-3\alpha+2\epsilon\leq \frac{\alpha}{2}-\frac{3}{2}+2\epsilon$, for $\alpha\geq \frac{5}{7}$ and $1-3\alpha+2\epsilon\geq \frac{\alpha}{2}-\frac{3}{2}+2\epsilon$ for $\alpha\leq \frac{5}{7}$. 

We can therefore conclude that for $\alpha\in [\frac{5}{7},1)$ and $\epsilon \in (0,\frac{1}{2}(3\alpha-2))$
\begin{equation*}
\begin{split}
\left|\phi (u,v)-2I_0 [\psi] \left((u+1)^{-1}-v^{-1}\right)\right|\leq&\: C\left(\sqrt{E^{\epsilon}_{0,I_0\neq 0;0}[\psi]}+I_0[\psi]\right)(u+1)^{\frac{\alpha}{2}-\frac{3}{2}+2\epsilon}\\
&+C\cdot P_{I_0,\beta}[\psi](u+1)^{-1-\beta}.
\end{split}
\end{equation*}
Furthermore, if we consider $\epsilon\in (0,\frac{1}{6}(1-\alpha))$, we have in particular that $\epsilon<\frac{1}{2}(3\alpha-2)$ and moreover $\frac{\alpha}{2}-\frac{3}{2}+2\epsilon<-1-\epsilon$.

Finally, note if $\delta $ satisfies $1>\delta> \frac{\alpha}{2}+\frac{1}{2}+2\epsilon>\alpha+2\epsilon$ then in the region $\mathcal{B}_{\delta}$ bounded by $\gamma_{\delta}$ we obtain
\[\left|\frac{1}{u+1}-\frac{1}{v}\right| \geq \left|\frac{1}{u+1}-\frac{1}{v_{\gamma_{\delta}}}\right|=
\left|\frac{ v_{\gamma_{\delta}}-(u+1)}{v_{\gamma_{\delta}}\cdot (u+1)}  \right| = \left| \frac{v^{\delta}_{\gamma_{\delta}}}{v_{\gamma_{\delta}}\cdot (u+1)}\right|\sim (u+1)^{-2+\delta}. \]
Since  $-2+\delta > \frac{\alpha}{2}-\frac{3}{2}+2\epsilon>-1-\beta$, if moreover $\frac{1-\alpha}{2}<\beta+2\epsilon$, estimate \eqref{asyphiprop} indeed provides the asymptotic behavior of $\phi$ in the region $\mathcal{B}_{\delta}$.
\end{proof}

\subsubsection{Asymptotics of $T^{k}\phi$}
\label{rftkhiasy}
We next obtain the asymptotics of $T^k\phi$ for each $k\in \N_0$. To that end, we define the following $L^{\infty}$ norm for higher-order $v$-derivatives of $\phi$:

\begin{equation}
\label{def:PI0k}
 P_{I_{0},\beta; k}  [\psi]:= \max_{1\leq j\leq k}\left\| v^{2+\beta+j}\cdot\left| \partial_{v}^{j}\left( \partial_{v}\phi-\frac{2I_{0}[\psi]}{v^{2}}  \right)  \right|\right\|_{L^{\infty}(\Sigma_{0})} .
\end{equation}

The first step towards obtaining asymptotics for $T^k\phi$ is to derive the following estimate for $\partial_v^k \phi$, with $k\in \N_0$.

\begin{proposition}\label{prop:tklower}
Let $k\in \N_0$. Consider the region $\mathcal{B}_{\alpha_k} =\{r\geq R\}\cap \{0\leq u\leq v-v^{\alpha_k}\}$, where $\alpha_k\in (\frac{k+2}{k+3},1)$. Let $\epsilon\in(0,\frac{1}{2}(k+3)\alpha-\frac{1}{2}(k+2))$ be arbitrarily small and assume that $E^{\epsilon}_{I_0\neq 0;0}[\psi]<\infty$, and moreover 
\begin{equation}
\label{eq: assinitialdatderT}
P_{I_{0},\beta; k}  [\psi]<\infty 
\end{equation}
for $\beta>0$.

Then, we have that
\begin{equation}
\label{est: assinitialdatderT}
\begin{split}
\left|v^{k+2} \partial_v^{k+1}\phi (u,v)-(-1)^{k} (k+1)! \cdot 2I_0 [\psi ]\right|\leq&\: C\sqrt{E^{\epsilon}_{0,I_0\neq 0;0}[\psi]}v^{(k+2)-(k+3)\alpha+2\epsilon} \\
&+C\cdot P_{I_{0},\beta; k}  [\psi]\cdot v^{-\beta},
\end{split}
\end{equation}
for all $(u,v)\in \mathcal{B}_{\alpha_k}$, with $C=C(D,\Sigma,R,k,\alpha_k,\epsilon)>0$ a constant.
\end{proposition}
\begin{proof}
By assuming
\begin{equation*}
D(r)=1-\frac{2M}{r}+O_{3+k}(v^{-1-\beta})
\end{equation*}
as in the start of Section \ref{sec:asymradfieldcaseI}, it follows in particular that
\begin{equation*}
\left|\partial_v^k\left(\frac{DD'}{r}\right)\right|\leq Cr^{-3-k},
\end{equation*}
for some constant $C=C(D,k)>0$.

By commuting \eqref{eqphi} with $\partial_v^{k}$, we therefore obtain:
\begin{equation*}
\begin{split}
v^{k+2}\partial_v^{k+1} \phi (u,v)=&\:v^{k+2}\partial_v^{k+1} \phi (0,v)-2v^{k+2} \int_{0}^u\partial_v^{k}\left(\frac{DD' }{r}\phi (u',v)\right)\,du'\\
=&\:v^{k+2}\partial_v^{k+1}\phi (0,v)+ \sum_{l=0}^{k} v^{k+2} \int_{0}^u O(r^{-3-k+l})\partial_v^l \phi (u',v)\,du' ,
\end{split}
\end{equation*}

The desired result can be easily seen to hold for $k = 1$, using the asymptotics of $\partial_v\phi$ derived in Proposition \ref{prop:lower1}. The general case follows by an induction argument, appealing to the results of Proposition \ref{prop:lower1} and Proposition \ref{cor1}, and the assumption \eqref{eq: assinitialdatderT}.
\end{proof}
Using Proposition \ref{prop:tklower} we are able to derive asymptotics for $\partial_v (T^{k} \phi)$, with $k\in \mathbb{N}$.
\begin{proposition}
\label{cor:estpartialvTkphi}
Let $k\in \N$. Under the assumptions of Proposition \ref{prop:tklower}, together with $E^{\epsilon}_{0,I_0\neq0;k}[\psi]<\infty$, we have that for any $1\leq k \leq n$:
\begin{equation*}
\begin{split}
|\partial_v(T^{k}& \phi) (u,v)- (-1)^{k} (k+1)! \cdot 2I_0 [ \psi ] v^{-k-2}|\\
\leq&\: C\left(\sqrt{E^{\epsilon}_{I_0\neq 0;k}[\psi]}+I_0[\psi]\right)\sum_{l=0}^{k-1} r^{-3-l} (u+1)^{-k+l+\epsilon} \\
&+ C\sqrt{E^{\epsilon}_{I_0\neq 0;0}[\psi]}v^{-(k+3)\alpha_k+2\epsilon}+C\cdot P_{I_{0},\beta; k}  [\psi]\cdot v^{-2-k-\beta},
\end{split}
\end{equation*}
for all $(u,v)\in \mathcal{B}_{\alpha_k}$, with $C=C(D,\Sigma,R,k,\alpha_k,\epsilon)>0$ a constant.
\end{proposition}
\begin{proof}
We first note that using the equation \eqref{eqphi} together with the fact that the equation commutes with $T$ and an inductive argument, we have that the following equation holds for all $k \geq 1$:
\begin{equation}\label{eq:wave_null_k}
\partial_v (T^{k} \phi ) = \partial_v^{k+1} \phi + \sum_{s=0}^{k-1}\sum_{l+m = k-1-s}  O (r^{-3-s} ) \cdot \partial_v^l ( T^m \phi ) .
\end{equation}
Now the claim follows by strong induction using equation \eqref{eq:wave_null_k}, Proposition \ref{prop:tklower}, and the pointwise decay estimates from Proposition \ref{prop:pointdecpsi}.
\end{proof}
Finally using Proposition \ref{cor:estpartialvTkphi} we are able to derive the asymptotic expansion for $T^{k} \phi$ in $\mathcal{B}_{\alpha_k}$ for any $k \in \mathbb{N}$.
\begin{proposition}\label{asymprecTk}
Let $k\in \N_0$. Under the assumptions of Proposition \ref{prop:tklower}, with additionally $\alpha_k\in [\frac{2k+5}{2k+7},1)$ and $\epsilon\in (0,\frac{1}{6}(1-\alpha_k))$, together with $E^{\epsilon}_{0,I_0\neq0;k}[\psi]<\infty$, we have that
\begin{equation}
\label{asyTkphiprop}
\begin{split}
|T^{k} \phi (u,v)&- (-1)^{k} k! \cdot  2I_0 [\psi ] \left(u^{-k-1}- v^{-k-1}\right)|\\
\leq&\: C\left(\sqrt{E^{\epsilon}_{I_0\neq 0;k}[\psi]}+I_0[\psi]\right)(u+1)^{-\frac{3}{2}-k+\frac{\alpha_k}{2}+2\epsilon}\\
&+C\cdot P_{I_{0},\beta; k}  [\psi]\cdot (u+1)^{-1-k-\beta},
\end{split}
\end{equation}
for all $(u,v)\in \mathcal{B}_{\alpha_k}$ where $C=C(D,\Sigma,R,k,\alpha_k,\epsilon)>0$ is a constant.

In particular, along $\mathcal{I}^+$ we have that
\begin{equation*}
\begin{split}
|T^{k} \phi (u,\infty)-(-1)^{k} k! \cdot 2I_0 [\psi ] u^{-k-1}|\leq&\: C\left(\sqrt{E^{\epsilon}_{I_0\neq 0;k}[\psi]}+I_0[\psi]\right)(u+1)^{-1-k+\epsilon}\\
&+C\cdot P_{I_{0},\beta; k}  [\psi]\cdot (u+1)^{-1-k-\beta}.
\end{split}
\end{equation*}
In fact, if we further impose $   \frac{1-\alpha_k}{2}<\beta+2\epsilon$, then the estimate \eqref{asyTkphiprop} provides first-order asymptotics for $\phi$ in the region $\mathcal{B}_{\delta_k}$ for $\delta_k $ such that $1>\delta_k> \frac{\alpha_k}{2}+\frac{1}{2}+2\epsilon>\alpha_k+2\epsilon$.
\end{proposition}
\begin{proof}
The $k=0$ case follows from Proposition \ref{asyrafield1}. For the remainder of the proof we let $k\geq 1$.

Let $\gamma_{\alpha_k}=\{v-u=v^{\alpha_k}\}$. We integrate the estimate from Proposition \ref{cor:estpartialvTkphi} starting from the curve $\gamma_{\alpha_k}$, which gives us that:
\begin{equation}\label{eq:tk_aux_asym}
\begin{split}
|T^{k} \phi& (u,v)-(-1)^{k} 2k! \cdot  I_0 [\psi ] ( v^{-k-1} - v_{\gamma_{\alpha_k}}^{-k-1}(u) ) | \leq  |T^{k} \phi| (u , v_{\gamma_{\alpha_k}} (u))\\ 
&+ C\left(\sqrt{E^{\epsilon}_{I_0\neq 0;k}[\psi]}+I_0[\psi]\right) \sum_{l=0}^{k-1}\int_{v_{\gamma_{\alpha_k}}}^v  r^{-3-l} u^{-k+l+\epsilon}\,dv'\\
&+C\sqrt{E^{\epsilon}_{I_0\neq 0;k}[\psi]}( v_{\gamma_{\alpha_k}}^{1-(k+3)\alpha_k+2\epsilon}(u)-v^{1-(k+3)\alpha_k+2\epsilon} ) \\
&+ C\cdot P_{I_{0},\beta; k}  [\psi](v_{\gamma_{\alpha_k}}^{-k-1-\beta}(u)-v^{-k-1-\beta}) . 
\end{split}
\end{equation}
Note that
$$ |  T^{k} \phi |  (u , v_{\gamma_{\alpha_k}} (u)) = r^{\frac{1}{2}} \cdot r^{\frac{1}{2}} | T^{k} \psi | (u , v_{\gamma_{\alpha_k}} (u)) \leq C \sqrt{E_{0,I_0\neq0;k}^{\epsilon}[\psi]} (u+1)^{-\frac{3}{2}-k+\frac{\alpha_k}{2}+\epsilon} , $$
by Proposition \ref{prop:pointdecpsi}.

Furthermore,
$$ \sum_{l=0}^{k-1}\int_{v_{\gamma_{\alpha_k}}}^v  r^{-3-l} (u+1)^{-k+l+\epsilon}\,dv' \leq C \sum_{l=0}^{k-1}(u+1)^{-k+l+\epsilon-(2+l)\alpha_k}\leq C (u+1)^{-1+\epsilon-(k+1)\alpha_k},
$$ 
and
\begin{equation*}
\begin{split}
C\sqrt{E^{\epsilon}_{I_0\neq 0;k}[\psi]}&( v_{\gamma_{\alpha_k}}^{1-(k+3)\alpha_k+2\epsilon}(u)-v^{1-(k+3)\alpha_k+2\epsilon})+ C\cdot P_{I_{0},\beta; k}  [\psi](v_{\gamma_{\alpha_k}}^{-k-1-\beta}(u)-v^{-k-1-\beta})\\
\leq&\:C\sqrt{E^{\epsilon}_{I_0\neq 0;k}[\psi]}(u+1)^{1-(k+3)\alpha_k+2\epsilon}+ C\cdot P_{I_{0},\beta; k}  [\psi](u+1)^{-k-1-\beta},
\end{split}
\end{equation*}
which follows from the arguments used in Proposition \ref{cor1} to obtain the above expression with $k=0$.

Finally, we use that
\begin{equation*}
\left|v_{\gamma_{\alpha_k}}(u)^{-k-1}-(u+1)^{-k-1}\right|\leq C(u+1)^{-2-k+\alpha_k}
\end{equation*}
and we combine the above estimates to obtain
\begin{equation*}
\begin{split}
|T^{k} \phi (u,v)&- (-1)^{k} k! \cdot  2I_0 [\psi ] \left(u^{-k-1}- v^{-k-1}\right)|\\
\leq &\: CI_0 [\psi ](u+1)^{-2-k+\alpha_k}+C\cdot P_{I_0,\beta;k}[\psi](u+1)^{-1-k-\beta}\\
&+C\sqrt{E^{\epsilon}_{I_0\neq 0;k}[\psi]}(u+1)^{-\frac{3}{2}-k+\frac{\alpha_k}{2}+\epsilon}\\
&+ C \sqrt{E^{\epsilon}_{I_0\neq 0;k}[\psi]}(u+1)^{1-(k+3)\alpha_k+2\epsilon}\\
&+C\left(\sqrt{E^{\epsilon}_{I_0\neq 0;k}[\psi]}+I_0[\psi]\right)(u+1)^{-1-(k+1)\alpha_k+\epsilon}.
\end{split}
\end{equation*}
For $\alpha_k\in (\frac{k+2}{k+3},1)$ we have that
\begin{equation*}
\begin{split}
\max&\left\{-\frac{3}{2}-k+\frac{\alpha_k}{2}, 1-(k+3)\alpha_k, -1-(k+1)\alpha_k,-2-k+\alpha_k\right\}\\
&=\max\left\{-\frac{3}{2}-k+\frac{\alpha_k}{2},1-(k+3)\alpha_k\right\}
\end{split}
\end{equation*}
and hence, for $\alpha_k\in (\frac{k+2}{k+3},1)$ and $\epsilon\in(0,\frac{1}{2}(k+3)\alpha_k-\frac{1}{2}(k+2))$:
\begin{equation*}
\begin{split}
|T^{k} \phi (u,v)&- (-1)^{k} k! \cdot  2I_0 [\psi ] \left(u^{-k-1}- v^{-k-1}\right)|\\
\leq&\: C\left(\sqrt{E^{\epsilon}_{I_0\neq 0;k}[\psi]}+I_0[\psi]\right)(u+1)^{\max\{-\frac{3}{2}-k+\frac{\alpha_k}{2},1-(k+3)\alpha_k\}+2\epsilon}\\
&+C\cdot P_{I_{0},\beta; k}  [\psi](u+1)^{-1-k-\beta}.
\end{split}
\end{equation*}
Furthermore,
\begin{equation*}
-\frac{3}{2}-k+\frac{\alpha_k}{2}\geq 1-(k+3)\alpha_k
\end{equation*}
if $\alpha_k\geq \frac{2k+5}{2k+7}$ (and the inequality is reversed if $\alpha_k\leq \frac{2k+5}{2k+7}$).

We also have that
\begin{equation*}
-\frac{3}{2}-k+\frac{\alpha_k}{2}+2\epsilon\leq -1-k-\epsilon
\end{equation*}
if $\epsilon<\frac{1}{6}(1-\alpha_k)$. Note that $\frac{1}{6}(1-\alpha_k)<\frac{1}{2}(k+3)\alpha_k-\frac{1}{2}(k+2)$.

So, for $\alpha_k\in [\frac{2k+5}{2k+7},1)$ and $\epsilon\in (0,\frac{1}{6}(1-\alpha_k))$ we obtain in particular
\begin{equation*}
\begin{split}
|T^{k} \phi (u,v)&- (-1)^{k} k! \cdot  2I_0 [\psi ] \left(u^{-k-1}- v^{-k-1}\right)|\\
\leq&\: C\left(\sqrt{E^{\epsilon}_{I_0\neq 0;k}[\psi]}+I_0[\psi]\right)(u+1)^{-1-k-\epsilon}\\
&+C\cdot P_{I_{0},\beta; k}  [\psi](u+1)^{-1-k-\beta}.
\end{split}
\end{equation*}

Finally, if $\delta_k $ satisfies $1>\delta_k> \frac{\alpha_k}{2}+\frac{1}{2}+2\epsilon>\alpha_k+2\epsilon$ then in the region $\mathcal{B}_{\delta_k}$ bounded by $\gamma_{\delta_k}$ we obtain
\begin{align*}
\left|\frac{1}{(u+1)^{k+1}}-\frac{1}{v^{k+1}}\right| \geq&\: \left|\frac{1}{(u+1)^{k+1}}-\frac{1}{v_{\gamma_{\delta_k}}^{k+1}}\right|=
\left|\frac{ v_{\gamma_{\delta_k}}^{k+1}-(u+1)^{k+1}}{v_{\gamma_{\delta_k}}^{k+1}\cdot (u+1)^{k+1}}  \right| \\
=&\:\sum_{j=0}^k\left|\frac{ v_{\gamma_{\delta_k}}-(u+1)}{v_{\gamma_{\delta_k}}^{j+1}\cdot (u+1)^{k+1-j}}  \right|\\
\sim &\: (u+1)^{-2-k+\delta_k},
\end{align*}
where we made use of the identity
\begin{equation*}
v^{k+1}-(u+1)^{k+1}=(v-u)\sum_{j=0}^{k}(u+1)^jv^{k-j}.
\end{equation*}
Since  $-2+\delta_k-k > \frac{\alpha_k}{2}-\frac{3}{2}-k+2\epsilon>-1-\beta-k$, if moreover $\frac{1-\alpha_k}{2}<\beta+2\epsilon$, estimate \eqref{asyTkphiprop} indeed provides the asymptotic behaviour of $\phi$ in the region $\mathcal{B}_{\delta_k}$.
\end{proof}

\subsection{Global asymptotics for the scalar field $\psi$}
\label{sec:asympsinonzeronp}

In this section, we propagate the late-time asymptotics of $\psi$ that follow from the asymptotics of $r\psi$ near infinity in the $I_0[\psi]\neq 0$ case to the entire spacetime region $\mathcal{R}$.
\begin{proposition}\label{prop:psi_lower}
Let $\epsilon\in (0,\frac{1}{98})$  such that $\beta>\epsilon$. Assume that the initial data for a solution $\psi$ to the wave equation satisfy \[ P_{I_{0},\beta}[\psi] <\infty, \] with $P_{I_{0},\beta}[\psi]$ defined in \eqref{def:PI0},
 and
\[\widetilde{E}^{\epsilon}_{0,I_0\neq0;1}[\psi]< \infty.\] 
Then for all $(u,v)\in \mathcal{R}\cap\{r\geq R\}$ we can estimate
\begin{equation}\label{price_0_vpsi}
\left|\psi (u,v) -\frac{4I_0 [\psi ]}{(u+1)v}\right| \leq C\left(\sqrt{\widetilde{E}^{\epsilon}_{0,I_0\neq 0;1}[\psi]}+I_0[\psi]+P_{I_0,\beta}[\psi]\right)(u+1)^{-1-\epsilon}v^{-1},
\end{equation}
where $C=C(D,\Sigma,R,\epsilon)>0$ is a constant.

Furthermore, we can estimate in $\mathcal{R}\cap\{r\leq R\}$:
\begin{equation}\label{price_0_v2}
\left|\psi (\tau,\rho) -\frac{4I_0 [\psi ]}{(\tau+1)^2}\right| \leq C\left(\sqrt{\widetilde{E}^{\epsilon}_{0,I_0\neq 0;1}[\psi]}+I_0[\psi]+P_{I_0,\beta}[\psi]\right)(\tau+1)^{-2-\epsilon},
\end{equation}
where $C=C(D,\Sigma,R,\epsilon)>0$ is a constant.
\end{proposition}
\begin{proof}
Let $\alpha$ be a real number such that 
\begin{equation}
\frac{5}{7}<\alpha<1-6\epsilon,
\label{relae}
\end{equation}
where we take $0<\epsilon<\frac{1}{21}$, such that the above choice of $\alpha$ is well-defined. From Proposition \ref{cor1} it therefore follows that for all $(u,v)\in \mathcal{B}_{\alpha}$ we have that
\begin{equation*}
\begin{split}
\left|r\psi (u,v)-2I_0 [\psi] \left((u+1)^{-1}-v^{-1}\right)\right|\leq&\: C\left(\sqrt{E^{\epsilon}_{0,I_0\neq 0;0}[\psi]}+I_0[\psi]\right)(u+1)^{\frac{\alpha}{2}-\frac{3}{2}+2\epsilon}\\
&+C\cdot P_{I_{0},\beta}[\psi]\cdot (u+1)^{-1-\beta}.
\end{split}
\end{equation*}

By \eqref{uvr} we can estimate $C(v-u-1)\leq r\leq C(v-u-1)$ in $\mathcal{B}_{\alpha}$ for some constants $c(D,R),C(D,R)>0$. Furthermore, from the asymptotics of $D$ it follows that
\begin{equation*}
\begin{split}
\left|(v-u-1)-2r\right|\leq&\: C \log(r)\leq C\log (v-u-1)\\
\leq &\: C \log v.
\end{split}
\end{equation*}
in $\mathcal{B}_{\alpha}$. Note moreover that
\begin{equation*}
\begin{split}
2I_0 [\psi] \left((u+1)^{-1}-v^{-1}\right)=&\:2I_0[\psi]\frac{v-u-1}{(u+1)v}\\
&\:=2I_0[\psi]\cdot 2r\cdot \frac{1}{(u+1)v}+2I_0[\psi]((v-u-1)-2r)\cdot \frac{1}{(u+1)v}.
\end{split}
\end{equation*}
Therefore,
\begin{equation}
\label{eq:mainestpsiBalpha}
\begin{split}
&\Bigg|\psi (u,v)-4I_0 [\psi]\cdot \frac{1}{(u+1)v}\Bigg|\\ \ \ \ \ \ \leq &\: C\left(\sqrt{E^{\epsilon}_{0,I_0\neq 0;0}[\psi]}+I_0[\psi]\right)(v-u-1)^{-1}\Big[(u+1)^{\frac{\alpha}{2}-\frac{3}{2}+2\epsilon}\\ \ \ \ \ \ 
&+\log(v)(u+1)^{-1}v^{-1}\Big]+C(v-u-1)^{-1}\cdot P_{I_{0},\beta}[\psi]\cdot (u+1)^{-1-\beta}\\
\ \ \ \ \  \leq&\:C\left(\sqrt{E^{\epsilon}_{0,I_0\neq 0;0}[\psi]}+I_0[\psi]+P_{I_{0},\beta}[\psi] \right)(v-u-1)^{-1}(u+1)^{\frac{\alpha}{2}-\frac{3}{2}+2\epsilon},
\end{split}
\end{equation}
for all $(u,v)\in\mathcal{B}_{\alpha}$, where in the last inequality we used  that $\beta>\epsilon$ and we also used that the logarithmic term is bounded as follows: by the definition of the region $\mathcal{B}_{\alpha}$ we have $v\geq (u+1)$ and hence
\[\left|\frac{\log v}{(u+1)v•}\right|\leq \frac{C}{(u+1)v^{1-\epsilon}•}\leq \frac{C}{(u+1)(u+1)^{1-\epsilon}•}=\frac{C}{(u+1)^{2-\epsilon}}\leq \frac{C}{(u+1)^{\frac{3}{2}}}.\]

\textbf{Asymptotics in $\mathcal{B}_{\alpha+6\epsilon}$:}

To estimate the right-hand side of \eqref{eq:mainestpsiBalpha}, we will restrict further to the region
\begin{equation*}
\mathcal{B}_{\alpha+6\epsilon}\subset \mathcal{B}_{\alpha}.
\end{equation*}

We will moreover partition $\mathcal{B}_{\alpha+6\epsilon}$ as follows:
\begin{equation*}
\mathcal{B}_{\alpha+6\epsilon}=\left(\mathcal{B}_{\alpha+6\epsilon}\cap\left \{v-u-1\leq \frac{v}{2}\right\}\right)\cup \left(\mathcal{B}_{\alpha+6\epsilon}\cap\left \{v-u-1> \frac{v}{2}\right\}\right).
\end{equation*}

\bigskip

\textbf{Asymptotics in $\mathcal{B}_{\alpha+6\epsilon}\cap\left \{v-u-1> \frac{v}{2}\right\}$}: 

\medskip

We can estimate
\begin{equation*}
(v-u-1)^{-1}(u+1)^{\frac{\alpha}{2}-\frac{3}{2}+2\epsilon}\leq 2v^{-1}(u+1)^{\frac{\alpha}{2}-\frac{3}{2}+2\epsilon}<2v^{-1}(u+1)^{-1-\epsilon},
\end{equation*}
where we used that $\frac{\alpha}{2}-\frac{3}{2}+2\epsilon<-1-\epsilon$ which follows from \eqref{relae}.

Therefore,
\begin{equation*}
\begin{split}
\Bigg|\psi (u,v)-4I_0 [\psi]\cdot \frac{1}{(u+1)v}\Bigg|\leq&\:C\left(\sqrt{E^{\epsilon}_{0,I_0\neq 0;0}[\psi]}+I_0[\psi]+P_{I_{0},\beta}[\psi]\right)v^{-1}(u+1)^{-1-\epsilon}
\end{split}
\end{equation*}
for all $(u,v)\in \mathcal{B}_{\alpha+6\epsilon}\cap\left \{v-u-1> \frac{v}{2}\right\}$.\\
\\
\textbf{Asymptotics in $\mathcal{B}_{\alpha+6\epsilon}\cap\left \{v-u-1\leq \frac{v}{2}\right\}$}:

 We use that in $\mathcal{B}_{\alpha+6\epsilon}\cap\left \{v-u-1\leq \frac{v}{2}\right\}$ we can estimate $v-u\geq v^{\alpha+6\epsilon}$ and $v\sim u+1$, to obtain:
\begin{equation}
\label{eq:asymppsileft}
\begin{split}
(v-u-1)^{-1}(u+1)^{\frac{\alpha}{2}-\frac{3}{2}+2\epsilon} & \leq C v^{-\alpha-6\epsilon}(u+1)^{\frac{\alpha}{2}-\frac{3}{2}+2\epsilon}\\
& \leq C v^{-1}v^{1-\alpha-6\epsilon}(u+1)^{\frac{\alpha}{2}-\frac{3}{2}+2\epsilon}\\ 
&\leq  C v^{-1}(u+1)^{1-\alpha-6\epsilon+\frac{\alpha}{2}-\frac{3}{2}+2\epsilon}=v^{-1}\cdot (u+1)^{-\frac{1}{2}-\frac{\alpha}{2}-4\epsilon }.
\end{split}
\end{equation}
We have however 
\[ -\frac{1}{2}-\frac{\alpha}{2}-4\epsilon\leq -1-\frac{\epsilon}{2},   \]
if we require 
\begin{equation}
1-\alpha\leq 7\epsilon.
\label{relae1}
\end{equation}
In fact, if we had simply considered the slightly larger region $\mathcal{B}_{\alpha}$ then the exponent of $(u+1)$ on the right-hand side of \eqref{eq:asymppsileft} would not have been smaller than $-1$ and hence we would not be able to obtain precise first-order asymptotics for $\psi$.

Therefore, for $\alpha$ satisfying both \eqref{relae} and \eqref{relae1}, that is for $\alpha$ such that 
$1-7\epsilon\leq \alpha\leq 1-6\epsilon$ we obtain 

\begin{equation*}
\begin{split}
\Bigg|\psi (u,v)-4I_0 [\psi]\cdot \frac{1}{(1+u)v}\Bigg|\leq&\:C\left(\sqrt{E^{\epsilon}_{0,I_0\neq 0;0}[\psi]}+I_0[\psi]+P_{I_{0},\beta}[\psi]\right)v^{-1}(u+1)^{-1-\frac{\epsilon}{2}}
\end{split}
\end{equation*}
for all $(u,v)\in \mathcal{B}_{\alpha+6\epsilon}\cap\left \{v-u-1\leq \frac{v}{2}\right\}$.\\
\\
\textbf{Asymptotics in $(\mathcal{R}\cap\{r\geq R\})\setminus  \mathcal{B}_{\alpha+6\epsilon}$}:

To extend the above estimates into the region $\{r\geq R\}$ we will make use of the estimate
\begin{equation*}
r^{\frac{1}{2}}|\partial_v\psi|\leq C\sqrt{\widetilde{E}^{\epsilon'}_{0,I_0\neq0;1}} (1+u)^{-\frac{5}{2}+\epsilon'},
\end{equation*}
which follows from \eqref{eq:pointdecayNYpsi1}, for an  $\epsilon'>0$ that will be determined later as a function of $\epsilon$. We integrate along $\mathcal{N}_{\tau}$ to obtain:
\begin{equation*}
\begin{split}
|\psi(u,v)-\psi(u,v_{\gamma_{\alpha+6\epsilon}}(u))|=&\:\left|\int_v^{v_{\gamma_{\alpha+6\epsilon}}(u)}\partial_v\psi(u,v')\,dv'\right|\\
\leq &\: \int_v^{v_{\gamma_{\alpha+6\epsilon}}(u)}r^{-\frac{1}{2}}\cdot r^{\frac{1}{2}}|\partial_v\psi|(u,v')\,dv' \\
\leq &\:C\sqrt{\widetilde{E}^{\epsilon'}_{0,I_0\neq0;1}} (1+u)^{-\frac{5}{2}+\epsilon'} \int_v^{v_{\gamma_{\alpha+6\epsilon}}(u)}r^{-\frac{1}{2}}(v')\,dv' \\
\end{split}
\end{equation*}
Note that for $r\geq R$ we have $r^{*}(u,v)\sim r(u,v)\sim \frac{1}{2}(v-u)>0$ and hence
\begin{equation*}
\begin{split}
\int_v^{v_{\gamma_{\alpha+6\epsilon}}(u)}r^{-\frac{1}{2}}(v')\,dv' & \leq C \int_{v}^{v_{\gamma_{\alpha+6\epsilon}}(u)}(v'-u)^{-\frac{1}{2}}\, dv' \\
& =2C \cdot (v_{\gamma_{\alpha+6\epsilon}}(u)-u )^{1/2}-2C \cdot (v-u )^{1/2}  
\\ &\leq 2 C \cdot (v_{\gamma_{\alpha+6\epsilon}}(u)-u )^{1/2}\\
&= 2C \cdot (v_{\gamma_{\alpha+6\epsilon}}(u))^{\frac{\alpha+6\epsilon}{2}}  
\\ &\leq 2 C(1+u)^{\frac{\alpha}{2}+3\epsilon }.
\end{split}
\end{equation*}
Therefore, 
\begin{equation*}
\begin{split}
|\psi(u,v)-\psi(u,v_{\gamma_{\alpha+6\epsilon}}(u))| \leq 
C &\:\sqrt{\widetilde{E}^{\epsilon'}_{0,I_0\neq0;1}} (1+u)^{-\frac{5}{2}+\frac{1}{2}\alpha+3\epsilon+\epsilon'}\\
\leq C &   \:\sqrt{\widetilde{E}^{\epsilon'}_{0,I_0\neq0;1}} v_{\gamma_{\alpha+6\epsilon}}^{-1}(1+u)^{-\frac{3}{2}+\frac{1}{2}\alpha+3\epsilon+\epsilon'}\\
\leq C &   \:\sqrt{\widetilde{E}^{\epsilon'}_{0,I_0\neq0;1}} v^{-1}(1+u)^{-\frac{3}{2}+\frac{1}{2}\alpha+3\epsilon+\epsilon'},
\end{split}
\end{equation*}
where we used that $1+u\sim v_{\gamma_{\alpha+6\epsilon}}(u) $ and that in the region under consideration $v\leq v_{\gamma_{\alpha+6\epsilon}}(u)$.

Note that it is very crucial that we integrate only up to the curve $\gamma_{\alpha+6\epsilon}$, where $r\sim v^{\alpha+6\epsilon}$ with $\alpha+6\epsilon<1 $ and not, in particular, up to a curve where $r\sim v$. That also underlines the importance of the fact that we were able to derive lower bounds for the weighted derivative $v^{2}\partial_{v}\phi$ all the way up to the curve $\gamma_{\alpha}$.

 Let us now \emph{fix}
\begin{equation*}
\alpha=1-7\epsilon.
\end{equation*}
Then we obtain
\begin{equation*}
\begin{split}
|\psi(u,v)-\psi(u,v_{\gamma_{\alpha+6\epsilon}}(u))|\leq & \,  C \: \sqrt{\widetilde{E}^{\epsilon'}_{0,I_0\neq0;1}} v^{-1}(1+u)^{ -1-\frac{7}{2}\epsilon+3\epsilon+\epsilon' }\\
\leq \, C &\:\sqrt{\widetilde{E}^{\epsilon'}_{0,I_0\neq0;1}} v^{-1}(1+u)^{-1-\frac{1}{2}\epsilon+\epsilon'},
\end{split}
\end{equation*}
for all $(u,v)\in (\mathcal{R}\cap\{r\geq R\})\setminus  \mathcal{B}_{\alpha+6\epsilon}$. If we choose \[\epsilon'=\frac{1}{4}\epsilon\] then
\begin{equation*}
\begin{split}
|\psi(u,v)-\psi(u,v_{\gamma_{\alpha+6\epsilon}}(u))|\leq & C\:\sqrt{\widetilde{E}^{\epsilon/4}_{0,I_0\neq0;1}} v^{-1}(1+u)^{-1-\epsilon/4}.
\end{split}
\end{equation*}

Since $E^{\epsilon}_{0,I_0\neq 0;0}[\psi]\geq \widetilde{E}^{\epsilon/4}_{0,I_0\neq 0;0}[\psi]$, we can conclude that for $\beta>\epsilon/4$
\begin{equation*}
\begin{split}
\Bigg|\psi (u,v)-4I_0 [\psi]\cdot \frac{1}{(1+u)v}\Bigg|\leq&\:C\left(\sqrt{\widetilde{E}^{\epsilon/4}_{0,I_0\neq 0;0}[\psi]}+I_0[\psi]+P_{I_{0},\beta}[\psi]\right)v^{-1}(u+1)^{-1-\epsilon/4}
\end{split}
\end{equation*}
for all $(u,v)\in \{r\geq R\}$ and for $\epsilon\in(0,\frac{2}{49})$.\\
\\
\textbf{Asymptotics in $(\mathcal{R}\cap\{r\leq R\})$}:

Finally, we will extend the asymptotics of $\psi$ into the rest of the spacetime by integrating along $\Sigma_{\tau}\cap \{r\leq R\}$. By \eqref{price_0_vpsi} we have that
\begin{equation*}
\left|\psi|_{\rho=R}(\tau)-4I_0 [\psi]\cdot \frac{1}{(\tau+1)^2}\right|\leq C\left(\sqrt{\widetilde{E}^{\epsilon}_{0,I_0\neq 0;1}[\psi]}+I_0[\psi]+P_{I_{0},\beta}[\psi]\right)(\tau+1)^{-2-\epsilon},
\end{equation*}
since $\tau=u$ along $\{r=R\}$. Furthermore, 
\begin{equation*}
\rho^{\frac{1}{2}}|\partial_{\rho}\psi|\leq C\sqrt{\widetilde{E}^{\epsilon}_{0,I_0\neq0;1}} (1+\tau)^{-\frac{5}{2}+\epsilon},
\end{equation*}
which follows from \eqref{eq:pointdecayNYpsi1}. We now integrate in $\rho$ to obtain
\begin{equation*}
\begin{split}
|\psi(\tau,\rho)-\psi|_{\rho=R}(\tau))|=&\:\left|\int_{\rho}^{R}\partial_{\rho}\psi(\tau,\rho')\,d\rho'\right|\\
\leq &\: \int_{\rho}^{R}{\rho}^{-\frac{1}{2}}\cdot \rho^{\frac{1}{2}}|\partial_{\rho}\psi|(\tau,\rho')\,dv'\\
\leq &\:C\sqrt{\widetilde{E}^{\epsilon}_{0,I_0\neq0;1}} (1+\tau)^{-\frac{5}{2}+\epsilon}.
\end{split}
\end{equation*}
\end{proof}

After commuting $\psi$ with $T$ we get the following improved asymptotics.
\begin{proposition}\label{prop:tkpsi_lower}
Let $k\in \N_0$. There exists an $\epsilon>0$ suitably small such that under the assumptions:
\[\widetilde{E}^{\epsilon}_{0,I_0\neq0;k+1}[\psi]< \infty,\] and
\begin{equation*}
P_{I_0,\beta;k}[\psi]<\infty,
\end{equation*}
with $\beta>\epsilon$ and $P_{I_{0},\beta;k}[\psi]$ defined in \eqref{def:PI0k}, we have that for all $(u,v)\in \mathcal{R}\cap\{r\geq R\}$:\begin{equation}\label{price_0_v}
\begin{split}
\Bigg|T^k\psi& (u,v) -4(-1)^k k!\cdot \frac{I_0 [\psi ]}{(u+1)^{k+1}v}\left(1+\sum_{j=1}^k\left(\frac{u+1}{v}\right)^j\right)\Bigg|\\
 \leq&\: C\left(\sqrt{\widetilde{E}^{\epsilon}_{0,I_0\neq 0;k+1}[\psi]}+I_0[\psi]+P_{I_0,\beta;k}[\psi]\right)(u+1)^{-k-1-\epsilon}v^{-1},
 \end{split}
\end{equation}
where $C=C(D,\Sigma,R,k,\epsilon)>0$ is a constant.

Furthermore, we can estimate in $\mathcal{R}\cap\{r\leq R\}$:
\begin{equation}\label{price_0_v2}
\begin{split}
\Bigg|T^k\psi& (\tau,\rho) -4(-1)^k(k+1)!\cdot \frac{I_0 [\psi ]}{(\tau+1)^{k+2}}\Bigg| \\
\leq&\: C\left(\sqrt{\widetilde{E}^{\epsilon}_{0,I_0\neq 0;k+1}[\psi]}+I_0[\psi]+P_{I_0,\beta;k}[\psi]\right)(\tau+1)^{-k-2-\epsilon},
\end{split}
\end{equation}
where $C=C(D,\Sigma,R,k,\epsilon)>0$ is a constant.
\end{proposition}
\begin{proof}
Note that the leading-order term in the asymptotics for $rT^k\psi$, which is on the left-hand side of \eqref{asyTkphiprop} can be written as:
\begin{equation*}
\begin{split}
2(-1)^k\cdot k!\cdot I_0((u+1)^{-k-1}-v^{-k-1})=&\:2 (-1)^k\cdot k!\cdot I_0\\
&\cdot (u+1)^{-k-1}v^{-k-1}(v-u)\left(\sum_{j=0}^{k}(u+1)^jv^{k-j}\right),
\end{split}
\end{equation*}
where we made use of the identity
\begin{equation*}
v^{k+1}-(u+1)^{k+1}=(v-u)\sum_{j=0}^{k}(u+1)^jv^{k-j}.
\end{equation*}
We have that
\begin{equation*}
\begin{split}
(u+1)^{-k-1}&v^{-k-1}(v-u-1)\left(\sum_{j=0}^{k}(u+1)^jv^{k-j}\right)\\
=&\:(u+1)^{-k-1}v^{-1}(v-u-1)\left(1+\sum_{j=1}^{k}\left(\frac{u+1}{v}\right)^j\right).\\
\end{split}
\end{equation*}

Note that $|1-\frac{v-u-1)}{r}|\leq Cr^{-1}\log(v-u-1)\leq C(u+1)^{-\frac{3}{4}\alpha}$, so we have that
\begin{equation*}
\begin{split}
\Bigg|(u+1)^{-k-1}v^{-1}\frac{v-u-1)}{r}&\left(1+\sum_{j=1}^{k}\left(\frac{u+1}{v}\right)^j\right)\\
&-(u+1)^{-k-1}v^{-1}\left(1+\sum_{j=1}^{k}\left(\frac{u+1}{v}\right)^j\right)\Bigg|\\
\leq& C v^{-1}(u+1)^{-k-1-\epsilon}.
\end{split}
\end{equation*}

From the above, it therefore follows that
\begin{equation*}
\begin{split}
\Bigg|2(-1)^k\cdot r^{-1}\cdot k!\cdot &I_0[\psi]((u+1)^{-k-1}-v^{-k-1})\\
&- 2(-1)^k\cdot k!\cdot I_0[\psi] (u+1)^{-k-1}v^{-1}\left(1+\sum_{j=1}^{k}\left(\frac{u+1}{v}\right)^j\right) \Bigg| \\
\leq&\: C \cdot I_0[\psi] v^{-1}(u+1)^{-k-1-\epsilon}.
\end{split}
\end{equation*}

To obtain the asymptotics in $\{r\leq R\}$ we additionally note, using \eqref{uvr}, that in $\{r\leq R\}$:
\begin{equation*}
\left|1+\sum_{j=1}^{k}\left(\frac{u+1}{v}\right)^j-(k+1)\right|\leq C(\tau+1)^{-1}.
\end{equation*}
The rest of the proof proceeds analogously to the proof of Proposition \ref{prop:psi_lower} but now we use Proposition \ref{prop:tklower} and choose $\alpha_k>\frac{2k+5}{2k+7}$, and we use the estimate \eqref{eq:pointdecayNYpsi1} for $r^{\frac{1}{2}}\partial_{\rho}T^k\psi$.
\end{proof}

\section{Inverting the time-translation operator $T$}
\label{sec:consttinvphi}

\subsection{Construction of the time integral $\psi^{(1)}$}
\label{sectiontimeintegral}

In this section, we construct of a regular solution ${\psi}^{(1)}$ to \eqref{waveequation} such that \[T{\psi}^{(1)}=\psi,\] where $\psi$ is a given solution to the wave equation \eqref{waveequation} that decays suitably in $r$ along the initial hypersurface $\Sigma_{0}$. In view of \eqref{nptpsi}, a necessary condition for the existence of $\psi^{(1)}$ is that $I_0[\psi]=0$. We have the following proposition
\begin{proposition}
\label{prop:constructionTinvpsi}
Let $D=1-\frac{2M}{r}+o_1(r^{-1})$ and let $\psi$ be a spherically symmetric solution to \eqref{waveequation}, arising from smooth initial data on $\Sigma_0$, such that moreover
\begin{align}
\label{asm:initialpsik2}
\lim_{r\to\infty}\:\left. r^3\partial_r(r\psi)\right|_{\Sigma_{0}}<&\:\infty
\end{align}
with respect to the Bondi coordinates $(u,r,\theta,\varphi)$.

Then, there exists a unique smooth spherically symmetric solution $ \psi^{(1)}: \mathcal{R} \to \R$ to \eqref{waveequation} that decays along the initial hypersurface $\Sigma_{0}$
\begin{align*}
\lim_{r\to \infty}\left.\psi^{(1)}\right|_{\Sigma_{0}}=&\:0
\end{align*}
and that satisfies
\[T\psi^{(1)}=\psi\]
everywhere in $\mathcal{R}$.

In particular, $\psi^{(1)}$ satisfies the following integrability condition:
\begin{equation}
\begin{split}
\lim_{r\rightarrow \infty}r^{2}\left.\!\partial_{r}{\psi}^{(1)}\right|_{\Sigma_{0}}= & R(2-Dh_{\Sigma_{0}}(R))\phi(0,R)+2\int_{r\geq R}r L\phi\Big|_{\mathcal{N}_{0}}\,dv'\\
&-\int_{r_{\rm min}}^R 2(1-h_{\Sigma_{0}}D)r\partial_{\rho}\phi-(2-Dh_{\Sigma_{0}})rh_{\Sigma_{0}} T\phi-(r\cdot (Dh_{\Sigma_{0}})')\cdot\phi\Big|_{\Sigma_{0}}\,d\rho',
\end{split}
\label{integrabi}
\end{equation}
where $R$ and $h_{\Sigma_{0}}$ are as defined in Section \ref{sec:geomassm}.
\end{proposition}
\begin{definition}
The solution $\psi^{(1)}$ to the wave equation that arises from Proposition \ref{prop:constructionTinvpsi} is called the \textbf{time integral} of the solution $\psi$.
\label{deftimeintegral}
\end{definition}

\begin{proof}
Note first that condition \eqref{asm:initialpsik2} implies that $I_{0}[\psi]=0$. 

First, let us restrict to the region $\{r\geq R\}$. The wave equation \eqref{waveequation} under the spherical symmetry assumption reduces to 
\begin{equation}
\label{eq:waveequationv2} 
\begin{split}
\partial_r(Dr^2\partial_r\psi^{(1)})=&\:2r^2\partial_r\partial_u\psi^{(1)}+2r\partial_u \psi^{(1)}=2r\partial_{r}\big(r\partial_{u}\psi^{(1)}\big)=2r\partial_{r}\big(rT\psi^{(1)}\big) \\
=&\:2r \partial_r\phi.
\end{split}
\end{equation}
where we use the $(u,r,\theta,\varphi)$ coordinate system (recall that $\partial_u=T$ in this system). Here, $\phi=r\psi$. By integrating \eqref{eq:waveequationv2} along $\mathcal{N}_{0}$ we obtain
\begin{equation}
Dr^2\partial_r\psi^{(1)}(0,r)=\lim_{r\rightarrow \infty}r^{2}\left.\!\partial_{r}\psi^{(1)}\right|_{\Sigma_{0}}
-2\int_{r}^{\infty}r' \partial_{r'}\phi\,dr'.
\label{preliequ}
\end{equation}
Using \eqref{relationlr} this yields
\begin{equation}
\label{eq:tildepsiatR}
2R^2L\psi^{(1)}(0,R)=\lim_{r\rightarrow \infty}r^{2}\left.\!\partial_{r}\psi^{(1)}\right|_{\Sigma_{0}}-2\int_{\mathcal{N}_0}r L\phi\,dv'.
\end{equation}
We next consider the region $\{r\leq R\}$. In this region it will be more convenient to work with $(v,r,\theta,\varphi)$ coordinates and $(\tau,\rho,\theta,\varphi)$ coordinates rather than $(u,r,\theta,\varphi)$ coordinates. With respect to $(v,r,\theta,\varphi)$ coordinates, the wave equation \eqref{eq:waveequationv2} can be rewritten as:
\begin{equation}
\label{eq:waveequationv2b}
\begin{split}
\partial_r(Dr^2\partial_r\psi^{(1)})=&\:-2r^2\partial_r\partial_v\psi^{(1)}-2r\partial_v \psi^{(1)}.\\
\end{split}
\end{equation}
Recall that for tangential vector field $\partial_{\rho}$ to the hypersurface $\Sigma_{0}$ we have $\partial_{\rho}=\partial_r+h_{\Sigma_{0}}\partial_v$. We can therefore replace $r$-derivatives with $\rho$-derivatives to obtain:
\begin{equation*}
\begin{split}
\partial_{\rho}(Dr^2\partial_{\rho}\psi^{(1)})=&\:\partial_r(Dr^2\partial_r\psi^{(1)})+2h_{\Sigma_{0}}(r)r^2 D\partial_{\rho}\partial_v\psi^{(1)}-Dh_{\Sigma_{0}}^2(r)r^2 \partial_v^2\psi^{(1)}\\
&+(\partial_r(Dr^2h_{\Sigma_{0}}(r))\partial_v\psi^{(1)}.
\end{split}
\end{equation*}
Hence, by \eqref{eq:waveequationv2b} it follows that
\begin{equation}
\label{eq:waveequationv3}
\begin{split}
\partial_{\rho}(Dr^2\partial_{\rho}\psi^{(1)})=&\:(2h_{\Sigma_{0}}D-2)r^2 \partial_{\rho}\partial_v\psi^{(1)}+(2h_{\Sigma_{0}}-Dh_{\Sigma_{0}}^2)r^2\partial_v^2\psi^{(1)}\\
&+((Dr^2h_{\Sigma_{0}})'-2r)\partial_v\psi^{(1)}\\
=&\:-2(1-h_{\Sigma_{0}}D)r\partial_{\rho}\phi+(2-Dh_{\Sigma_{0}})rh_{\Sigma_{0}} T\phi+(r\cdot (Dh_{\Sigma_{0}})')\cdot\phi.
\end{split}
\end{equation}
Then we use \eqref{eq:waveequationv3} to obtain in $(\tau,\rho)$ coordinates on $\Sigma_{0}\cap \{r\leq R\}$:
\begin{equation}
\label{eq:drpsikplus0}
\begin{split}
Dr^2\partial_{\rho}\psi^{(1)}(0,\rho)=&\:DR^2\partial_{\rho}\psi^{(1)}(0,R)\\
&+\int_{\rho}^R2(1-h_{\Sigma_{0}}D)r'\partial_{\rho}\phi-(2-Dh_{\Sigma_{0}})rh_{\Sigma_{0}} T\phi-(r\cdot (Dh_{\Sigma_{0}})')\cdot\phi\Big|_{\tau=0}\,d\rho
\end{split}
\end{equation}
Recalling \eqref{relationlr0} we can write
\begin{equation*}
\begin{split}
\partial_{\rho}=&\:\partial_r+h_{\Sigma_{0}}\partial_v=-\frac{2}{D}\underline{L}+h_{\Sigma_{0}}T\\
=&\:\frac{2}{D}L -\left(\frac{2}{D}-h_{\Sigma_{0}}\right)T,
\end{split}
\end{equation*}
where we used that $T=L+\underline{L}$. Together with \eqref{eq:tildepsiatR} this implies that
\begin{equation*}
\begin{split}
DR^2\partial_{\rho}\psi^{(1)}(0,R)=&\:2R^2L\psi^{(1)}|_{u=0}(R)-R^2(2-h_{\Sigma_{0}}(R)D)T\psi^{(1)}(0,R)\\
=&\:\lim_{r\rightarrow \infty}r^{2}\left.\!\partial_{r}\psi^{(1)}\right|_{\Sigma_{0}}-R(2-Dh_{\Sigma_{0}}(R))\phi(0,R)-2\int_{\mathcal{N}_{0}}r L\phi\,dv'.
\end{split}
\end{equation*}
Using the above equation together with \eqref{eq:drpsikplus0} therefore gives:
\begin{equation}
\label{eq:drpsikplus1}
\begin{split}
Dr^2\partial_{\rho}\psi^{(1)}(0,\rho)=&\:\lim_{r\rightarrow \infty}r^{2}\left.\!\partial_{r}\psi^{(1)}\right|_{\Sigma_{0}}-R(2-Dh_{\Sigma_{0}}(R))\phi(0,R)-2\int_{\mathcal{N}_{0}}r L\phi\,dv'\\
&+\int_{\rho}^R2(1-hD)r\partial_{\rho}\phi-(2-Dh_{\Sigma_{0}})rh_{\Sigma_{0}} T\phi-(r\cdot (Dh_{\Sigma_{0}})')\cdot\phi\Big|_{\Sigma_{0}}\,d\rho'
\end{split}
\end{equation}
for all $\rho \leq R$. 

\textbf{The horizon case:} Suppose that $r_{\rm min}=r_+ >0$. Working in $(\tau, \rho)$ coordinates, by \eqref{eq:drpsikplus1} we have that
\begin{equation*}
\begin{split}
\partial_{\rho}\psi^{(1)}(0,r_+)=&\:\lim_{\rho\rightarrow r_{+}}D^{-1}(\rho)\rho^{-2}\cdot\Bigg[\lim_{r\rightarrow \infty}r^{2}\left.\!\partial_{r}\psi^{(1)}\right|_{\Sigma_{0}}-R(2-Dh_{\Sigma_{0}}(R))\phi(0,R)-2\int_{\mathcal{N}_{0}}r L\phi\,dv'\\
&+\int_{\rho}^R2(1-h_{\Sigma_{0}}D)r\partial_{\rho}\phi-(2-Dh_{\Sigma_{0}})rh_{\Sigma_{0}} T\phi-(r\cdot (Dh_{\Sigma_{0}})')\cdot\phi\Big|_{\Sigma_{0}}\,d\rho'\Bigg].
\end{split}
\end{equation*}

Note that we can write $D^{-1}=d^{-1}(r)\cdot(r-r_+)^{-1}$ with $d(r_+)>0$. Then $\partial_{\rho}\psi^{(1)}(0,r_+)$ is well-defined if and only if the following \textit{integrability condition} is satisfied:
\begin{equation*}
\begin{split}
\lim_{r\rightarrow \infty}r^{2}\left.\!\partial_{r}\psi^{(1)}\right|_{\Sigma_{0}}= & R(2-Dh_{\Sigma_{0}}(R))\phi(0,R)+2\int_{\mathcal{N}_{0}}r L\phi\,dv'\\
&-\int_{r_{+}}^R 2(1-h_{\Sigma_{0}}D)r\partial_{\rho}\phi-(2-Dh_{\Sigma_{0}})rh_{\Sigma_{0}} T\phi-(r\cdot (Dh_{\Sigma_{0}})')\cdot\phi\Big|_{\Sigma_{0}}\,d\rho'
\end{split}
\end{equation*}

\textbf{The case where $r_{\rm min}=0$}: We similarly obtain
\begin{equation*}
\begin{split}
\lim_{\rho\downarrow 0}\partial_{\rho}\psi^{(1)}(0,\rho)=&\:\lim_{\rho\downarrow 0}D^{-1}(\rho)\rho^{-2}\cdot\Bigg[\lim_{r\rightarrow \infty}r^{2}\left.\!\partial_{r}\psi^{(1)}\right|_{\Sigma_{0}}-R(2-Dh_{\Sigma_{0}}(R))\phi(0,R)-2\int_{\mathcal{N}_{0}}r L\phi\,dv'\\
&+\int_{\rho}^R2(1-h_{\Sigma_{0}}D)r\partial_{\rho}\phi-(2-Dh_{\Sigma_{0}})rh_{\Sigma_{0}} T\phi-(r\cdot (Dh_{\Sigma_{0}})')\cdot\phi\Big|_{\Sigma_{0}}\,d\rho'\Bigg],
\end{split}
\end{equation*}
so we require the following \textit{integrability condition}:
\begin{equation*}
\begin{split}
\lim_{r\rightarrow \infty}r^{2}\left.\!\partial_{r}\psi^{(1)}\right|_{\Sigma_{0}}= & R(2-Dh_{\Sigma_{0}}(R))\phi(0,R)+2\int_{\mathcal{N}_{0}}r L\phi\,dv'\\
&-\int_{0}^R 2(1-h_{\Sigma_{0}}D)r\partial_{\rho}\phi-(2-Dh_{\Sigma_{0}})rh_{\Sigma_{0}} T\phi-(r\cdot (Dh_{\Sigma_{0}})')\cdot\phi\Big|_{\Sigma_{0}}\,d\rho'.
\end{split}
\end{equation*}
Moreover, we can bound
\begin{equation*}
\left|2(1-h_{\Sigma_{0}}D) r\partial_{\rho} \phi-(2-Dh_{\Sigma_{0}})rh_{\Sigma_{0}} T\phi-(r\cdot (Dh_{\Sigma_{0}})')\cdot\phi\Big|_{\Sigma_{0}}\right|\leq C r,
\end{equation*}
so we can conclude that $\lim_{r\downarrow 0}\partial_{\rho}\psi^{(1)}(0,r)$ is indeed well-defined.

In both case, it easily follows that arbitrarily many $\rho$-derivatives of $\psi^{(1)}(0,r)$ are continuous for $\rho\in [r_{\rm min},\infty)$. Continuity of any (higher-order) derivative of $\psi^{(1)}$ involving $T$-derivatives follows directly from the requirement $T\psi^{(1)}=\psi$. Hence, $\psi^{(1)}$ is a smooth function on $\mathcal{R}$.

\end{proof}

\subsection{The time-inverted Newman--Penrose constants $I_{0}^{(k)}$}
\label{timeinvertednpsection}

The next proposition provides a formula for the Newman--Penrose constant associated to the time integral $\psi^{(1)}$ of a solution to the wave equation $\psi$.
\begin{proposition}
\label{prop:timeinvertedNP}
Let $\psi^{(1)}$ be the time integral associated to a solution $\psi$ to the wave equation. Then the Newman--Penrose constant of $\psi^{(1)}$ is given by 
\begin{equation}
\begin{split}
I_0[\psi^{(1)}]=& -\lim_{r'\to \infty}r'^3\partial_{r'}\phi(0,r')+M R(2-Dh_{\Sigma_{0}}(R))\phi(0,R)+2M\int_{r\geq R}r L\phi\Big|_{\mathcal{N}_{0}}\,dv'\\
&-M\int_{r_{\rm min}}^R 2(1-h_{\Sigma_{0}}D)r\partial_{\rho}\phi-(2-Dh_{\Sigma_{0}})rh_{\Sigma_{0}} T\phi-(r\cdot (Dh_{\Sigma_{0}})')\cdot\phi\Big|_{\Sigma_{0}}\,d\rho',
\end{split}
\label{nptimeintegralformula}
\end{equation}
where $\phi=r\psi$. In particular, if the initial data for $\psi$ is compactly supported in $\{r_{\rm min}\leq r<R\}$, we have that
\begin{equation}
\begin{split}
I_0[\psi^{(1)}]=&\:-M\int_{r_{\rm min}}^R 2(1-h_{\Sigma_{0}}D)r\partial_{\rho}\phi-(2-Dh_{\Sigma_{0}})rh_{\Sigma_{0}} T\phi-(r\cdot (Dh_{\Sigma_{0}})')\cdot\phi\Big|_{\Sigma_{0}}\,d\rho'.
\end{split}
\label{npformulatincompact}
\end{equation}
\label{nptimeintegralprop}
\end{proposition}
\begin{proof}

 By  \eqref{asm:initialpsik2}, we have that for $r\geq R$
\begin{equation*}
r^2\partial_r\phi(0,r)=\lim_{r'\to \infty}r'^3\partial_{r'}(r'\psi)(0,r')\cdot \frac{1}{r}+o(r^{-1}),
\end{equation*}
with respect to $(u,r,\theta,\varphi)$ coordinates. 
Let us also denote 
\[C_{0}= \lim_{r\rightarrow \infty}r^{2}\left.\!\partial_{r}\psi^{(1)}\right|_{\Sigma_{0}}. \]
By integrating \eqref{preliequ}, using that $D^{-1}=1+2Mr^{-1}+o_1(r^{-1})$ and assuming that $\lim_{r\to \infty}\psi^{(0)}\left.\right|_{\Sigma_{0}}=0$, we obtain for $r\geq R$ in $(u,r,\theta,\varphi)$ coordinates:
\begin{equation}
\label{eq:exprpsi1}
\begin{split}
\psi^{(1)}(0,r)=&\:-\int_r^{\infty} D^{-1}r'^{-2}\left[C_{0}-2\int_{r'}^{\infty}r''\partial_{r''}\phi\,dr''\right]\,dr'\\
=&-C_0\int_r^{\infty} r'^{-2}+2Mr'^{-3}+o_1(r'^{-3})\,dr'\\
&+2\int_{r}^{\infty} r'^{-3}\cdot \lim_{r''\to \infty}r''^3\partial_r\phi(0,r'')+o_1(r'^{-3})\,dr'\\
&=-C_0r^{-1}-MC_0r^{-2}+r^{-2}\cdot \lim_{r'\to \infty}r'^3\partial_r\phi(0,r')+o_2(r^{-2}).
\end{split}
\end{equation}
Now, it follows that
\begin{equation*}
r^2\partial_r(r\psi^{(1)})(0,r)=MC_0-\lim_{r'\to \infty}r'^3\partial_r\phi(0,r')+o_1(1).
\end{equation*}
In particular, using \eqref{integrabi} we obtain
\begin{equation*}
\begin{split}
I_0[\psi^{(1)}]=& -\lim_{r'\to \infty}r'^3\partial_{r'}\phi(0,r')+MC_0\\
=& -\lim_{r'\to \infty}r'^3\partial_{r'}\phi(0,r')+M R(2-Dh_{\Sigma_{0}}(R))\phi(0,R)+2M\int_{\mathcal{N}_{0}}r L\phi\,dv'\\
&-M\int_{r_{\rm min}}^R 2(1-h_{\Sigma_{0}}D)r\partial_{\rho}\phi-(2-Dh_{\Sigma_{0}})rh_{\Sigma_{0}} T\phi-(r\cdot (Dh_{\Sigma_{0}})')\cdot\phi\Big|_{\Sigma_{0}}\,d\rho'.
\end{split}
\end{equation*}
In particular, if the initial data for $\psi$ is compactly supported in $\{r_{\rm min}\leq r<R\}$, we have that
\begin{equation*}
\begin{split}
I_0[\psi^{(1)}]=&\:-M\int_{r_{\rm min}}^R 2(1-h_{\Sigma_{0}}D)r\partial_{\rho}\phi-(2-Dh_{\Sigma_{0}})rh_{\Sigma_{0}} T\phi-(r\cdot (Dh_{\Sigma_{0}})')\cdot\phi\Big|_{\Sigma_{0}}\,d\rho'.
\end{split}
\end{equation*}
\end{proof}

\begin{remark}
If we consider the time function $t=\frac{1}{2}(u+v)$ and the hypersurface $\Sigma_0=\{t=0\}$, the corresponding functions $h_{\Sigma_{0}}: (r_{+},\infty)\to \R$ and $h_{\Sigma_{0}}: [0,\infty)\to \R$ satisfy $h_{\Sigma_{0}}(r)=D^{-1}(r)$. Note that in the $r_{\rm min}=r_+$ case, $h_{\Sigma_{0}}$ cannot be smoothly extended to $[r_+,\infty)$. However, we can still apply the arguments in the proof of Proposition \ref{prop:constructionTinvpsi} in this case to construct $\psi^{(1)}$, if we make the additional assumption that \underline{the initial data are supported away from the bifurcation sphere}. With this choice of $h_{\Sigma_{0}}$, assuming moreover that $\psi$ is compactly supported on $\Sigma_0\cap \{r< R\}$, the expression for $I_0[\psi^{(1)}]$ becomes:
\begin{equation*}
I_0[\psi^{(1)}]=M\int_{\{t=0\}}D^{-1}T\psi\:r^2\, \sin\theta \,d\theta\, d\varphi \,d\rho.
\end{equation*} 

In particular, we can immediately see that $I_0[\psi^{(1)}]=0$ if the initial data for $\psi$ satisfies $T\psi|_{\Sigma}=0$, i.e. it is ``time-symmetric'' and $\psi$ vanishes at the bifurcation sphere. As a consequence, the decay in rate in time of $\psi$ will be one power higher than the generic (``time-asymmetric'') case, where $I_0[\psi^{(1)}]\neq 0$; see Proposition \ref{prop:tkpsi_lowert}. See also the discussion in Section \ref{introo1} and Remark \ref{rmk:initstatic}.
\end{remark}
We have the following important definition
\begin{definition}
Let $\psi$ be a solution to the wave equation \eqref{waveequation} satisfying the condition \eqref{asm:initialpsik2} of Proposition \ref{prop:constructionTinvpsi}. We define the \textbf{time-inverted Newman--Penrose constant} $I_{0}^{(1)}[\psi]$ of $\psi$  to be the Newman--Penrose constant $I_{0}[\psi^{(1)}]$ of the time integral $\psi^{(1)}$ of $\psi$. That is,
\[ I_{0}^{(1)}[\psi]:=I_{0}[\psi^{(1)}]. \]
\label{timeinvertedNP}
\end{definition}
Clearly, the time-inverted Newman--Penrose constant $I_{0}^{(1)}[\psi]$ is only defined for solutions $\psi$ with vanishing Newman--Penrose constants $I_{0}[\psi]=0$. In view of \eqref{nptimeintegralformula}, if $M>0$ then the time-inverted Newman--Penrose constant does not vanish for generic (compactly-supported) initial data for $\psi$. That is,
\[\text{If }M>0 \ \text{ and  } I_{0}[\psi]=0\  \Rightarrow \  I^{(1)}_{0}[\psi]\neq 0 \text{ generically}.\]
 Note, on the other hand, that on Minkowski space the time-inverted Newman--Penrose constant must necessarily vanish if the initial data of $\psi$ is compactly supported. The latter is of course related to Huygens' principle on Minkowski space. 

If, on the hand, for $M>0$, $\psi$ is such that its  time integral $\psi^{(1)}$ satisfies
\[\lim_{r\rightarrow \infty}r^{3}\partial_{r}(r\psi^{(1)})\left.\right|_{\Sigma_{0}} < \infty \]
then, using Proposition \ref{prop:constructionTinvpsi}, one can make sense of the time integral $\psi^{(2)}$ of $\psi^{(1)}$. Then,
\[T^{2}\psi^{(2)}=\psi. \]
Generically, we have that $I_{0}[\psi^{(2)}]\neq 0$. However, if $I_{0}[\psi^{(2)}]=0$ and in fact \[ \lim_{r\rightarrow \infty}r^{3}\partial_{r}(r\psi^{(2)})\left.\right|_{\Sigma_{0}} < \infty\]
then one can make sense of the time integral $\psi^{(3)}$ of $\psi^{(2)}$ 
which satisfies
\[T^{3}\psi^{(3)}=\psi. \]
Inductively, if the $k^{\text{th}}$ time integral $\psi^{(k)}$ of $\psi$ is defined and if 
\[\lim_{r\rightarrow \infty}r^{3}\partial_{r}(r\psi^{(k)})\left.\right|_{\Sigma_{0}} < \infty \]
then one can defined  the $(k+1)^{\text{th}}$ time integral $\psi^{(k+1)}$ of $\psi$ which satisfies 
\[T^{k+1}\psi^{(k+1)}=\psi. \]
We have the following definition
\begin{definition}
Let $\psi$ be a solution to the wave equation such that the $k^{\text{th}}$ time integral  $\psi^{(k)}$ is defined. We define the $k^{\text{th}}$-order time-inverted Newman--Penrose $I_{0}^{(k)}[\psi]$ of $\psi$ to be equal to the Newman--Penrose constant of $\psi^{(k)}$, that is 
\[I_{0}^{(k)}[\psi]:=I_{0}[\psi^{(k)}].  \]
\label{hightimeinvnp}
\end{definition}
Clearly, if $I_{0}^{(k+1)}[\psi]$ is well-defined, then we necessarily have 
\[I_{0}[\psi]=I_{0}^{(1)}[\psi]=\cdots I_{0}^{(k)}[\psi]=0. \]
The next proposition provides asymptotic conditions on the metric and the initial data for $\psi$ such that the relations
\[I_{0}[\psi]=I_{0}^{(1)}[\psi]=\cdots I_{0}^{(k)}[\psi]=0 \]
are sufficient to make sense of the $(k+1)^{\text{th}}$ time integral $\psi^{(k+1)}$ of $\psi$ and hence of $I_{0}^{(k+1)}[\psi]$.
\begin{proposition}\label{prop:tinverseconstr_k}
Let $k\in \N_0$, $N\in\mathbb{N}$ and $\beta>0$ and assume the following additional asymptotics for $D$:
\begin{equation}
\label{asm:additionalasmD}
D(r)=1+\sum_{m=0}^{N-1}d_mr^{-m-1}+O_{N+k}(r^{-N-\beta}),
\end{equation}
where $d_m\in \mathbb{R}$ for $m=0,1,\cdots, N-1$ and $d_1=-2M$. Let also $\psi$ be a spherically symmetric solution to \eqref{waveequation}, arising from smooth initial data on $\Sigma_0$ such that
\begin{align}
\label{asm:initialpsik1v2}
r^2\partial_r(r\psi)\left.\!\right|_{\Sigma_{0}}
=\sum_{m=1}^{N}p_{m}r^{-m}+O_k(r^{-N-\beta}),
\end{align}
where $p_m\in \R$  for $m=1,\cdots, N$. Then the $1^{\text{st}}$ time integral $\psi^{(1)}$ exists and satisfies
\begin{align}
\label{eq:initialpsikmin1}
r^2\partial_r(r\psi^{(1)})\left.\right|_{\Sigma_{0}}=&\:I_0[\psi^{(1)}]+O_{k+1}(r^{-N-\beta+1})\quad\textnormal{if}\quad N=1,\\
r^2\partial_r(r\psi^{(1)})\left.\right|_{\Sigma_{0}}=&\:I_0[\psi^{(1)}]+\sum_{m=1}^{N-1}p_m^{(1)}r^{-m}+O_{k+1}(r^{-N-\beta+1})\quad\textnormal{if}\quad N\geq 2,
\end{align}
where $p_m^{(1)}\in \R$ are determined by the initial data of $\psi$. 

Hence, if 
\[I_{0}[\psi]=I_0[\psi^{(1)}]=0\]
then one can make sense of the $2^{\text{nd}}$ time integral of $\psi$, and more generally, if
\[I_{0}[\psi]=I_{0}^{(1)}[\psi]=\cdots =I_{0}^{(j)}[\psi]=0 \]
for $0\leq j\leq N-1$, then one can make sense of the $(j+1)^{\text{th}}$ time integral $\psi^{(j+1)}$ of $\psi$ and hence of the $(j+1)^{\text{th}}$-order time-inverted Newman--Penrose constant $I_{0}^{(j+1)}[\psi]$, i.e.
\begin{align*}
r^2\partial_r(r\psi^{(j+1)})\left.\right|_{\Sigma_{0}}=&\:I_0[\psi^{(j+1)}]+O_{k+1+j}(r^{-N-\beta+1+j})\quad\textnormal{if} \quad j= N-1,\\
r^2\partial_r(r\psi^{(j+1)})\left.\right|_{\Sigma_{0}}=&\:I_0[\psi^{(j+1)}]+\sum_{m=1}^{N-1-j}p_m^{(j+1)}r^{-m}+O_{k+1+j}(r^{-N-\beta+1+j})\\\textnormal{if} \quad 0\leq j\leq N-2,
\end{align*}
where $p_m^{(j+1)}\in \R$ are determined by the initial data of $\psi$.
\end{proposition}

\begin{proof}
The proof proceeds identically to the proof of Proposition \ref{prop:constructionTinvpsi}, where we obtain the precise asymptotics \eqref{eq:initialpsikmin1} by using the more precise asymptotics of $D$ and the initial data of $\psi$ in the derivation of \eqref{eq:exprpsi1}.
\end{proof}

A direct corollary of the asymptotics of  the time integrals $\psi^{(k)}$ in the above proposition is finiteness of the $P_{I_0,\beta;k}$ norms of $\psi^{(k)}$.
\begin{corollary}
\label{cor:estPnorms}
Under the assumptions of Proposition \ref{prop:tinverseconstr_k} with $N=1$ and $k\in \N$, we have that
\begin{equation*}
P_{I_0,\beta;k+1}[\psi^{(1)}]<\infty.
\end{equation*}
Furthermore, for general $N\in \N$, if
\[I_{0}[\psi]=I_{0}^{(1)}[\psi]=\cdots =I_{0}^{(N-1)}[\psi]=0, \]
then
\begin{equation*}
P_{I_0,\beta;k+N}[\psi^{(N)}]<\infty.
\end{equation*}
\end{corollary}

\begin{remark}
Consider the \emph{conformal} coordinate $s=-1/r$ along $\mathcal{N}_0$. Note that $s=0$ at $\mathcal{I}^+$. We can interpret the assumptions \eqref{asm:additionalasmD} and \eqref{asm:initialpsik1v2} as resulting from regularity assumptions of $D$ and $\psi|_{\mathcal{N}_0}^{(N)}$ with respect to the coordinate $s$. Indeed, if $\psi|_{\Sigma_0}(s)\in C^{N+2}((-\epsilon,0])$ and $D(s)\in C_s^{N+1}((-\epsilon,0])$ for $\epsilon>0$ arbitrarily small, then we have by Taylor's theorem that
\begin{align*}
\partial_s\phi(s)=&\:\partial_s\phi(0)+\sum_{m=1}^{N} \frac{1}{m!}\partial_s^{m+1}\phi(0)\cdot s^m+o(s^{N+1}),\\
D(s)=&\:1+\sum_{m=1}^{N} \frac{1}{m!} \frac{d^mD}{ds^m}(0)\cdot s^m+o(s^{N+1}),
\end{align*}
with $\partial_s\phi(0)=I_0[\psi]$ and $D'(0)=2M$.

The above expansions imply \eqref{asm:additionalasmD} and \eqref{asm:initialpsik1v2} if we use that $s=-1/r$ and we take $\partial_s\phi(0)=0$.
\end{remark}

\subsection{Initial energy norms for the time integral $\psi^{(1)}$}
\label{estiamtefortimeinte}
In the following propositions we obtain estimates for the initial energy norms of time integrals $\psi^{(1)}$ in terms of initial energy norms of $\psi$. This allows us to apply the decay estimates of Section \ref{gdeforpsi1} to $\psi^{(1)}$. 

\begin{proposition}
\label{prop:initenergytimeint}
Let $\psi^{(1)}$ be the time integral associated to a solution $\psi$ to the wave equation, with $E^{\epsilon}_{0,I_0=0;0}[\psi]<\infty$ for some $\epsilon>0$ and $$\int_{\Sigma_0}\Big(J^N[N\psi]+J^N[NT\psi]+J^N[N^2\psi]\Big)\cdot n_0\,d\mu_0<\infty.$$
Then there exists a constant $C=C(D,R,\Sigma,\epsilon)>0$ such that
\begin{align}
\label{eq:energytimeint}
E^{\epsilon}_{0,I_0\neq 0;0}[\psi^{(1)}]\leq&\: C\cdot \left(E^{\epsilon}_{0,I_0= 0;0}[\psi]+\int_{\Sigma_0}J^N[N\psi]\cdot n_0\,d\mu_0\right),\\
\label{eq:NenergyNtimeint}
\int_{\Sigma_0}J^N[N\psi^{(1)}]\cdot n_0\,d\mu_0\leq&\: C \cdot \bigg(E^{\epsilon}_{0,I_0= 0;0}[\psi]+\int_{\Sigma_0}\Big(J^N[N\psi]+J^N[NT\psi]\\
&+J^N[N^2\psi]\Big)\cdot n_0\,d\mu_0\bigg).
\end{align}
\end{proposition}
\begin{proof}
See Appendix \ref{sec:initenergytimeint}.
\end{proof}

\begin{proposition}
\label{propositiontransition}
Let $k\in \N_0$. Let $\psi^{(1)}$ be the time integral associated to a solution $\psi$ to the wave equation with $\widetilde{E}^{\epsilon}_{0,I_0=0;k}[\psi]<\infty$ and $$\int_{\Sigma_0}J^N[N^2\psi]\cdot n_0\,d\mu_0<\infty.$$ Then there exists a constant $C=C(\Sigma,R,D,k,\epsilon)>0$, such that
\begin{align}
\label{eq:maininitenergyesttimeint}
{E}_{0,I_0\neq 0;k+1}^{\epsilon}[\psi^{(1)}]\leq&\: C\cdot \left({E}_{0,I_0=0;k}^{\epsilon}[\psi]+\int_{\Sigma_0}J^N[N\psi]\cdot n_0\,d\mu_0\right),\\
\label{eq:maininitenergyesttimeint2}
\widetilde{E}_{0,I_0\neq 0;k+1}^{\epsilon}[\psi^{(1)}]\leq&\: C\cdot \left(\widetilde{E}_{0,I_0=0;k}^{\epsilon}[\psi]+\int_{\Sigma_0}J^N[N^2\psi]\cdot n_0\,d\mu_0\right).
\end{align}
\end{proposition}
\begin{proof}
See Appendix \ref{sec:propositiontransition}.
\end{proof}

\section{Asymptotics II: The case $I_{0}=0$}
\label{sec:asympsi2}

\subsection{Asymptotics for the radiation field $r\psi$}
\label{sec:asmradfieldzeroNP}

We consider now spherically symmetric solutions $\psi$ for which $I_0 [\psi ] = 0$. We use Proposition \ref{prop:constructionTinvpsi} to construct a solution ${\psi}^{(1)}$ to \eqref{waveequation} such that $T{\psi}^{(1)}=\psi$ and apply Proposition \ref{asymprecTk} to $T{\psi}^{(1)}$ to arrive at the following result.

\begin{proposition}
\label{prop:asymphinp0}
Let $\alpha_1\in [\frac{7}{9},1)$ and $\epsilon\in (0,\frac{1}{6}(1-\alpha_1))$ and assume that
\begin{equation*}
\widetilde{E}^{\epsilon}_{0,I_0= 0;0}[\psi]<\infty,
\end{equation*}
and assume moreover that $$ v^3\partial_v(r\psi)(0,v)=\lim_{v'\to \infty}v'^3\partial_v(r\psi)(0,v')+O(v^{-\beta}). $$

Then we have that
\begin{equation}
\label{eq:asymphinp0}
\begin{split}
|r\psi (u,v)&+2I_0^{(1)} [\psi ] \left(u^{-2}- v^{-2}\right)|\\
\leq&\: C\left(\sqrt{E^{\epsilon}_{I_0\neq 0;1}[\psi^{(1)}]}+I_0^{(1)}[\psi]\right)(u+1)^{-2-\epsilon}\\
&+C\cdot P_{I_0,\beta;1}[\psi^{(1)}]\cdot (u+1)^{-2-\beta},
\end{split}
\end{equation}
for all $(u,v)\in \mathcal{B}_{\alpha_1}$ where 
\begin{enumerate}
\item $C=C(D,\Sigma,R,\alpha_1,\epsilon)>0$ is a constant,
\item  $\psi^{(1)}$ is the time integral of $\psi$, 
\item $I_{0}^{(1)}[\psi]$ is the time-inverted Newman--Penrose constant of $\psi$ and
\item  $P_{I_0,\beta;1}[\psi^{(1)}]$ is as defined in \eqref{def:PI0k}.

\end{enumerate}

In particular, along $\mathcal{I}^+$ we have that the following asymptotics for the radiation field $\phi=r\psi$:
\begin{equation*}
\begin{split}
|\phi (u,\infty)+ 2I_0^{(1)} [\psi ] u^{-2}|\leq&\: C\left(\sqrt{E^{\epsilon}_{I_0\neq 0;1}[\psi^{(1)}]}+I_0^{(1)}[\psi]\right)(u+1)^{-2-\epsilon}\\
&+C\cdot P_{I_0,\beta;1}[\psi^{(1)}]\cdot (u+1)^{-2-\beta}.
\end{split}
\end{equation*}
In fact, if we further impose $\frac{1-\alpha_1}{2}<\beta+2\epsilon$, the estimate \eqref{eq:asymphinp0} provides first-order asymptotics for $\phi$ in the region $\mathcal{B}_{\delta_1}$ for $\delta_1 $ such that $1>\delta_1> \frac{\alpha_1}{2}+\frac{1}{2}+2\epsilon>\alpha_1+2\epsilon$.
\end{proposition}
\begin{proof}
From Proposition \ref{prop:tinverseconstr_k} it follows that we can construct the time integral $\psi^{(1)}$ of $\psi$, satisfying $T\psi^{(1)}=\psi$. Furthermore, by Corollary \ref{cor:estPnorms} we have that $P_{I_0,\beta;1}[\psi^{(1)}]<\infty$ and by Proposition \ref{propositiontransition} we have that $E_{0,I_0\neq0;1}[\psi^{(1)}]<\infty$. We are therefore able to apply Proposition \ref{asymprecTk} with $k=1$ to $\psi^{(1)}$.
\end{proof}

\begin{proposition}\label{prop:tkpsi_lowert}
Let $n\in\mathbb{N}$ and assume the following additional asymptotics for $D$:
\begin{equation*}
D(r)=1-2Mr^{-1}+\sum_{m=0}^{n-1}d_mr^{-m-1}+O_{3+n}(r^{-n-\beta}),
\end{equation*}
where $d_m\in \mathbb{R}$ for $m=0,1,\cdots, n-1$. Consider the region $\mathcal{B}_{\alpha_n} =\{r\geq R\}\cap \{0\leq u\leq v-v^{\alpha_n}\}$, where $\alpha_n\in (\frac{2n+5}{2n+7},1)$. Let $\epsilon\in(0,\frac{1}{6}(1-\alpha_n))$ be arbitrarily small and assume that $$\widetilde{E}^{\epsilon}_{I_0= 0;n-1}[\psi]<\infty.$$ Assume moreover that
\begin{align*}
r^2\partial_r\phi\left.\!\right|_{\Sigma_{0}}
=\sum_{m=1}^{n}p_{m}r^{-m}+O(r^{-n-\beta}),
\end{align*}
where $p_m\in \R$  for $m=1,\cdots, n$.

Let $k\leq n$ and assume further more that
\begin{equation*}
I_0^{(0)}[\psi]=\ldots=I^{(k-1)}[\psi]=0.
\end{equation*}
Then,
\begin{equation}
\label{eq:asymTkphinp0}
\begin{split}
|\phi (u,v)&- (-1)^{k} k! \cdot  2I_0^{(k)} [\psi ] \left(u^{-k-1}- v^{-k-1}\right)|\\
\leq&\: C\left(\sqrt{E^{\epsilon}_{I_0\neq 0;k}[\psi^{(k)}]}+I_0^{(k)}[\psi]\right)(u+1)^{-1-k-\epsilon}\\
&+C\cdot P_{I_0,\beta;k}[\psi^{(k)}]\cdot (u+1)^{-2-k},
\end{split}
\end{equation}
for all $(u,v)\in \mathcal{B}_{\alpha_N}$ where 
\begin{enumerate}
\item $C=C(D,\Sigma,R,k,\alpha_n,\epsilon)>0$ is a constant,
\item $\psi^{(k)}$ is the $k^{\text{th}}$ time integral of $\psi$,
\item $I_{0}^{(k)}[\psi] $ is the $k^{\text{th}}$-order time-inverted Newman--Penrose constant of $\psi$ and
\item $P_{I_0,\beta;k}[\psi^{(k)}]$ is defined in \eqref{def:PI0k}.
\end{enumerate}

In particular, along $\mathcal{I}^+$ we have that
\begin{equation*}
\begin{split}
|\phi (u,\infty)-(-1)^{k} k! \cdot 2I_0^{(k)} [\psi ] u^{-k-1}|\leq&\: C\left(\sqrt{E^{\epsilon}_{I_0\neq 0;k}[\psi^{(k)}]}+I_0^{(k)}[\psi]\right)(u+1)^{-k-1-\epsilon}\\
&+C\cdot P_{I_0,\beta;k}[\psi^{(k)}]\cdot (u+1)^{-k-2}.
\end{split}
\end{equation*}
In fact, if we further impose $\frac{1-\alpha_n}{2}<\beta+2\epsilon$, the estimate \eqref{eq:asymTkphinp0} provides first-order asymptotics for $\phi$ in the region $\mathcal{B}_{\delta_n}$ for $\delta_n $ such that $1>\delta_n> \frac{\alpha_n}{2}+\frac{1}{2}+2\epsilon>\alpha_n+2\epsilon$.
\end{proposition}
\begin{proof}
From Proposition \ref{prop:tinverseconstr_k} it follows that we can construct the $k$-th time integral $\psi^{(k)}$ of $\psi$, satisfying $T^k\psi^{(k)}=\psi$. Furthermore, by Corollary \ref{cor:estPnorms} we have that $P_{I_0,1;k}[\psi^{(1)}]<\infty$ and by Proposition \ref{propositiontransition} we have that $E_{0,I_0\neq0;k}[\psi^{(1)}]<\infty$. We are therefore able to apply Proposition \ref{asymprecTk} to $\psi^{(k)}$.
\end{proof}

\subsection{Global asymptotics for the scalar field $\psi$}
\label{asymptoticspsi2}

The next proposition determines the late-time asymptotics of spherically symmetric solutions to the wave equation $\psi$.

\begin{proposition}
\label{prop:asympsinp0}
There exists an $\epsilon=\epsilon(k)>0$ suitably small, such that under the assumption $$\widetilde{E}^{\epsilon}_{0,I_0=0;1}[\psi]+\int_{\Sigma_0}J^N[N^2\psi]\cdot n_0\,d\mu_0< \infty,$$ and moreover
\begin{equation*}
v^3\partial_v(r\psi)(0,v)=\lim_{v'\to \infty}v'^3\partial_v(r\psi)(0,v')+O(v^{-\beta})<\infty,
\end{equation*}
we have that for all $(u,v)\in \mathcal{R}\cap\{r\geq R\}$ we can estimate
\begin{equation*}
\begin{split}
\Bigg|\psi& (u,v) +4\frac{I_0^{(1)} [\psi ]}{(u+1)^2v}\left(1+\frac{u}{v}\right)\Bigg|\\
 \leq&\: C\left(\sqrt{\widetilde{E}^{\epsilon}_{0,I_0\neq 0;2}[\psi^{(1)}]}+I_0^{(1)}[\psi]+P_{I_0,\beta;1}[\psi^{(1)}]\right)(u+1)^{-1-\epsilon}v^{-1},
 \end{split}
\end{equation*}
where 
\begin{enumerate}
\item $C=C(D,\Sigma,R,\epsilon)>$ is a constant,
\item $I_{0}^{(1)}[\psi] $ is the time-inverted Newman--Penrose constant of $\psi$,
\item $\psi^{(1)}$ is the time integral of $\psi$ and
\item $P_{I_0,\beta;1}[\psi^{(1)}]$ is as defined in \eqref{def:PI0k}.
\end{enumerate}

Furthermore, we can estimate in $\mathcal{R}\cap\{r\leq R\}$:
\begin{equation*}
\begin{split}
\Bigg|\psi& (\tau,\rho) +8\frac{I_0^{(1)} [\psi ]}{(\tau+1)^{3}}\Bigg| \\
\leq&\: C\left(\sqrt{\widetilde{E}^{\epsilon}_{0,I_0\neq 0;2}[\psi^{(1)}]}+I_0^{(1)}[\psi]+P_{I_0,\beta;1}[\psi^{(1)}]\right)(\tau+1)^{-3-\epsilon}.
\end{split}
\end{equation*}
\end{proposition}
\begin{proof}
The proposition follows immediately by using Proposition \ref{prop:tinverseconstr_k} with $k=1$, Corollary \ref{cor:estPnorms} with $N=1$ and Proposition \ref{propositiontransition} to show that $\psi^{(1)}$ satisfies the assumptions needed to subsequently apply Proposition \ref{prop:tkpsi_lower} with $k=1$.
\end{proof}

More generally, we can apply Proposition \ref{prop:tinverseconstr_k}, Corollary \ref{cor:estPnorms}, Proposition \ref{propositiontransition} and Proposition \ref{prop:tkpsi_lower} to determine the late-time behaviour of suitably decaying $\psi$ if $I_0^{(k)}[\psi]=0$ for all $0\leq k\leq n-1$, where $n\in \N$. This concludes the proof of Theorem \ref{prop:tkpsi_lower1}.

Finally, observe that Theorem \ref{thm:asmpsinpn0rfo00}  follows from Proposition  \ref{prop:tkpsi_lower} and \ref{prop:constructionTinvpsi}.

\appendix

\section{Energy norms}
\label{apx:energynorms}
Consider the following energy norms for the initial data of a function $\psi$ that satisfies \eqref{waveequation}.

Let $\psi$ be spherically symmetric and  $\epsilon>0$. Then the relevant energy norms are defined as follows:
\begin{equation*}
\begin{split}
E^{\epsilon}_{0,I_0\neq0;k}[\psi]=&\:\sum_{l\leq 3+3k}\int_{\Sigma_{0}}J^N[T^l\psi]\cdot n_{0}\,d\mu_{\Sigma_0}\\
&+\sum_{l\leq 2k}\int_{\mathcal{N}_{0}} r^{3-\epsilon}(\partial_rT^l\phi)^2\,dr+r^{2}(\partial_rT^{l+1}\phi)^2+r(\partial_rT^{2+l}\phi)^2\,dr\\
&+\sum_{\substack{m\leq k\\ l\leq 2k-2m+\min\{k,1\}}} \int_{\mathcal{N}_{0}}r^{2+2m-\epsilon}(\partial_r^{1+m}T^{l}\phi)^2\, dr\\
&+\int_{\mathcal{N}_{0}}r^{3+2k-\epsilon}(\partial_r^{1+k}\phi)^2\, dr
\end{split}
\end{equation*}
and
\begin{equation*}
\begin{split}
E^{\epsilon}_{0,I_0=0;k}[\psi]=&\:\sum_{l\leq 5+3k}\int_{\Sigma_{0}}J^N[T^l\psi]\cdot n_{0}\,d\mu_{\Sigma_0}\\
&+\sum_{l\leq 2k}\int_{\mathcal{N}_{0}} r^{5-\epsilon}(\partial_rT^l\phi)^2+r^{4-\epsilon}(\partial_rT^{1+l}\phi)^2+r^{3-\epsilon}(\partial_rT^{2+l}\phi)^2\\
&+r^{2}(\partial_rT^{3+l}\phi)^2+r(\partial_rT^{4+l}\phi)^2\, dr\\
&+\sum_{\substack{m\leq k\\ l\leq 2k-2m+\min\{k,1\}}} \int_{\mathcal{N}_{0}}r^{4+2m-\epsilon}(\partial_r^{1+m}T^{l}\phi)^2\, dr\\
&+\int_{\mathcal{N}_{0}}r^{5+2k-\epsilon}(\partial_r^{1+k}\phi)^2\, dr.
\end{split}
\end{equation*}

Let $\psi$ be supported on angular frequencies with $\ell=1$ and let $\epsilon>0$. Then the relevant energy norm is defined as follows:
\begin{equation*}
\begin{split}
E_{1;k}^{\epsilon}[\psi]\doteq &\:\sum_{\substack{l\leq 6+3k}}\int_{\Sigma_{0}}J^N[T^l\psi]\cdot n_{0}\;d\mu_{\Sigma_0}\\
&+\sum_{ l\leq 4+2k}\int_{\mathcal{N}_0}r^{2}(\partial_rT^l\phi)^2+r^{1}(\partial_rT^{1+l}\phi)^2\,d\omega dr\\
&+\sum_{l\leq 3,m\leq 2k}\int_{\mathcal{N}_{0}}r^{4-l-\epsilon}(\partial_rT^{l+m}\widetilde{\Phi})^2\,d\omega dr\\
&+\sum_{\substack{ m\leq \max\{k-1,0\}\\ l\leq k-2m+\min\{k,1\}}}\int_{\mathcal{N}_{0}}r^{4+2m-\epsilon}(\partial_r^{1+m}T^{l}\widetilde{\Phi})^2\,d\omega dr\\
&+\sum_{\substack{ m\leq k\\ l\leq 2k-2m+1}}\int_{\mathcal{N}_{0}} r^{3+2m-\epsilon}(\partial_r^{1+m}T^{l}\widetilde{\Phi})^2\,d\omega dr\\
&+\int_{\mathcal{N}_{0}} r^{4+2k-\epsilon}(\partial_r^{1+k}\widetilde{\Phi})^2\,d\omega dr.
\end{split}
\end{equation*}

Let $\psi$ be supported on angular frequencies with $\ell\geq 1$ and let $\epsilon>0$. Then the relevant energy norm is defined as follows:
\begin{equation*}
\begin{split}
E_{2;k}^{\epsilon}[\psi]\doteq &\:\sum_{\substack{|\alpha|\leq k\\ l+|\alpha|\leq 6+3k}}\int_{\Sigma_{0}}J^N[T^l\Omega^{\alpha}\psi]\cdot n_{0}\;d\mu_{\Sigma_0}\\
&+\sum_{ l\leq 4+2k}\int_{\mathcal{N}_0}r^{2}(\partial_rT^l\phi)^2+r^{1}(\partial_rT^{1+l}\phi)^2\,d\omega dr\\
&+\sum_{ l\leq 2k+2}\int_{\mathcal{N}_{0}} r^{2-\epsilon}(\partial_rT^{l}{\Phi})^2+r^{1-\epsilon}(\partial_rT^{l+1}{\Phi})^2\,d\omega dr\\
&+\sum_{\substack{|\alpha|\leq k\\l+|\alpha|\leq 2k}}\int_{\mathcal{N}_{0}} r^{2-\epsilon}(\partial_rT^{l}\Omega^{\alpha}{\Phi}_{(2)})^2+r^{1-\epsilon}(\partial_rT^{l+1}\Omega^{\alpha}{\Phi}_{(2)})^2\,d\omega dr\\
&+\sum_{\substack{|\alpha|\leq \max\{0,k-1\}\\ m\leq \max\{k-1,0\}\\ l+|\alpha|\leq k-2m+\min\{k,1\}}}\int_{\mathcal{N}_{0}}r^{2+2m-\epsilon}(\partial_r^{1+m}\Omega^{\alpha}T^{l}{\Phi}_{(2)})^2\,d\omega dr\\
&+\sum_{\substack{|\alpha|\leq \max\{0,k-1\}, m\leq k\\ l+|\alpha|\leq 2k-2m+1}}\int_{\mathcal{N}_{0}} r^{1+2m-\epsilon}(\partial_r^{1+m}\Omega^{\alpha}T^{l}{\Phi}_{(2)})^2\,d\omega dr\\
&+\int_{\mathcal{N}_{0}} r^{2+2k-\epsilon}(\partial_r^{1+k}{\Phi}_{(2)})^2\,d\omega dr.
\end{split}
\end{equation*}

\section{Proof of Proposition \ref{prop:initenergytimeint}}
\label{sec:initenergytimeint}
\begin{proof}
From the expressions in Appendix \ref{apx:energynorms} with $k=0$, it follows that the only terms in $E^{\epsilon}_{0,I_0=0;0}[\psi^{(1)}]$ that are not contained in $E^{\epsilon}_{0,I_0=0;0}[\psi]$ are
\begin{equation*}
\int_{\Sigma_0}J^N[\psi^{(1)}]\cdot n_0\,d\mu_0
\end{equation*}
and
\begin{equation*}
\int_{\mathcal{N}_0}r^{3-\epsilon}(\partial_r\phi^{(1)})^2\,dr.
\end{equation*}
We first split and estimate
\begin{equation*}
\begin{split}
\int_{\Sigma_0}J^N[\psi^{(1)}]\cdot n_0\,d\mu_0=&\:\int_{\Sigma_0\cap \{\rho\leq R\}}J^N[\psi^{(1)}]\cdot n_0\,d\mu_0+\int_{\mathcal{N}_0}J^N[\psi^{(1)}]\cdot L\,r^2\,dr\\
\sim&\: \int_{r_{\rm min}}^R(\partial_{\rho}\psi^{(1)})^2+(T\psi^{(1)})^2\,dr+\int_{R}^{\infty}r^2(\partial_r\psi^{(1)})^2\,dr,
\end{split}
\end{equation*}
where the second estimate can for example be found in the appendix of \cite{paper1}. We can estimate
\begin{equation*}
\int_{r_{\rm min}}^R(T\psi^{(1)})^2\,dr\lesssim \sup_{r_{\rm min}\leq r\leq R}|\phi|^2(0,r)\lesssim \int_{\Sigma_0}J^N[\psi]\cdot n_0\,d\mu_0,
\end{equation*}
where the second inequality follows from a straightforward application of the fundamental theorem of calculus together with Cauchy--Schwarz.

By Proposition \ref{prop:constructionTinvpsi} we have the following expression for $\rho\leq R$:
\begin{equation}
\label{eq:timeintboundedR}
Dr^2\cdot \partial_{\rho}\psi^{(1)}(0,\rho)= \int_{r_{\rm min}}^{\rho} -2(1-h_{\Sigma_0}D)r\partial_{\rho}\phi+(2-Dh_{\Sigma_{0}})rh_{\Sigma_{0}} T\phi+(r\cdot (Dh_{\Sigma_{0}})')\cdot\phi\Big|_{\Sigma_{0}}\,d\rho'.
\end{equation}

Hence,
\begin{equation*}
\begin{split}
Dr^2\cdot |\partial_{\rho}\psi^{(1)}|(0,\rho) \lesssim&\: \int_{r_{\rm min}}^{\rho}\rho'\,d\rho'\cdot \sup_{r_{\rm min}\leq \rho'\leq R} (|\partial_{\rho}\psi|+|T\psi|+|\psi|)(0,\rho')\\
\lesssim& Dr^2\cdot \sup_{r_{\rm min}\leq \rho'\leq R} (|\partial_{\rho}\psi|+|T\psi|+|\psi|)(0,\rho')\\
\lesssim& Dr^2\cdot\sqrt{\int_{\Sigma_0}(J^N[\psi]+J^N[N\psi]+J^N[T\psi])\cdot n_0\,d\mu_0},
\end{split}
\end{equation*}
where we applied once more the fundamental theorem of calculus and Cauchy--Schwarz to arrive at the final inequality.

Therefore,
\begin{equation*}
\int_{\Sigma_0\cap \{\rho\leq R\}}J^N[\psi^{(1)}]\cdot n_0\,d\mu_0\lesssim \int_{\Sigma_0}(J^N[\psi]+J^N[N\psi]+J^N[T\psi])\cdot n_0\,d\mu_0.
\end{equation*}

By \eqref{preliequ}, together with the integrability condition from Proposition \ref{prop:constructionTinvpsi} and the estimates in $\{r\leq R\}$ above, we moreover have that for $r\geq R$
\begin{equation*}
\begin{split}
r^2|\partial_r \psi^{(1)}|(0,r)\lesssim&\: \sqrt{\int_{\Sigma_0}(J^N[\psi]+J^N[N\psi]+J^N[T\psi])\cdot n_0\,d\mu_0}+ \int_{R}^{\infty}r'|\partial_{r'}\phi|\,dr'\\
\lesssim&\: \sqrt{\int_{\Sigma_0}(J^N[\psi]+J^N[N\psi]+J^N[T\psi])\cdot n_0\,d\mu_0}\\
&+\sqrt{\int_{R}^{\infty}r'^{-3+\epsilon}\,dr'}\cdot \sqrt{\int_R^{\infty} r^{5-\epsilon} (\partial_r\phi)^2\,dr'}\\
&\lesssim \sqrt{E_{0,I_0= 0;0}[\psi]+\int_{\Sigma_0}J^N[N\psi]\cdot n_0\,d\mu_0}
\end{split}
\end{equation*}
We can therefore conclude that
\begin{equation*}
\int_{\mathcal{N}_0}r^2(\partial_{\rho}\psi^{(1)})^2\,dr\lesssim E_{0,I_0= 0;0}[\psi]+\int_{\Sigma_0}J^N[N\psi]\cdot n_0\,d\mu_0.
\end{equation*}

In order to obtain \eqref{eq:NenergyNtimeint}, we need to additionally estimate
\begin{equation*}
\int_{r_{\rm min}}^R (\partial_{\rho}^2\psi^{(1)})^2\,d\rho.
\end{equation*}

By \eqref{eq:timeintboundedR}, we can estimate
\begin{equation*}
Dr^2\cdot \partial_{\rho}\psi^{(1)}(0,\rho)=F(\rho),
\end{equation*}
where
\begin{align*}
F(\rho)=&\:\int_{r_{\rm min}}^{\rho} -2(1-h_{\Sigma_0}D)r\partial_{\rho}\phi+(2-Dh_{\Sigma_{0}})rh_{\Sigma_{0}} T\phi+(r\cdot (Dh_{\Sigma_{0}})')\cdot\phi\Big|_{\Sigma_{0}}\,d\rho',\\
F'(\rho)=&\:\left(-2(1-h_{\Sigma_0}D)r\partial_{\rho}\phi(2-Dh_{\Sigma_{0}})rh_{\Sigma_{0}} T\phi+(r\cdot (Dh_{\Sigma_{0}})')\cdot\phi\right)(0,\rho),\\
F''(\rho)=&\: \partial_{\rho}\left(-2(1-h_{\Sigma_0}D)r\partial_{\rho}\phi+(2-Dh_{\Sigma_{0}})rh_{\Sigma_{0}} T\phi+(r\cdot (Dh_{\Sigma_{0}})')\cdot\phi\right)(0,\rho).
\end{align*}
Hence, we can apply Taylor's theorem to estimate
\begin{equation*}
F(\rho)=0+F'(r_{\rm min})\cdot (r-r_{\rm min})+\frac{1}{2}F''(r_{\rm min})\cdot (r-r_{\rm min})^2+ h(\rho)\cdot (r-r_{\rm min})^2,
\end{equation*}
where $h:[r_{\rm min},\infty)\to \R$ is a smooth function such that $\lim_{\rho\downarrow r_{\rm min}}g(\rho)=0$.

Consider first the case $r_{\rm min}=r_+$. Then we use that $D(r)=d(r)(r-r_+)$ to express for $\rho\leq R$:
\begin{equation*}
\partial_{\rho}\psi^{(1)}(0,\rho)=r^{-2}d^{-1}(r)F'(r_{+})+\frac{1}{2}r^{-2}d^{-1}(r)F''(r_{+})\cdot (r-r_{+})+ r^{-2}d^{-1}(r)h(\rho)\cdot (r-r_{+}).
\end{equation*}
It therefore follows immediately that for all $\rho\leq R$
\begin{equation*}
\begin{split}
|\partial_{\rho}^2\psi^{(1)}|(0,\rho)\lesssim&\: \sup_{r_{\rm min}\leq \rho'\leq R}(|F'|+|F''|+|h|)(0,\rho')\\
\lesssim&\: \sup_{r_{\rm min}\leq \rho'\leq R}(|F|+|F'|+|F''|)(0,\rho')
\end{split}
\end{equation*}

Now, consider the $r_{\rm min}=0$ case. Then, we also have that $F'(r_{\rm min})=0$, so we obtain
\begin{equation*}
\partial_{\rho}\psi^{(1)}(0,\rho)=\frac{1}{2}D^{-1}(r)F''(0)+D^{-1}(r)h(\rho).
\end{equation*}
So we have in this case also that
\begin{equation*}
\begin{split}
|\partial_{\rho}^2\psi^{(1)}|(0,\rho)\lesssim&\: \sup_{r_{\rm min}\leq \rho'\leq R}(|F|+|F'|+|F''|)(0,\rho').
\end{split}
\end{equation*}

Finally, we use that
\begin{equation*}
\sup_{r_{\rm min}\leq \rho'\leq R}(|F|+|F'|+|F''|)(0,\rho')\lesssim \sum_{0\leq j_1+j_2\leq 2}\int_{\Sigma_0}J^N[N^{j_1}T^{j_2}\psi]\cdot n_0\,d\mu_0
\end{equation*}
to obtain \eqref{eq:NenergyNtimeint}.

Let us now consider $\partial_r\phi^{(1)}$ in the region where $r\geq R$. We split
\begin{equation*}
\begin{split}
D\partial_r\phi^{(1)}=&\:D\partial_r(r\psi^{(1)})=Dr\partial_r\psi^{(1)}+D\psi^{(1)}\\
=&\:(Dr\partial_r\psi^{(1)}-C_0r^{-1})+(D\psi^{(1)}-C_0r^{-1}),
\end{split}
\end{equation*}
where $C_0=\lim_{r\to \infty}r^2\partial_r\psi^{(1)}(0,r)$.

By \eqref{preliequ} we can estimate
\begin{equation*}
\begin{split}
\left|Dr\partial_r\psi^{(1)}-C_0r^{-1}\right|(0,r)\lesssim&\: \frac{1}{r}\cdot \left|\int_{r}^{\infty}r' \partial_{r'}\phi\,dr'\right|\\
\leq&\: \frac{1}{r}\cdot \sqrt{\int_{r}^{\infty}r'^{-3+\epsilon}\,dr'}\cdot \sqrt{\int_{r}^{\infty}r'^{5-\epsilon} (\partial_{r'}\phi)^2\,dr'}\\
\leq &\: r^{-2+\frac{\epsilon}{2}}\int_{R}^{\infty}r'^{5-\epsilon} (\partial_{r'}\phi)^2\,dr.
\end{split}
\end{equation*}

By \eqref{eq:exprpsi1} and the estimates above, we can estimate
\begin{equation*}
\begin{split}
\left|\psi^{(1)}+C_0r^{-1}\right|(0,r)\lesssim&\: |C_0| \int_{r}^{\infty} O(r'^{-3})\,dr'+ \int_{r}^{\infty} r'^{-2} \left( \int_{r'}^{\infty}\left|r'' \partial_{r''}\phi\right|\,dr''\right)\,dr'\\
\lesssim&\: r^{-2+\frac{\epsilon}{2}}\cdot \sqrt{E_{0,I_0=0;0}[\psi]+\int_{\Sigma_0}J^N[N\psi]\cdot n_0\,d\mu_0}
\end{split}
\end{equation*}

Therefore, we can conclude that
\begin{equation*}
\int_{\mathcal{N}_0}r^{3-\epsilon}(\partial_r\phi^{(1)})^2\,dr\lesssim E_{0,I_0=0;0}[\psi]+\int_{\Sigma_0}J^N[N\psi]\cdot n_0\,d\mu_0.
\end{equation*}
\end{proof}

\section{Proof of Proposition \ref{propositiontransition}}
\label{sec:propositiontransition}
\begin{proof}
The only terms in $\widetilde{E}_{0,I_0\neq 0;k}^{\epsilon}[\psi^{(1)}]$ that are not contained in $\widetilde{E}_{0,I_0=0;k}^{\epsilon}[\psi]$ (i.e.\ do not only involve (higher-order) time-derivatives of $\psi^{(1)}$) are
\begin{equation}
\label{eq:Nenergies}
\int_{\Sigma_0}(J^N[\psi^{(1)}]+J^N[N\psi^{(1)})\cdot n_0\,d\mu_0
\end{equation}
and
\begin{equation}
\label{eq:rweights}
\sum_{j=0}^k\int_{\mathcal{N}_0}r^{3+2(k+1)-\epsilon-2j} (\partial_r^{k+2-j}\phi^{(1)})^2\,dr.
\end{equation}
The integral \eqref{eq:Nenergies} can be estimated directly by applying Proposition \ref{prop:initenergytimeint}.

Since $\psi^{(1)}$ is a solution to \eqref{waveequation}, we have that
\begin{equation*}
D\partial_r^2\phi^{(1)}=2\partial_r\phi+D'\partial_r \phi^{(1)}+D'r^{-1}\phi^{(1)}.
\end{equation*}
Therefore (using the appropriate asymptotics for $D$),
\begin{equation*}
\partial_r^{j+2}\phi^{(1)}=2\partial_r^{j+1}\phi+\sum_{m=0}^jO(r^{-1-m})\partial_r^{j+1-m} \phi^{(1)}+O(r^{-3-j})\phi^{(1)}.
\end{equation*}
for all $j\leq k$.

Hence,
\begin{equation*}
\begin{split}
\sum_{j=0}^k\int_{\mathcal{N}_0}r^{3+2(k+1)-\epsilon-2j}(\partial_r^{k+2-j}\phi^{(1)})^2\,dr'\lesssim&\: \int_{\mathcal{N}_0}r^{5+2k-\epsilon}(\partial_r^{k+1}\phi)^2+\sum_{j=0}^kr^{3+2k-\epsilon-2j}(\partial_r^{k+1-j}\phi^{(1)})^2\\
&+O(r^{-1-\epsilon})(\phi^{(1)})^2\,dr'\\
\lesssim &\: E^{\epsilon}_{0,I_0=0;k}[\psi]+E^{\epsilon}_{0,I_0\neq 0;k}[\psi^{(1)}]
\end{split}
\end{equation*}

By using Proposition \ref{prop:initenergytimeint} to estimate $E^{\epsilon}_{0,I_0\neq 0;0}[\psi^{(1)}]$, we can therefore conclude that \eqref{eq:maininitenergyesttimeint} and \eqref{eq:maininitenergyesttimeint2} hold for $k=0$. By the above equation, the general $k$ case then follows by induction.
\end{proof}

Department of Mathematics, University of California, Los Angeles, CA 90095, United States, yannis@math.ucla.edu

\bigskip

Princeton University, Department of Mathematics, Fine Hall, Washington Road, Princeton, NJ 08544, United States, aretakis@math.princeton.edu

\bigskip

Department of Mathematics, University of Toronto Scarborough 1265 Military Trail, Toronto, ON, M1C 1A4, Canada, aretakis@math.toronto.edu

\bigskip

Department of Mathematics, University of Toronto, 40 St George Street, Toronto, ON, Canada, aretakis@math.toronto.edu

\bigskip

Department of Mathematics, Imperial College London, SW7 2AZ, London, United Kingdom, dejan.gajic@imperial.ac.uk

\bigskip

Department of Applied Mathematics and Theoretical Physics, University of Cambridge, Wilberforce Road, Cambridge CB3 0WA, United Kingdom, dg405@cam.ac.uk

\end{document}